\documentclass[11pt,a4paper]{article}
\usepackage{color}
\usepackage{enumerate}
\usepackage{comment}
\usepackage{caption}
\usepackage[english]{babel}
\usepackage{exscale,times}

\usepackage{amssymb}
\usepackage{amsmath}
\usepackage{psfrag}
\usepackage{ifthen}
\usepackage{graphicx}
\usepackage{tikz}
\usepackage{pgfplots}
\usepackage{fullpage}

\newenvironment{proof}[1][]
{
   \noindent\textrm{\bf Proof%
   \ifthenelse{\equal{#1}{}}{:}{~#1:} }\rm
}
{
   \hfill$\square$
   \bigskip
}

\newtheorem{theorem}{Theorem}[section]
\newtheorem{proposition}[theorem]{Proposition}
\newtheorem{lemma}[theorem]{Lemma}
\newtheorem{corollary}[theorem]{Corollary}
\newtheorem{remark}[theorem]{Remark}

\newtheorem{assumption}{Assumption}
\newcommand{\eex}{\hbox{}\hfill\rule{0.8ex}{0.8ex}}
\newcommand{\eremk}{\eex}

\numberwithin{equation}{section}

\newcommand{\abs}[1]{\left\vert #1 \right\vert}
\newcommand{\norm}[1]{\left\| #1 \right\|}
\newcommand{\skp}[1]{\left< #1 \right>}
\renewcommand{\epsilon}{\varepsilon}

\title{}
\author{}

\begin{document}


\hskip 10 pt

\begin{center}
{\fontsize{14}{20}\bf Local convergence of the boundary element method
on polyhedral domains}
\end{center}

\begin{center}
\textbf{Markus Faustmann, Jens Markus Melenk}\\
\bigskip
{Institute for Analysis and Scientific Computing}\\
Vienna University of Technology\\
Wiedner Hauptstr. 8-10, 1040 Wien, Austria\\markus.faustmann@tuwien.ac.at, melenk@tuwien.ac.at\\
\bigskip
\end{center}
 
\begin{abstract}
The local behavior of the lowest order boundary element method on quasi-uniform meshes 
for Symm's integral equation and the 
stabilized hyper-singular integral equation on polygonal/polyhedral Lipschitz domains
is analyzed. 
We prove local {\sl a priori} estimates 
in $L^2$ for Symm's integral equation and in $H^1$ for the hypersingular equation. 
The local rate of convergence is limited by the local regularity of the sought solution
and the sum of the global regularity and additional regularity provided by the shift theorem for a dual problem. 
\end{abstract}

\section{Introduction}
The boundary element method (BEM) for the discretization of boundary integral equation is 
an established numerical method for solving partial differential 
equations on (un)bounded domains. 
As an energy projection method, the Galerkin BEM is, like the finite element method (FEM),
(quasi-)optimal in some global norm. 
However, often the quantity of interest is not the
error on the whole domain, but rather a local error on part of the computational domain.

For the FEM, the analysis of local errors goes back at least to \cite{nitsche-schatz74}; advanced versions 
can be found in \cite{wahlbin91,demlow-guzman-schatz11}. 
For Poisson's problem, the local error estimates typically have the form
\begin{equation}\label{eq:localFEM}
\norm{u-u_h}_{H^1(B_0)} \lesssim \inf_{\chi_h \in X_h}\norm{u - \chi_h}_{H^1(B_1)} + R^{-1}\norm{u-u_h}_{L^2(B_1)},
\end{equation} 
where $u$ is the exact solution, $u_h$ the finite element approximation from a space $X_h$ of piecewise polynomials, 
and $B_0 \Subset B_1$ are open subsets of $\Omega$ with $R:=\text{dist}(B_0,\partial B_1)$.
Thus, the local error in the energy norm is bounded by the local best approximation on a larger domain
and the error in the weaker $L^2$-norm. The local best approximation 
allows convergence rates up to the local regularity; the $L^2$-error is typically controlled 
with a duality argument and limited by the regularity of the 
dual problem as well as the global regularity of the solution. Therefore, if the solution is smoother locally, we 
can expect better rates of convergence for the local error. 

Significantly fewer works study the local behavior of the BEM.
The case of smooth two dimensional curves 
is treated in \cite{Saranen,Tran}, and in \cite{StephanTran}  
three dimensional screen problems are studied. \cite{RannacherWendlandI,RannacherWendlandII} provide 
estimates in the $L^{\infty}$-norm on smooth domains. However, for the case of piecewise smooth 
geometries such as polygonal and polyhedral domains, sharp local 
error estimates that exploit the maximal (local) regularity of the solution are not available. 
Moreover, the analyses of \cite{Saranen,Tran,StephanTran} are tailored to the energy norm and do not 
provide optimal local estimates in stronger norms. 

In this article, we obtain sharp local error estimates for lowest order discretizations on quasi-uniform meshes 
for Symm's integral equation in the $L^2$-norm and for the (stabilized) hyper-singular integral equation 
in the $H^1$-seminorm on polygonal/polyhedral domains. 
Structurally, the local estimates are similar to \eqref{eq:localFEM}: The local error is bounded by 
a local best approximation error and a global error in a weaker norm.
More precisely, our local convergence rates 
depend only on the local regularity and the sum of the global regularity and 
the additional regularity of the dual problem on polygonal/polyhedral domains. 
Numerical examples show the sharpness of our analysis. 
As discussed in Remark~\ref{rem:StephanTran} below, our results improve \cite{Saranen,Tran,StephanTran} 
as estimates in $L^2$ (for Symm's equation) and $H^1$
(for the hyper-singular equation) are obtained there from local energy norm estimates with the aid of 
inverse estimates, thereby leading to a loss of $h^{-1/2}$. In contrast, we avoid using an inverse inequality 
to go from the energy norm to a stronger norm. 

The paper is structured as follows. We start with some notations and then present the main results
for both Symm's integral equation and the hyper-singular integral equation in 
Section~\ref{sec:main-results}. In Section~\ref{sec:proofs} we are concerned with the proofs of these results. 
First, some technical preliminaries that exploit the additional regularity on piecewise smooth 
geometries to prove some improved {\sl a priori} estimates for solutions of Poisson's equation as well as 
for the boundary integral operators are presented. Then, we prove the main results, first for Symm's equation, then 
for the stabilized hyper-singular equation. In principle, the proofs take ideas from \cite{wahlbin91}, but 
due to the non-locality of the BEM solutions, important modifications are needed. However, similarly to 
\cite{wahlbin91} 
a key step is to apply interior regularity estimates, provided recently by \cite{FMPBEM,FMPHypSing}, and to
use some additional smoothness of localized boundary integral operators (commutators). 
Finally, Section~\ref{sec:numerics} provides numerical examples that underline the sharpness of our 
theoretical local {\sl a priori} estimates.

\subsection{Notation on norms}
For open sets $\omega \subset \mathbb{R}^d$, we define the integer order Sobolev spaces 
$H^k(\omega)$, $k \in \mathbb{N}_0$,
in the  standard  way \cite[p.~{73}ff]{mclean00}. 
The fractional Sobolev space $H^{k + s}(\omega)$, $k \in \mathbb{N}_0$, $s \in (0,1)$ are defined
by the Slobodeckii norm as described in \cite[p.~{73}ff]{mclean00}. The spaces $\widetilde H^s(\omega)$, $s \ge 0$,
consist of those function whose zero extension to $\mathbb{R}^d$ is in $H^s(\mathbb{R}^d)$. 
The spaces $H^{-s}(\omega)$, $s \ge 0$, 
are taken to be the dual space of $\widetilde H^s(\omega)$. We will make use of the fact that for bounded
Lipschitz domains $\omega$ 
\begin{equation}
\label{eq:H=widetildeH}
H^s(\omega) = \widetilde H^s(\omega)  \qquad \forall s \in [0,1/2). 
\end{equation}
For Lipschitz domains $\Omega \subset \mathbb{R}^d$ with boundary $\Gamma:= \partial\Omega$ we define Sobolev spaces 
$H^s(\Gamma)$ with $s \in [0,1]$ as described  in \cite[p.~{96}ff]{mclean00} using local charts. 
For $s > 1$, we \emph{define} the spaces 
$H^s(\Gamma)$ in a non-standard way: $H^s(\Gamma)$ consists of those functions that have a 
lifting to $H^{1/2+s}(\mathbb{R}^d)$, and we define the norm $\|\cdot\|_{H^s(\Gamma)}$ by  
\begin{equation}\label{eq:tracenorm}
 \norm{u}_{H^{s}(\Gamma)} := \inf_{\substack{v\in H^{1/2+s}(\mathbb{R}^d),\\ v|_{\Gamma}=u}} 
\norm{v}_{H^{1/2+s}(\mathbb{R}^d)}.
\end{equation}
Correspondingly, there is a lifting operator 
\begin{equation}\label{eq:deflifting}
\mathcal{L}:H^{1+s}(\Gamma)\rightarrow H^{3/2+s}(\mathbb{R}^d)
\end{equation}
with the lifting property $ (\mathcal{L} u)|_{\Gamma} = u$, 
which is by definition of the norm \eqref{eq:tracenorm} bounded. 
The spaces $H^{-s}(\Gamma)$, $s \ge 0$, are the duals of $H^s(\Gamma)$. Their norm is defined as 
\begin{equation*}
\|u\|_{H^{-s}(\Gamma)}:= \sup_{v \in H^s(\Gamma)} \frac{\langle u,v\rangle}{\|v\|_{H^s(\Gamma)}}. 
\end{equation*}
\begin{remark}[equivalent norm definitions]
\label{rem:alternative-norms}
\begin{enumerate}[(i)]
\item 
\label{item:rem:alternative-norms-i}
For $s > 1$ an equivalent definition of the norm 
$\|\cdot\|_{H^s(\Gamma)}$ in (\ref{eq:tracenorm}) 
would be to replace
$\|\cdot\|_{H^{s+1/2}({\mathbb R}^d)}$ with $\|\cdot\|_{H^{s+1/2}(\Omega)}$, i.e.,  
$$
 \norm{u}_{H^{s}(\Gamma)} := \inf_{\substack{v\in H^{1/2+s}(\Omega),\\ v|_{\Gamma}=u}} \|v\|_{H^{s+1/2}(\Omega)}. 
$$
This follows from the existence of the universal extension operator $E:L^2(\Omega) \rightarrow L^2(\mathbb{R}^d)$ 
described in 
\cite[Chap.~{VI.3}]{stein70}, which asserts that $E$ is also a bounded linear operator 
$H^{k}(\Omega) \rightarrow H^k(\mathbb{R}^d)$ for any $k \ge 0$. 
\item 
\label{item:rem:alternative-norms-ii}
The trace operator $\gamma_0: H^{s+1/2}(\mathbb{R}^d) \rightarrow H^{s}(\Gamma)$ is a continuous
operator for $0 < s < 1$ (cf. \cite[Thm.~{3.38}]{mclean00}, \cite[Thm.~{2.6.8}]{SauterSchwab}). 
\cite[Thm.~{2.6.11}]{SauterSchwab} (cf. also \cite[Thm.~{3.37}]{mclean00}) assert the existence of 
a continuous lifting $\mathcal{L}$ in the range $0 < s < 1$ as well so that (\ref{eq:tracenorm}) is an
equivalent norm for $0 < s < 1$ as well. 
\item 
\label{item:rem:alternative-norms-iii}
For polygonal (in 2D) and polyhedral (in 3D) Lipschitz domains the spaces $H^{s}(\Gamma)$ in the range 
$s \in (1,3/2)$ can be characterized alternatively as follows: Let $\Gamma_i$, $i=1,\ldots,N$, be the affine
pieces of $\Gamma$, which may be identified with an interval (for the 2D case) or a polygon (for the 3D case). 
Then 
\begin{equation}
\label{eq:piecewise-sobolev-spaces}
u \in H^s(\Gamma) \quad \Longleftrightarrow \quad 
u|_{\Gamma_i} \in H^s(\Gamma_i) \quad \forall i \in \{1,\ldots,N\} \quad \mbox{ and } 
u \in C^0(\Gamma). 
\end{equation}
The equivalence (\ref{eq:piecewise-sobolev-spaces}) gives rise to yet another norm equivalence
for the space $H^s(\Gamma)$, namely, $\|u\|_{H^s(\Gamma)} \sim \sum_{i=1}^N \|u\|_{H^s(\Gamma_i)}$. 

The condition $u \in C^0(\Gamma)$ is a \emph{compatibility} condition. More generally, for $s > 3/2$ similar, 
more complicated compatibility conditions can be formulated to describe the space $H^s(\Gamma)$ in terms 
of piecewise Sobolev spaces.
\eremk
\end{enumerate}
\end{remark}
Finally, we will need local norms on the boundary. For an open subset $\Gamma_0 \subset \Gamma$ and $s \ge 0$, we 
define local negative norms by
\begin{equation}
\norm{u}_{H^{-s}(\Gamma_0)} = \sup_{\substack{w\in H^{s}(\Gamma) \\ 
\operatorname*{supp} w \subset \overline{\Gamma_0}}}\frac{\skp{u,w}}{\norm{w}_{H^{s}(\Gamma)}}.
\end{equation} 
In the following, we write 
$\gamma_0^{\rm int}$ for the interior trace operator, i.e., the trace operator from the inside of the domain and
$\gamma_0^{\rm ext}$ for the exterior trace operator. 
For the jump of the trace of a function $u$ we use the notation 
$[\gamma_0 u] = \gamma_0^{\rm int} u-\gamma_0^{\rm ext} u$. 
In order to shorten notation, we write $\gamma_0$ for the trace, if the interior and exterior trace are 
equal, i.e., $[\gamma_0 u]=0$. 

We denote the interior and exterior conormal derivative by
$\gamma_1^{\rm int} u := \nabla u \cdot n_i$, $\gamma_1^{\rm ext} u := \nabla u \cdot n_e$, with the 
interior and exterior normal vectors $n_i,n_e$.
The jump of the normal derivative across the boundary is defined by 
$[\partial_n u]:= \gamma_1^{\rm int} u- \gamma_1^{\rm ext} u$, and we write $\partial_n u$ 
for the normal derivative if $[\partial_n u] = 0$.

\section{Main Results}
\label{sec:main-results}
We study bounded Lipschitz domains $\Omega \subset \mathbb{R}^d$, $d \ge 2$ with \emph{polygonal/polyhedral boundary} 
$\Gamma:=\partial\Omega$. 

\subsection{Symm's integral equation}
The elliptic shift theorem for the Dirichlet problem is valid in a range that 
is larger than for general Lipschitz domains. We characterize this extended range by a 
parameter $\alpha_D \in (0,1/2)$ that will pervade most of the estimates of the present work. It is defined
by the following assumption: 
\begin{assumption}\label{ass:shift}
$\Omega \subset \mathbb{R}^d$, $d\ge 2$ is a bounded Lipschitz domain whose boundary consists of finitely many 
affine pieces (i.e., $\Omega$ is the intersection of finitely many half-spaces). $R_\Omega > 0$ is such that 
the open ball $B_{R_{\Omega}}(0) \subset \mathbb{R}^d$ of radius $R_\Omega$ that is centered at the origin contains 
$\overline\Omega$. 
The parameter $\alpha_D \in (0,1/2)$ is such that for every $\varepsilon \in (0,\alpha_D]$ 
there is $C_\varepsilon > 0$ such that the {\sl a priori} bound 
\begin{equation}\label{eq:shift}
\norm{T f}_{H^{3/2+\epsilon}(B_{R_{\Omega}}(0)\backslash\Gamma)}\leq C_\varepsilon 
\norm{f}_{H^{-1/2+\epsilon}(B_{R_{\Omega}}(0)\backslash\Gamma)} 
\qquad \forall f \in H^{-1/2+\varepsilon}(B_{R_\Omega}(0)\setminus\Gamma)
\end{equation}
holds, where $u:= Tf \in H^1(B_{R_\Omega}(0)\setminus\Gamma)$ denotes the solution of
\begin{equation}
\label{eq:assumption-10}
-\Delta u = f \quad \mbox{ in $B_{R_\Omega}(0)\setminus\Gamma$}, 
\qquad \gamma_0 u = 0 \quad \mbox{ on $\Gamma \cup \partial B_{R_\Omega}(0)$}. 
\end{equation}
\end{assumption}

The norms $\norm{\cdot}_{H^{s}(B_{R_{\Omega}}(0)\backslash\Gamma)}$, $s>0$ are understood as the sum of the 
norm on $\Omega$ and $B_{R_{\Omega}}(0)\backslash \overline{\Omega}$, i.e.,
$$\norm{u}_{H^{s}(B_{R_{\Omega}}(0)\backslash\Gamma)}^2:=\norm{u}_{H^{s}(\Omega)}^2+
\norm{u}_{H^{s}(B_{R_{\Omega}}(0)\backslash\overline{\Omega})}^2.$$

\begin{remark} 
\label{rem:alpha_D}
The condition on the parameter $\alpha_D$ in Assumption~\ref{ass:shift} can be described in terms of 
two Dirichlet problems, one posed on $\Omega$ and one posed on $B_{R_\Omega}(0) \setminus \overline\Omega$. 
For each of these two domains, a shift theorem is valid, and $\alpha_D$ is determined by the more stringent
of the two conditions. It is worth stressing that the type of boundary condition on $\partial B_{R_{\Omega}}(0)$
is not essential in view of the smoothness of $\partial B_{R_{\Omega}}(0)$ and 
$\operatorname*{dist}(\Gamma,\partial B_{R_{\Omega}}(0))>0$.

In the case $d = 2$ the parameter $\alpha_D$ is determined by the extremal angles of the polygon $\Omega$. 
Specifically, let $0 < \omega_j < 2 \pi$, $j=1,\ldots,J$, be the interior angles of the polygon $\Omega$. Then, 
Assumption~\ref{ass:shift} is valid for any $\alpha_D >0 $ that satisfies 
$$
\frac{1}{2} < \frac{1}{2}+\alpha_D < \min_{j=1,\ldots,J} \min\left\{\frac{\pi}{\omega_j}, \frac{\pi}{2\pi-\omega_j}\right\} < 1. 
$$
(Note that $\omega_j \ne \pi$ for all $j$  so that the  right inequality is indeed strict.)
\eremk
\end{remark}

We consider Symm's integral equation in its weak form: Given $ f \in H^{1/2}(\Gamma)$  
find $\phi \in H^{-1/2}(\Gamma)$ such that 
\begin{equation}\label{eq:BIE}
\skp{V\phi,\psi}_{L^2(\Gamma)} = \skp{f,\psi}_{L^2(\Gamma)} \quad \forall \psi \in H^{-1/2}(\Gamma). 
\end{equation}
Here, the single-layer operator $V$ is given by
\begin{equation*}
V\phi(x) = \int_{\Gamma}G(x-y)\phi(y) ds_y, \quad x \in \Gamma,
\end{equation*}
where, 
with the surface measure $|S^{d-1}|$ of the Euclidean sphere in $\mathbb{R}^d$, we set 
\begin{align}
 G(x,y) = \begin{cases}
 -\frac{1}{|S^1|}\,\log|x-y|,\quad&\text{for }d=2,\\
 +\frac{1}{|S^{d-1}|}\,|x-y|^{-(d-2)},&\text{for }d\ge3.
 \end{cases}
\end{align}
The single layer operator $V$ is a bounded linear operator in $L(H^{-1/2+s}(\Gamma),H^{1/2+s}(\Gamma))$
for $\abs{s}\leq \frac{1}{2}$, \cite[Thm. 3.1.16]{SauterSchwab}. 
It is elliptic for $s=0$ with the usual proviso for $d=2$ that $\operatorname*{diam}(\Omega)<1$, 
which we can assume by scaling.

Let $\mathcal{T}_h = \{T_1,\dots,T_N\}$ be a quasiuniform, regular and $\gamma$-shape regular triangulation of the 
boundary $\Gamma$. 
By $S^{0,0}(\mathcal{T}_h):=\{u \in L^2(\Gamma):u|_{T_j} \text{is constant} \,\forall T_j \in \mathcal{T}_h\}$
we denote the space of piecewise constants on the mesh $\mathcal{T}_h$.
The Galerkin formulation of (\ref{eq:BIE}) reads: Find $\phi_h \in S^{0,0}(\mathcal{T}_h)$ such that
\begin{equation}\label{eq:BIEdiscrete}
\skp{V\phi_h,\psi_h}_{L^2(\Gamma)} = \skp{f,\psi_h}_{L^2(\Gamma)} \quad \forall \psi_h \in S^{0,0}(\mathcal{T}_h).
\end{equation}

The following theorem is one of the main results of this paper. It estimates the  
Galerkin error in the $L^2$-norm on a subdomain
by the local best approximation error in $L^2$ on a slightly larger subdomain and the global error in a weaker norm.

\begin{theorem}\label{th:localSLP}
Let Assumption~\ref{ass:shift} hold and let $\mathcal{T}_h$ be a quasiuniform, $\gamma$-shape regular triangulation. 
Let $\phi \in H^{-1/2}(\Gamma)$ and $\phi_h \in S^{0,0}(\mathcal{T}_h)$ satisfy the Galerkin orthogonality condition
\begin{equation}
\label{eq:Galerkin-V}
\langle V( \phi - \phi_h),\psi_h\rangle= 0 \qquad \forall \psi_h \in S^{0,0}(\mathcal{T}_h). 
\end{equation}
Let $\Gamma_0$, $\widehat{\Gamma}$ be open subsets of $\Gamma$ 
with $\Gamma_0\subset \widehat{\Gamma} \subsetneq \Gamma$
and $R:=\operatorname*{dist}(\Gamma_0,\partial\widehat{\Gamma}) > 0$. 
Let $h$ be sufficiently small such that at least
$C_{\alpha_D}\frac{h}{R}\leq \frac{1}{12}$ with a fixed constant $C_{\alpha_D}$ depending only on $\alpha_D$. 
Assume that $\phi\in L^{2}(\widehat{\Gamma})$. Then,
we have
\begin{eqnarray*}
\norm{\phi-\phi_h}_{L^{2}(\Gamma_0)} \leq C\left(\inf_{\chi_h\in S^{0,0}(\mathcal{T}_h)}\norm{\phi-\chi_h}_{L^{2}(\widehat{\Gamma})}  
+ \norm{\phi-\phi_h}_{H^{-1-\alpha_D}(\Gamma)}\right).
\end{eqnarray*} 
The constant $C>0$ depends only on $\Gamma,\Gamma_0,\widehat{\Gamma},d,R,$ and the $\gamma$-shape regularity of $\mathcal{T}_h$.
\end{theorem}

If we additionally assume higher local regularity as well as some (low) global regularity of the solution
$\phi$, this local estimate implies that the local error converges faster than the global error, 
which is stated in the following corollary.

\begin{corollary}\label{cor:localSLP}
Let the assumptions of Theorem~\ref{th:localSLP} be fulfilled. Let 
$\widetilde{\Gamma}\subset\Gamma$ be a subset with $\widehat{\Gamma}\subsetneq\widetilde{\Gamma}$ and 
$\operatorname*{dist}(\widehat{\Gamma},\partial\widetilde{\Gamma})\geq R>0$.
 Additionally, assume
$\phi \in H^{-1/2+\alpha}(\Gamma) \cap H^{\beta}(\widetilde{\Gamma})$ with $\alpha\geq 0$, 
$\beta \in [0,1]$. Then, we have 
\begin{equation*}
\norm{\phi-\phi_h}_{L^{2}(\Gamma_0)} \leq C h^{\min\{1/2+\alpha+\alpha_D,\beta\}}
\end{equation*}
with a constant $C>0$ depending only on $\Gamma,\Gamma_0,\widehat{\Gamma},\widetilde{\Gamma},d,R,\alpha,\beta$, 
and the $\gamma$-shape regularity of $\mathcal{T}_h$.
\end{corollary}

In the results of \cite{nitsche-schatz74,wahlbin91} singularities 
far from the domain of interest have a weaker influence on the local convergence for the FEM. Corollary~\ref{cor:localSLP} shows 
that this is similar in the BEM. 
Singularities either of the solution (represented by $\alpha$) or the 
geometry (represented by $\alpha_D$) are somewhat smoothed on distant parts of the boundary, 
but still persist even far away. 

\begin{remark}
\label{rem:StephanTran}
In comparison to \cite{StephanTran}, Corollary~\ref{cor:localSLP} gives a better result for the rate of 
convergence of the local error in the case where the convergence is limited by the global error in the weaker norm. 
More precisely, for 
the case $\phi \in H^{1/2}(\widetilde{\Gamma}) \cap L^2(\Gamma)$, \cite{StephanTran} obtains the local rate of $1/2$, 
which coincides with our local rate. However, if $\phi \in H^{1}(\widetilde{\Gamma})$, we obtain a rate of 
$1$ in the $L^2$-norm, whereas the rate in \cite{StephanTran} remains at $1/2$. 
\eremk
\end{remark}

\begin{remark} 
Even for smooth functions $f$, the solution $\phi$ of (\ref{eq:BIE}) is, in general, not better than 
$H^{\alpha}(\Gamma)$ with $\alpha  =  \frac{1}{2} + \alpha_D$. Recall from Remark~\ref{rem:alpha_D} that 
$\alpha_D$ is determined by the mapping properties for \emph{both} the interior and the exterior Dirichlet
problem.  A special situation therefore arises if Symm's integral equation is obtained from reformulating 
an interior (or exterior) Dirichlet problem. To be specific, consider again the case $d = 2$ of a polygon
$\Omega$ with interior angles $\omega_j$, $j=1,\ldots,J$. We rewrite the boundary value problem 
$-\Delta u =  0$ in $\Omega$ with $u|_\Gamma = g$ as the integral equation 
$$
V \phi = \left(\frac{1}{2} + K\right) g 
$$
for the unknown function $\phi = \partial_n u$ with the double layer operator $K$ 
defined by $K\phi(x):=\int_{\Gamma}\partial_n G(x,y)\phi(y)ds_y$. 
Then, $\phi \in H^\alpha(\Gamma)$ for any $\alpha$ with $\alpha < 1/2 + \min_j \frac{\pi}{\omega_j}$. 
\eremk
\end{remark}

\subsection{The hyper-singular integral equation}
For the Neumann problem, we assume an extended shift theorem as well.

\begin{assumption}\label{ass:shift2}
$\Omega \subset \mathbb{R}^d$, $d\ge 2$ is a bounded Lipschitz domain whose boundary consists of finitely many 
affine pieces (i.e., $\Omega$ is the intersection of finitely many half-spaces). $R_\Omega > 0$ is such that 
the open ball $B_{R_{\Omega}}(0) \subset \mathbb{R}^d$ of radius $R_\Omega$ that is centered at the origin contains 
$\overline\Omega$. 
The parameter $\alpha_N \in (0,\alpha_D]$, where $\alpha_D$ is the parameter from 
Assumption~\ref{ass:shift}, is such that for every $\varepsilon \in (0,\alpha_N]$ 
there is $C_\varepsilon > 0$ such that for all $f \in H^{-1/2+\varepsilon}(B_{R_\Omega}(0)\setminus\Gamma)$ 
and $g \in H^{\varepsilon}(\Gamma)$ with $\int_{\Omega}f +\int_{\Gamma} g = 0$
the {\sl a priori} bound 
\begin{equation}\label{eq:shiftNeumann}
\norm{T f}_{H^{3/2+\epsilon}(B_{R_{\Omega}}(0)\backslash\Gamma)}\leq C_\varepsilon \left(
\norm{f}_{H^{-1/2+\epsilon}(B_{R_{\Omega}}(0)\backslash\Gamma)} + 
\norm{g}_{H^{\epsilon}(\Gamma)} \right)
\end{equation}
holds, where $u:= Tf \in H^1(B_{R_\Omega}(0)\setminus\Gamma)$ denotes the solution of
\begin{align*}\label{eq:assumption-20}
 -\Delta u &= f \quad \mbox{in $\Omega$}, \qquad
&&\gamma_1^{\rm int} u = g \quad \mbox{on $\Gamma$}, &&\qquad \skp{u,1}_{L^2(\Omega)} = 0,\\
-\Delta u &= f \quad \mbox{in $B_{R_{\Omega}}(0)\backslash\overline{\Omega}$}, \qquad
&&\gamma_1^{\rm ext} u = g \quad \mbox{on $\Gamma$},
&&\qquad \gamma_0^{\rm int} u = 0 \quad \mbox{on $\partial B_{R_{\Omega}}(0)$}.
\end{align*}
\end{assumption}
The condition on the parameter $\alpha_N$ again can be described in terms of two problems, a pure 
Neumann problem posed in $\Omega$, for which we need a compatibility condition, and a mixed Dirichlet-Neumann 
problem posed on $B_{R_\Omega}(0) \backslash \overline{\Omega}$, which is uniquely solvable without 
the need to impose a solvability condition for $f,g$.

The parameter $\alpha_N$ again depends only on the geometry and the corners/edges that induce singularities.
In fact, on polygonal domains, i.e., $d=2$, $\alpha_D = \alpha_N$, see, e.g., \cite{Dauge88}.\\

Studying the inhomogeneous Neumann boundary value problem
$-\Delta u = 0$, $\partial_n u = g$,
leads to the boundary integral equation of finding $\varphi \in H^{1/2}(\Gamma)$ such that $W\varphi = f$
with $f \in H^{-1/2}(\Gamma)$ satisfying the compatibility condition $\skp{f,1}_{L^2(\Gamma)} = 0$ and
the hyper-singular integral operator $W \in L(H^{1/2}(\Gamma),H^{-1/2}(\Gamma))$
defined by
\begin{equation*}
W\varphi(x) = -\partial_{n_x}\int_{\Gamma}\partial_{n_y}G(x-y)\varphi(y) ds_y, \quad x \in \Gamma.
\end{equation*}
We additionally assume that $\Gamma$ is connected, so that 
the hyper-singular integral operator has a kernel of dimension one 
consisting of the constant functions. Therefore, the boundary integral equation is not uniquely solvable. 
Employing the constraint $\skp{\varphi,1} = 0$ leads to the stabilized variational formulation  
\begin{equation}\label{eq:BIEHS}
\skp{W\varphi,\psi}_{L^2(\Gamma)} + \skp{\varphi,1}_{L^2(\Gamma)}\skp{\psi,1}_{L^2(\Gamma)}  = 
\skp{f,\psi}_{L^2(\Gamma)} \quad \forall \psi \in H^{1/2}(\Gamma),
\end{equation}
which has a unique solution $\varphi \in H^{1/2}(\Gamma)$, see, e.g., \cite{Steinbach}.

For the Galerkin discretization we employ lowest order test and trial functions in 
$S^{1,1}(\mathcal{T}_h):=\{u \in H^1(\Gamma):u|_{T_j} \in \mathcal{P}_1 \,\forall T_j \in \mathcal{T}_h\}$, 
which leads to the discrete variational problem of finding $\psi_h \in S^{1,1}(\mathcal{T}_h)$ such that
\begin{equation}\label{eq:BIEdiscreteHS}
\skp{W\varphi_h,\psi_h}_{L^2(\Gamma)} + \skp{\varphi_h,1}_{L^2(\Gamma)}\skp{\psi_h,1}_{L^2(\Gamma)} 
= \skp{f,\psi_h}_{L^2(\Gamma)} \quad \forall \psi_h \in S^{1,1}(\mathcal{T}_h).
\end{equation}

The following theorem presents a result analogous to Theorem~\ref{th:localSLP} for the hyper-singular integral 
equation. The local error in the $H^1$-seminorm is estimated by the local best approximation error 
and the global error in a weak norm.

\begin{theorem}\label{th:localHypSing}
Let Assumption~\ref{ass:shift2} hold and let $\mathcal{T}_h$ be a quasiuniform, $\gamma$-shape regular triangulation. 
Let $\varphi \in H^{1/2}(\Gamma)$ and $\varphi_h \in S^{1,1}(\mathcal{T}_h)$ 
satisfy the Galerkin orthogonality condition
\begin{equation}
\label{eq:Galerkin-W}
\langle W( \varphi - \varphi_h),\psi_h\rangle 
+ \skp{\varphi-\varphi_h,1}\skp{\psi_h,1}= 0 \qquad \forall \psi_h \in S^{1,1}(\mathcal{T}_h). 
\end{equation}
Let $\Gamma_0$, $\widehat{\Gamma}$ be open subsets of $\Gamma$ 
with $\Gamma_0\subset \widehat{\Gamma} \subsetneq \Gamma$
and $R:=\operatorname*{dist}(\Gamma_0,\partial\widehat{\Gamma}) > 0$. 
Let $h$ be sufficiently small such that at least
$C_{\alpha_N}\frac{h}{R}\leq \frac{1}{12}$ with a fixed constant $C_{\alpha_N}$ depending only on $\alpha_N$. 
Assume that $\varphi\in H^{1}(\widehat{\Gamma})$. Then,
we have
\begin{eqnarray*}
\norm{\varphi-\varphi_h}_{H^{1}(\Gamma_0)} \leq C \left(  \inf_{\chi_h\in S^{1,1}(\mathcal{T}_h)}
\norm{\varphi - \chi_h}_{H^{1}(\widehat{\Gamma})}  
+ \norm{\varphi-\varphi_h}_{H^{-\alpha_N}(\Gamma)}\right).
\end{eqnarray*} 
The constant $C>0$ depends only on $\Gamma,\Gamma_0,\widehat{\Gamma},d,R,$ and the $\gamma$-shape regularity of $\mathcal{T}_h$.
\end{theorem}

Again, assuming additional regularity, the local estimate of Theorem~\ref{th:localHypSing} leads to
a faster rate of local convergence of the BEM for the stabilized hyper-singular integral equation. 

\begin{corollary}\label{cor:localHS}
Let the assumptions of Theorem~\ref{th:localHypSing} be fulfilled. Let 
$\widetilde{\Gamma}\subset\Gamma$ be a subset with $\widehat{\Gamma}\subsetneq\widetilde{\Gamma}$,
$\operatorname*{dist}(\widehat{\Gamma},\partial\widetilde{\Gamma})\geq R>0$.
 Additionally, assume
$\varphi \in H^{1/2+\alpha}(\Gamma) \cap H^{1+\beta}(\widetilde{\Gamma})$ with $\alpha\geq 0$, 
$\beta \in [0,1]$. Then, we have 
\begin{equation*}
\norm{\varphi-\varphi_h}_{H^{1}(\Gamma_0)} \leq C h^{\min\{1/2+\alpha+\alpha_N,\beta\}}
\end{equation*}
with a constant $C>0$ depending only on $\Gamma,\Gamma_0,\widehat{\Gamma},\widetilde{\Gamma},d,R,\alpha,\beta$, 
and the $\gamma$-shape regularity of $\mathcal{T}_h$.
\end{corollary}

\section{Proof of main results}
\label{sec:proofs}
This section is dedicated to the proofs of Theorem~\ref{th:localSLP},
Corollary~\ref{cor:localSLP} for Symm's integral equation 
and Theorem~\ref{th:localHypSing} and Corollary~\ref{cor:localHS} for the hyper-singular integral equation.

We start with some technical results that are direct consequences of the assumed shift theorems from 
Assumption~\ref{ass:shift} for the Dirichlet problem and Assumption~\ref{ass:shift2} for the Neumann problem.

\subsection{Technical preliminaries}
The shift theorem of Assumption~\ref{ass:shift} implies the following shift theorem for Dirichlet problems: 

\begin{lemma}\label{lem:shiftapriori}
Let the shift theorem from Assumption~\ref{ass:shift} hold and let $u$ be the solution 
of the inhomogeneous Dirichlet problem $-\Delta u = 0$ in $B_{R_{\Omega}}(0)\backslash\Gamma$, 
$\gamma_0 u = g$ on $\Gamma \cup \partial B_{R_{\Omega}}(0)$
for some $g\in H^{1/2}(\Gamma\cup \partial B_{R_{\Omega}}(0))$.

\begin{enumerate}[(i)]
 \item 
\label{item:lem:shiftapriori-i}
There is a constant $C>0$ depending only on $\Omega$ and $\alpha_D$ such that 
\begin{equation}
\label{eq:lem:shiftapriori-10}
\norm{u}_{H^{1/2-\alpha_D}(B_{R_{\Omega}}(0)\backslash\Gamma)}\leq 
C \norm{g}_{H^{-\alpha_D}(\Gamma \cup\partial B_{R_{\Omega}}(0))}.  
\end{equation}

\item 
\label{item:lem:shiftapriori-ii}
Let $\epsilon \in (0,\alpha_D]$ and $B\subset B' \subset B_{R_{\Omega}}(0)$ be
nested subdomains with $\operatorname*{dist}(B,\partial B') > 0$.
Let $\eta \in C^{\infty}(\mathbb{R}^d)$ be a cut-off function satisfying $\eta \equiv 1$ on $B$,
$\operatorname*{supp} \eta \subset \overline{B'}$, 
and $\norm{\eta}_{C^k(B')}\lesssim \operatorname*{dist}(B,\partial B')^{-k}$ 
for $k \in \{0,1,2\}$. 
Assume $\eta g\in H^{1+\epsilon}(\Gamma)$. Then 
 \begin{equation}
\label{eq:lem:shiftapriori-20}
 \norm{u}_{H^{3/2+\epsilon}(B\backslash\Gamma)} \leq C\left(\norm{u}_{H^1(B'\backslash\Gamma)}
+\norm{\eta g}_{H^{1+\epsilon}(\Gamma)}\right).
\end{equation}
Here, the constant $C>0$ additionally depends on $\operatorname*{dist}(B,\partial B')$.
\end{enumerate}

\end{lemma}
\begin{proof}
\emph{Proof of (\ref{item:lem:shiftapriori-i}):}
Let $v$ solve $-\Delta v = w$ in $B_{R_{\Omega}}(0)\backslash\Gamma$, 
$\gamma_0 v = 0$ on $\Gamma\cup \partial B_{R_{\Omega}}(0)$ for 
$w \in H^{-1/2+\alpha_D}(B_{R_{\Omega}}(0)\backslash\Gamma)$. Then, in view of (\ref{eq:H=widetildeH}), 
\begin{eqnarray*}
\norm{u}_{H^{1/2-\alpha_D}(B_{R_{\Omega}}(0)\backslash\Gamma)} &=& 
\sup_{w \in H^{-1/2+\alpha_D}(B_{R_{\Omega}}(0)\backslash\Gamma)} 
\frac{\skp{u,w}_{L^2(B_{R_{\Omega}}(0)\backslash\Gamma)}}{\norm{w}_{H^{-1/2+\alpha_D}(B_{R_{\Omega}}(0)\backslash\Gamma)}} \\
&=&
\sup_{w \in H^{-1/2+\alpha_D}(B_{R_{\Omega}}(0)\backslash\Gamma)} 
\frac{-\skp{u,\Delta v}_{L^2(B_{R_{\Omega}}(0)\backslash\Gamma)}}{\norm{w}_{H^{-1/2+\alpha_D}(B_{R_{\Omega}}(0)\backslash\Gamma)}}.
\end{eqnarray*}
Integration by parts on $\Omega$ and $B_{R_{\Omega}}(0)\backslash\overline{\Omega}$ leads to 
\begin{eqnarray*}
\skp{u,\Delta v}_{L^2(B_{R_{\Omega}}(0)\backslash\Gamma)} &=& \skp{\Delta u,v}_{L^2(B_{R_{\Omega}}(0)\backslash\Gamma)} +
 \skp{\gamma_0 u,[\partial_n v]}_{L^2(\Gamma)} +
 \skp{\gamma_0 u,\partial_n v}_{L^2(\partial B_{R_{\Omega}}(0))} \\ &=&
 \skp{g,[\partial_n v]}_{L^2(\Gamma)} +
 \skp{g,\partial_n v}_{L^2(\partial B_{R_{\Omega}}(0))}.
\end{eqnarray*}
We split the polygonal boundary 
$\Gamma = \bigcup_{\ell=1}^m \overline{\Gamma_{\ell}}$ into its (smooth) faces 
$\Gamma_{\ell}$ and prolong each face $\Gamma_{\ell}$ to the hyperplane $\Gamma_{\ell}^{\infty}$, which 
decomposes $\mathbb{R}^d$ into two half spaces $\Omega_{\ell}^{\pm}$.
Let $\chi_{\ell} \in L^2(\Gamma)$ be the characteristic function 
for $\Gamma_{\ell}$.
Since  the normal vector on a face does not change, we may use the trace estimate (note: $0 <\alpha_D < 1/2$) facewise, 
to estimate
\begin{eqnarray}\label{eq:facewisenormalder}
\norm{[\partial_n v]}_{H^{\alpha_D}(\Gamma)} &\lesssim& 
\sum_{\ell=1}^m\norm{\chi_{\ell}[\nabla v \cdot n]}_{H^{\alpha_D}(\Gamma_{\ell})} \lesssim
\sum_{\ell=1}^m \norm{\nabla v}_{H^{1/2+\alpha_D}(\Omega_{\ell}^{\pm}\cap B_{R_{\Omega}}(0))}\nonumber \\ &\lesssim&
 \norm{v}_{H^{3/2+\alpha_D}(B_{R_{\Omega}}(0)\backslash\Gamma)}.
\end{eqnarray}
As the boundary $\partial B_{R_{\Omega}}(0)$ is smooth, standard elliptic regularity yields 
$\norm{\partial_n v}_{H^{\alpha_D}(\partial B_{R_{\Omega}}(0))} \lesssim 
 \norm{v}_{H^{3/2+\alpha_D}(B_{R_{\Omega}}(0)\backslash\Gamma)}$.
This leads to
\begin{align}
&\norm{u}_{H^{1/2-\alpha_D}(B_{R_{\Omega}}(0)\backslash\Gamma)} \lesssim
\sup_{w \in H^{-1/2+\alpha_D}(B_{R_{\Omega}}(0)\backslash\Gamma)} 
\frac{\abs{\skp{g,[\partial_n v]}_{L^2(\Gamma)}+\skp{g,\partial_n v}_{L^2(\partial B_{R_{\Omega}}(0))}}}
{\norm{w}_{H^{-1/2+\alpha_D}(B_{R_{\Omega}}(0)\backslash\Gamma)}}\nonumber \\
&\qquad\quad\lesssim \sup_{w\in H^{-1/2+\alpha_D}(B_{R_{\Omega}}(0)\backslash\Gamma)} 
\frac{\norm{g}_{H^{-\alpha_D}(\Gamma\cup\partial B_{R_{\Omega}}(0))}\left(\norm{[\partial_n v]}_{H^{\alpha_D}(\Gamma)}+
\norm{\partial_n v}_{H^{\alpha_D}(\partial B_{R_{\Omega}}(0))}\right)}
{\norm{w}_{H^{-1/2+\alpha_D}(B_{R_{\Omega}}(0)\backslash\Gamma)}}\nonumber
\\ &\qquad\quad\lesssim 
\norm{g}_{H^{-\alpha_D}(\Gamma \cup \partial B_{R_{\Omega}}(0))}
\sup_{w \in H^{-1/2+\alpha_D}(B_{R_{\Omega}}(0)\backslash\Gamma)} 
\frac{\norm{v}_{H^{3/2+\alpha_D}(B_{R_{\Omega}}(0)\backslash\Gamma)}}
{\norm{w}_{H^{-1/2+\alpha_D}(B_{R_{\Omega}}(0)\backslash\Gamma)}}\nonumber\\
&\qquad\quad\lesssim \norm{g}_{H^{-\alpha_D}(\Gamma\cup \partial B_{R_{\Omega}}(0))},\nonumber
\end{align}
where the last inequality follows from Assumption~\ref{ass:shift}.

\emph{Proof of (\ref{item:lem:shiftapriori-ii}):}
With the bounded lifting operator  
$\mathcal{L}:H^{1+\epsilon}(\Gamma)\rightarrow H^{3/2+\epsilon}(B_{R_{\Omega}}(0)\backslash\Gamma)$ 
from \eqref{eq:deflifting}, 
the function $\widetilde{u}:=\eta^2 u - \eta\mathcal{L}(\eta g)$ satisfies 
\begin{align*}
-\Delta \widetilde{u} &= -2\eta\nabla \eta \cdot \nabla u - (\Delta \eta^2) u+ \Delta(\eta\mathcal{L}(\eta g)) 
&& \; \text{in}\; B_{R_{\Omega}}(0)\backslash\Gamma, \\
\gamma_0\widetilde{u} &= 0 && \; \text{on}\; \Gamma\cup \partial B_{R_{\Omega}}(0).
\end{align*}
With the shift theorem from Assumption~\ref{ass:shift} we get
\begin{eqnarray*}
\norm{u}_{H^{3/2+\epsilon}(B\backslash\Gamma)} &\leq& \norm{\eta^2 u}_{H^{3/2+\epsilon}(B\backslash\Gamma)}\leq 
\norm{\widetilde{u}}_{H^{3/2+\epsilon}(B_{R_{\Omega}}(0)\backslash\Gamma)}+
\norm{\eta\mathcal{L}(\eta g)}_{H^{3/2+\epsilon}(B_{R_{\Omega}}(0)\backslash\Gamma)} \\
&\lesssim&  \norm{2\eta\nabla\eta\cdot\nabla u + (\Delta \eta^2) u}_{L^2(B_{R_{\Omega}}(0)\backslash\Gamma)} + 
\norm{\Delta(\eta\mathcal{L}(\eta g))}_{H^{-1/2+\epsilon}(B_{R_{\Omega}}(0)\backslash\Gamma)} \\ & &+
\norm{\mathcal{L}(\eta g)}_{H^{3/2+\epsilon}(B_{R_{\Omega}}(0)\backslash\Gamma)}
 \\
&\lesssim& \norm{u}_{H^1(B'\backslash\Gamma)} + 
\norm{\mathcal{L}(\eta g)}_{H^{3/2+\epsilon}(B_{R_{\Omega}}(0)\backslash\Gamma)}
\lesssim \norm{u}_{H^1(B'\backslash\Gamma)} + \norm{\eta g}_{H^{1+\epsilon}(\Gamma)},
\end{eqnarray*}
which proves the second statement.
\end{proof}

The following lemma collects mapping properties of the single-layer operator $V$ 
that exploits the present setting of piecewise smooth geometries: 
\begin{lemma}\label{lem:potentialreg}
Define the \emph{single layer potential} $\widetilde{V}$ by 
\begin{equation}
\label{eq:def-Vtilde}
\widetilde{V}{\phi}(x) := \int_{\Gamma}G(x-y){\phi}(y) ds_y, 
\qquad x \in \mathbb{R}^d \setminus\Gamma. 
\end{equation}
\begin{enumerate}[(i)]
\item 
\label{item:lem:potentialreg-i}
The single layer potential $\widetilde{V}$ is a bounded linear operator 
from $H^{s}(\Gamma)$ to $H^{3/2+s}(B_{R_\Omega}(0)\backslash\Gamma)$ for $-1\leq s < 1/2$. 
\item 
\label{item:lem:potentialreg-ii}
The single-layer operator $V$ is a bounded linear operator from $H^{-1/2+s}(\Gamma)$ to $H^{1/2+s}(\Gamma)$
for $-1/2 \leq s < 1$. 
\item
\label{item:lem:potentialreg-iii}
The adjoint double-layer operator $K'$ is a bounded linear operator from $H^{-1/2+s}(\Gamma)$ to $H^{-1/2+s}(\Gamma)$
for $-1/2 \leq s < 1$. 
\end{enumerate}
\end{lemma}
\begin{proof}
\emph{Proof of (\ref{item:lem:potentialreg-i}):}
The case $s \in (-1,0)$ is shown in \cite[Thm.~{3.1.16}]{SauterSchwab}, and for $s=-1$ we refer to 
\cite{Verchota}. For the case $s \in [0,1/2)$, we 
exploit that $\Gamma$ is piecewise smooth. We split the polygonal boundary 
$\Gamma = \bigcup_{\ell=1}^m \overline{\Gamma_{\ell}}$ into its (smooth) faces 
$\Gamma_{\ell}$. Let $\chi_{\ell} \in L^2(\Gamma)$ be the characteristic function 
for $\Gamma_{\ell}$.
Then, for $\varphi \in H^{s}(\Gamma)$, we have 
$\widetilde{V} \varphi =\sum_{\ell=1}^m\widetilde{V}(\chi_{\ell}\varphi)$.
We prolong each face $\Gamma_{\ell}$ to the hyperplane $\Gamma_{\ell}^{\infty}$, 
which decomposes $\mathbb{R}^d$ into two half spaces $\Omega_{\ell}^{\pm }$. 
Due to $s <\frac{1}{2}$, we have $\chi_{\ell}\varphi \in H^{s}(\Gamma_{\ell}^{\infty})$. 
Since the half spaces $\Omega_{\ell}^{\pm}$ are smooth, we may use the mapping properties of 
$\widetilde{V}$ on smooth geometries, see, e.g., \cite[Thm.~{6.13}]{mclean00} to estimate
\begin{equation*}
\norm{\widetilde{V} \varphi }_{H^{3/2+s}(B_{R_\Omega}(0)\setminus\Gamma)} \lesssim 
\sum_{\ell=1}^m\norm{\widetilde{V}(\chi_{\ell}\varphi)}_{H^{3/2+s}(\Omega_{\ell}^{\pm })}
\lesssim \sum_{\ell=1}^m\norm{\chi_{\ell}\varphi}_{H^{s}(\Gamma_{\ell}^{\infty})}
\lesssim \norm{\varphi}_{H^{s}(\Gamma)}.
\end{equation*}
\emph{Proof of (\ref{item:lem:potentialreg-ii}):} The case $-1/2 \leq s \leq 1/2$ is taken from
\cite[Thm.~{3.1.16}]{SauterSchwab}. For $s \in (1/2,1)$ the result follows from 
part (\ref{item:lem:potentialreg-i}) and the definition of the norm $\|\cdot\|_{H^s(\Gamma)}$ given in
(\ref{eq:tracenorm}). 

\emph{Proof of (\ref{item:lem:potentialreg-iii}):} The case $-1/2 \leq s \leq 1/2$ is taken from
\cite[Thm.~{3.1.16}]{SauterSchwab}. With $K' = \partial_n \widetilde{V}-\frac{1}{2}\mathrm{Id}$ the case 
$s \in (1/2,1)$ follows from 
part (\ref{item:lem:potentialreg-i}) and a facewise trace estimate \eqref{eq:facewisenormalder} since 
\begin{equation*}
 \norm{\partial_n \widetilde{V}\varphi}_{H^{s-1/2}(\Gamma)} \lesssim \norm{\widetilde{V}\varphi}_{H^{1+s}(\Omega)}
 \lesssim \norm{\varphi}_{H^{s-1/2}(\Gamma)},
\end{equation*}
which finishes the proof.
\end{proof}

In addition to the single layer operator $V$, we will need to understand localized versions of these operators, 
i.e., the properties of commutators. For a smooth cut-off function $\eta$ we define the commutators 
\begin{eqnarray}\label{eq:commutator}
(C_{\eta} \widehat{\phi})(x) &:=& (V(\eta \widehat{\phi}) - \eta V(\widehat{\phi}))(x) = 
\int_{\Gamma}G(x,y)(\eta(x)-\eta(y))\widehat{\phi}(y) ds_y,\\
(C_{\eta}^{\eta} \widehat{\phi})(x) &:=& (C_{\eta}(\eta \widehat{\phi}) - \eta C_{\eta}(\widehat{\phi}))(x) = 
\int_{\Gamma}G(x,y)(\eta(x)-\eta(y))^2\widehat{\phi}(y) ds_y.
\end{eqnarray}
Since the singularity of the Green's function at $x=y$ is smoothed by $\eta(x)-\eta(y)$, we expect that
the commutators $C_{\eta}$, $C_{\eta}^{\eta}$ have better mapping properties than the single-layer operator, 
which is stated in the following lemma.

\begin{lemma}\label{lem:commutator}
Let $\eta \in C^\infty_0(\mathbb{R}^d)$ be fixed. 
\begin{enumerate}[(i)]
\item 
\label{item:lem:communtator-i}
The commutator $C_{\eta}$ is a skew-symmetric and
continuous mapping $C_{\eta}:H^{-1+\epsilon}(\Gamma)\rightarrow H^{1+\epsilon}(\Gamma)$ 
for all $-\alpha_D\leq \epsilon\leq \alpha_D$.
\item 
\label{item:lem:communtator-ii}
The commutator $C_{\eta}^{\eta}$ is a symmetric and continuous mapping 
$C_{\eta}^{\eta}:H^{-1-\alpha_D}(\Gamma) \rightarrow H^{1+\alpha_D}(\Gamma)$.
\end{enumerate}
In both cases, the continuity constant depends only on $\|\eta\|_{W^{1,\infty}(\mathbb{R}^d)}$, $\Omega$, 
and the constants appearing in Assumption~\ref{ass:shift}. 
\end{lemma}
\begin{proof}
\emph{Proof of (\ref{item:lem:communtator-i}):}
Since $V$ is symmetric, we have
\begin{equation*}
\skp{C_{\eta}\widehat{\phi},\psi} = \skp{V(\eta \widehat{\phi}) - \eta V(\widehat{\phi}),\psi} = 
\skp{\widehat{\phi},\eta V(\psi) - V(\eta\psi)} = - \skp{\widehat{\phi},C_{\eta}\psi},
\end{equation*}
i.e., the skew-symmetry of the commutator $C_{\eta}$. A similar computation proves 
the symmetry of the commutator $C_{\eta}^{\eta}$. \\

Let $0\leq\epsilon\leq \alpha_D$, $\widehat{\phi} \in H^{-1+\epsilon}(\Gamma)$, and 
$u:=\widetilde{C}_{\eta}\widehat{\phi} := \widetilde{V}(\eta \widehat{\phi})-\eta \widetilde{V}(\widehat{\phi})$
with the single-layer potential $\widetilde{V}$. 
Since the volume potential $\widetilde{V} \widehat{\phi} $ is harmonic
and in view of the jump relations
$[\gamma_0 \widetilde{V}\phi] = 0$, $[\partial_n \widetilde{V}\phi] = -\phi$
satisfied by $\widetilde{V}$, c.f. \cite[Thm.~{3.3.1}]{SauterSchwab}, we have
\begin{align*}
-\Delta u &= 2\nabla \eta \cdot \nabla \widetilde{V}\widehat{\phi} + \Delta \eta \widetilde{V}\widehat{\phi} 
&&
\; \text{in}\, \mathbb{R}^d\backslash\Gamma,\\
[\gamma_0 u] &=0, \quad
[\partial_n u] = 0 && \; \text{on}\, \Gamma. 
\end{align*}
We may write 
$u = \mathcal{N}(2\nabla \eta \cdot \nabla \widetilde{V}\widehat{\phi} + \Delta \eta \widetilde{V}\widehat{\phi})$ 
with the Newton potential 
\begin{equation}
\label{eq:newton-potential} 
\mathcal{N}f(x):= \int_{\mathbb{R}^d}G(x-y)f(y)dy, 
\end{equation}
since $u$ and $\mathcal{N}(2\nabla \eta \cdot \nabla \widetilde{V}\widehat{\phi} + \Delta \eta \widetilde{V}\widehat{\phi})$
have the same decay for $\abs{x}\rightarrow \infty$.
The mapping properties of the 
Newton potential (see, e.g., \cite[Thm.~{3.1.2}]{SauterSchwab}), 
as well as the mapping properties of $\widetilde{V}$ of Lemma~\ref{lem:potentialreg}, (\ref{item:lem:potentialreg-i})
provide
\begin{eqnarray}\label{eq:commutvolest}
\norm{u}_{H^{3/2+\epsilon}(B_{R_{\Omega}}(0)\backslash\Gamma)} &\lesssim& 
\norm{\nabla \eta \cdot \nabla \widetilde{V}\widehat{\phi} + 
\Delta \eta \widetilde{V}\widehat{\phi}}_{H^{-1/2+\epsilon}(B_{R_{\Omega}}(0)\backslash\Gamma)} \nonumber \\
&\lesssim& 
\norm{\widetilde{V}\widehat{\phi}}_{H^{1/2+\epsilon}(B_{R_{\Omega}}(0)\backslash\Gamma)}\lesssim 
\norm{\widehat{\phi}}_{H^{-1+\epsilon}(\Gamma)}.
\end{eqnarray}
The definition of $C_{\eta}$ and the definition of the norm $\norm{\cdot}_{H^{1+\epsilon}(\Gamma)}$ from 
\eqref{eq:tracenorm} prove the mapping properties of $C_{\eta}$ for $0 \leq \epsilon \leq \alpha_D$. The skew-symmetry of $C_{\eta}$ 
directly leads to the mapping properties for the case $-\alpha_D \leq \epsilon \leq 0$.

Using different mapping properties of the Newton potential 
(
see, e.g., \cite[Thm.~{3.1.2}]{SauterSchwab}), 
we may also estimate in the same way
\begin{eqnarray}\label{eq:commutvolest2}
\norm{\widetilde{C}_{\eta}\widehat{\phi}}_{H^{1/2+\epsilon}(B_{R_{\Omega}}(0)\backslash\Gamma)}=\norm{u}_{H^{1/2+\epsilon}(B_{R_{\Omega}}(0)\backslash\Gamma)} \lesssim 
\norm{\widetilde{V}\widehat{\phi}}_{H^{-1/2+\epsilon}(B_{R_{\Omega}}(0)\backslash\Gamma)}.
\end{eqnarray}

\emph{Proof of (\ref{item:lem:communtator-ii}):}
Let $v:=\widetilde{C_{\eta}^{\eta}}\widehat{\phi} := 
\widetilde{C}_{\eta}(\eta\widehat{\phi})-\eta \widetilde{C}_{\eta}\widehat{\phi}$. 
Since 
$$
\Delta \widetilde{C}_{\eta}(\eta\widehat{\phi})-\eta \Delta \widetilde{C}_{\eta} \widehat{\phi}  = 
-2\nabla \eta \cdot \nabla \widetilde{C}_{\eta}\widehat{\phi} - 
\Delta \eta \widetilde{C}_{\eta}\widehat{\phi}-2\abs{\nabla \eta}^2\widetilde{V}\widehat{\phi},
$$
the function $v$ solves 
\begin{align*}
-\Delta v &= 4\nabla \eta \cdot \nabla \widetilde{C}_{\eta}\widehat{\phi} + 
2\Delta \eta \widetilde{C}_{\eta}\widehat{\phi} +
2\abs{\nabla \eta}^2\widetilde{V}\widehat{\phi}
&&\; \text{in}\, \mathbb{R}^d\backslash\Gamma,\\
[\gamma_0 v] &=0, \quad
[\partial_n v] = 0 &&\; \text{on}\, \Gamma. 
\end{align*}
Again, the function $v$ and the Newton potential have the same decay for $\abs{x}\rightarrow\infty$, and
the mapping properties of the Newton potential 
as well as the previous estimate \eqref{eq:commutvolest2} for $\widetilde{C}_{\eta}\widehat{\phi}$ provide
\begin{eqnarray}\label{eq:commutatorvolreg}
\norm{v}_{H^{3/2+\alpha_D}(B_{R_{\Omega}}(0)\backslash\Gamma)} &\lesssim& 
\norm{4\nabla \eta \cdot \nabla \widetilde{C}_{\eta}\widehat{\phi} + 
2\Delta \eta \widetilde{C}_{\eta}\widehat{\phi}+
2\abs{\nabla \eta}^2\widetilde{V}\widehat{\phi}}_{H^{-1/2+\alpha_D}(B_{R_{\Omega}}(0)\backslash\Gamma)} \nonumber
\\
\nonumber 
&\lesssim& \norm{\widetilde{C}_{\eta}\widehat{\phi}}_{H^{1/2+\alpha_D}(B_{R_{\Omega}}(0)\backslash\Gamma)}+
\norm{\widetilde{V}\widehat{\phi}}_{H^{-1/2+\alpha_D}(B_{R_{\Omega}}(0)\backslash\Gamma)}  \\
&\lesssim & \norm{\widetilde{V}\widehat{\phi}}_{H^{-1/2+\alpha_D}(B_{R_{\Omega}}(0)\backslash\Gamma)}
\stackrel{\alpha_D < 1/2}{\lesssim} 
\norm{\widetilde{V}\widehat{\phi}}_{H^{1/2-\alpha_D}(B_{R_{\Omega}}(0)\backslash\Gamma)}. 
\end{eqnarray}
We apply Lemma~\ref{lem:shiftapriori} to $\widetilde{V}\widehat{\phi}$. 
Since $\operatorname*{dist}(\Gamma,\partial B_{R_\Omega}(0)) > 0$ we have that $\widetilde{V}\widehat{\phi}$ 
is smooth on $\partial B_{R_{\Omega}}(0)$, and we can estimate this term
by an arbitrary negative norm of $\widehat{\phi}$ on $\Gamma$ to obtain
\begin{eqnarray*}
\norm{\widetilde{V}\widehat{\phi}}_{H^{1/2-\alpha_D}(B_{R_{\Omega}}(0)\backslash\Gamma)} 
&\stackrel{(\ref{eq:lem:shiftapriori-10})}{\lesssim} &
 \norm{\gamma_0(\widetilde{V}\widehat{\phi})}_{H^{-\alpha_D}(\Gamma\cup\partial B_{R_{\Omega}}(0))}
\lesssim 
  \norm{V\widehat{\phi}}_{H^{-\alpha_D}(\Gamma)}+  \norm{\widehat{\phi}}_{H^{-1-\alpha_D}(\Gamma)}.
\end{eqnarray*}
The additional mapping properties of $V$ of Lemma~\ref{lem:potentialreg}, (\ref{item:lem:potentialreg-ii}) 
and the symmetry of $V$ imply
\begin{eqnarray*}
\norm{V\widehat{\phi}}_{H^{-\alpha_D}(\Gamma)} &=& \sup_{w\in H^{\alpha_D}(\Gamma)}
\frac{\skp{V\widehat{\phi},w}_{L^2(\Gamma)}}{\norm{w}_{H^{\alpha_D}(\Gamma)}} = 
\sup_{w\in H^{\alpha_D}(\Gamma)}
\frac{\skp{\widehat{\phi},Vw}_{L^2(\Gamma)}}{\norm{w}_{H^{\alpha_D}(\Gamma)}} \\
&\lesssim& \sup_{w\in H^{\alpha_D}(\Gamma)}
\frac{\norm{\widehat{\phi}}_{H^{-1-\alpha_D}(\Gamma)}\norm{Vw}_{H^{1+\alpha_D}(\Gamma)}}
{\norm{w}_{H^{\alpha_D}(\Gamma)}} \lesssim \norm{\widehat{\phi}}_{H^{-1-\alpha_D}(\Gamma)}.
\end{eqnarray*}
Inserting this in 
\eqref{eq:commutatorvolreg} leads to
\begin{equation*}
\norm{v}_{H^{3/2+\alpha_D}(B_{R_{\Omega}}(0)\backslash\Gamma)}
\lesssim \norm{\widehat{\phi}}_{H^{-1-\alpha_D}(\Gamma)},
\end{equation*}
which, together with the definition of the $H^{1+\alpha_D}(\Gamma)$-norm from \eqref{eq:tracenorm}, 
proves the lemma.
\end{proof}


The shift theorem for the Neumann problem from Assumption~\ref{ass:shift2} 
implies the following shift theorem.

\begin{lemma}\label{lem:shiftaprioriNeumann}
Let the shift theorem from Assumption~\ref{ass:shift2} hold, and let $u$ be the solution 
of the inhomogeneous problems 
\begin{align*}
 -\Delta u &= 0 \quad \mbox{in $\Omega$}, \qquad
&&\gamma_1^{\rm int} u = g_N \quad \mbox{on $\Gamma$}, &&\qquad \skp{u,1}_{L^2(\Omega)} = 0,\\
-\Delta u &= 0 \quad \mbox{in $B_{R_{\Omega}}(0)\backslash\overline{\Omega}$}, \qquad
&&\gamma_1^{\rm ext} u = g_N \quad \mbox{on $\Gamma$},
&&\qquad \gamma_0^{\rm int} u = g_D \quad \mbox{on $\partial B_{R_{\Omega}}(0)$},
\end{align*}
where
$g_N\in H^{-1/2}(\Gamma)$ with 
$\skp{g_N,1}_{L^2(\Gamma)} = 0$, and $g_D\in H^{1/2}(\partial B_{R_{\Omega}}(0))$.

\begin{enumerate}[(i)]
 \item 
\label{item:lem:shiftaprioriNeumann-i}
There is a constant $C>0$ depending only on $\Omega$ and $\alpha_N$ such that 
\begin{equation}
\label{eq:lem:shiftaprioriNeumann-10}
\norm{u}_{H^{1/2-\alpha_N}(B_{R_{\Omega}}(0)\backslash\Gamma)}\leq 
C \left(\norm{g_N}_{H^{-1-\alpha_N}(\Gamma)} 
+ \norm{g_D}_{H^{-\alpha_N}(\partial B_{R_{\Omega}}(0))}\right).  
\end{equation}

\item 
\label{item:lem:shiftaprioriNeumann-ii}
Let $\epsilon \in (0,\alpha_N]$.
Let $B\subset B' \subset B_{R_{\Omega}}(0)$ be nested subdomains  with $\operatorname*{dist}(B,\partial B') > 0$ and
$\eta$ be a cut-off function $\eta \in C_0^{\infty}(\mathbb{R}^d)$ satisfying $\eta \equiv 1$ on $B$,
$\operatorname*{supp} \eta \subset \overline{B'}$, and $\norm{\eta}_{C^k(B')}\lesssim \operatorname*{dist}(B,\partial B')^{-k}$ 
for $k \in \{0,1,2\}$.
Assume $\eta g_N\in H^{\epsilon}(\Gamma)$. Then
 \begin{equation}
\label{eq:lem:shiftaprioriNeumann-20}
 \norm{u}_{H^{3/2+\epsilon}(B\backslash\Gamma)} \leq C\left(\norm{u}_{H^1(B'\backslash\Gamma)}
+\norm{\eta g_N}_{H^{\epsilon}(\Gamma)}\right).
\end{equation}
Here, the constant $C>0$ depends on $\Omega, \alpha_N$, and $\operatorname*{dist}(B,\partial B')$.
\end{enumerate}

\end{lemma}
\begin{proof}
\emph{Proof of (\ref{item:lem:shiftaprioriNeumann-i}):}
Let $v$ solve 
\begin{align*}
 -\Delta v &= w-\overline{w} \quad \mbox{in $\Omega$}, \qquad
&&\gamma_1^{\rm int} v = 0 \quad \mbox{on $\Gamma$}, &&\qquad \skp{v,1}_{L^2(\Omega)} = 0,\\
-\Delta v &= w \quad \mbox{in $B_{R_{\Omega}}(0)\backslash\overline{\Omega}$}, \qquad
&&\gamma_1^{\rm ext} v = 0 \quad \mbox{on $\Gamma$},
&&\qquad \gamma_0^{\rm int} v = 0 \quad \mbox{on $\partial B_{R_{\Omega}}(0)$},
\end{align*}
for 
$w \in H^{-1/2+\alpha_N}(B_{R_{\Omega}}(0)\backslash\Gamma)$ and 
$\overline{w}:=\frac{1}{\abs{\Omega}}\skp{w,1}_{L^2(\Omega)}$. 
Then, with $\skp{u,1}_{L^2(\Omega)} = 0$ we have 
\begin{eqnarray*}
\norm{u}_{H^{1/2-\alpha_N}(B_{R_{\Omega}}(0)\backslash\Gamma)} &=& 
\sup_{w \in H^{-1/2+\alpha_N}(B_{R_{\Omega}}(0)\backslash\Gamma)} 
\frac{\skp{u,w}_{L^2(B_{R_{\Omega}}(0)\backslash\Gamma)}}
{\norm{w}_{H^{-1/2+\alpha_N}(B_{R_{\Omega}}(0)\backslash\Gamma)}} \\ &=&
\sup_{w \in H^{-1/2+\alpha_N}(B_{R_{\Omega}}(0)\backslash\Gamma)} 
\frac{\skp{u,w-\overline{w}}_{L^2(\Omega)}+\skp{u,w}_{L^2(B_{R_{\Omega}}(0)\backslash\overline{\Omega})}}
{\norm{w}_{H^{-1/2+\alpha_N}(B_{R_{\Omega}}(0)\backslash\Gamma)}} \\
&=&
\sup_{w \in H^{-1/2+\alpha_N}(B_{R_{\Omega}}(0)\backslash\Gamma)} 
\frac{-\skp{u,\Delta v}_{L^2(B_{R_{\Omega}}(0)\backslash\Gamma)}}
{\norm{w}_{H^{-1/2+\alpha_N}(B_{R_{\Omega}}(0)\backslash\Gamma)}}.
\end{eqnarray*}
Integration by parts on $\Omega$ and $B_{R_{\Omega}}(0)\backslash\overline{\Omega}$ leads to 
\begin{eqnarray*}
\skp{u,\Delta v}_{L^2(B_{R_{\Omega}}(0)\backslash\Gamma)} &=& 
\skp{\Delta u,v}_{L^2(B_{R_{\Omega}}(0)\backslash\Gamma)} -
 \skp{\partial_n u,[\gamma_0 v]}_{L^2(\Gamma)} +
 \skp{\gamma_0 u,\partial_n v}_{L^2(\partial B_{R_{\Omega}}(0))} \\ &=&
 -\skp{g_N,[\gamma_0 v]}_{L^2(\Gamma)} +
 \skp{g_D,\partial_n v}_{L^2(\partial B_{R_{\Omega}}(0))}.
\end{eqnarray*}
The definition of the norm \eqref{eq:tracenorm} implies
\begin{eqnarray*}
\norm{\gamma_0^{\rm int} v}_{H^{1+\alpha_N}(\Gamma)} \lesssim
 \norm{v}_{H^{3/2+\alpha_N}(B_{R_{\Omega}}(0)\backslash\Gamma)},
\end{eqnarray*}
and the same estimate holds for $\gamma_0^{\rm ext} v$. Since $\partial B_{R_{\Omega}}(0)$ is smooth, we
may estimate
\begin{eqnarray*}
\norm{\partial_n v}_{H^{\alpha_N}(\partial B_{R_{\Omega}}(0))} \lesssim
 \norm{v}_{H^{3/2+\alpha_N}(B_{R_{\Omega}}(0)\backslash\Gamma)}.
\end{eqnarray*}
This leads to
\begin{align}
&\norm{u}_{H^{1/2-\alpha_N}(B_{R_{\Omega}}(0)\backslash\Gamma)} \lesssim
\sup_{w \in H^{-1/2+\alpha_N}(B_{R_{\Omega}}(0)\backslash\Gamma)} 
\frac{\abs{\skp{g_N,[\gamma_0 v]}_{L^2(\Gamma)} -
 \skp{g_D,\partial_n v}_{L^2(\partial B_{R_{\Omega}}(0))}}}{\norm{w}_{H^{-1/2+\alpha_N}(B_{R_{\Omega}}(0)\backslash\Gamma)}}\nonumber \\
&\qquad\quad\lesssim \sup_{w\in H^{-1/2+\alpha_N}(B_{R_{\Omega}}(0)\backslash\Gamma)} 
\frac{\norm{g_N}_{H^{-1-\alpha_N}(\Gamma)}\norm{[\gamma_0 v]}_{H^{1+\alpha_N}(\Gamma)}+
\norm{g_D}_{H^{-\alpha_N}(\partial B_{R_{\Omega}}(0))}\norm{\partial_n v}_{H^{\alpha_N}(\partial B_{R_{\Omega}}(0))}}{\norm{w}_{H^{-1/2+\alpha_N}(B_{R_{\Omega}}(0)\backslash\Gamma)}}\nonumber
\\ &\qquad\quad\lesssim 
\left(\norm{g_N}_{H^{-1-\alpha_N}(\Gamma)}+\norm{g_D}_{H^{-\alpha_N}(\partial B_{R_{\Omega}}(0))}\right)\sup_{w \in H^{-1/2+\alpha_N}(B_{R_{\Omega}}(0)\backslash\Gamma)} 
\frac{\norm{v}_{H^{3/2+\alpha_N}(B_{R_{\Omega}}(0)\backslash\Gamma)}}{\norm{w}_{H^{-1/2+\alpha_N}(B_{R_{\Omega}}(0)\backslash\Gamma)}}\nonumber\\
&\qquad\quad\lesssim \norm{g_N}_{H^{-1-\alpha_N}(\Gamma)}+\norm{g_D}_{H^{-\alpha_N}(\partial B_{R_{\Omega}}(0))},\nonumber
\end{align}
where the last inequality follows from Assumption~\ref{ass:shift2}. 

\emph{Proof of (\ref{item:lem:shiftaprioriNeumann-ii}):} 
Since $\eta \equiv 0$ on 
$\partial B_{R_{\Omega}}$, the function $\widetilde{u}:=\eta u$ satisfies
\begin{align*}
-\Delta \widetilde{u} &= -2\nabla \eta \cdot \nabla u - (\Delta \eta) u
&& \; \text{in}\; B_{R_{\Omega}}(0)\backslash\Gamma, \\
\gamma_1^{\rm int} \widetilde{u} &= (\partial_n \eta) \gamma_0^{\rm int}u + \eta g_N && \; \text{on}\; \Gamma,\\
\gamma_1^{\rm ext} \widetilde{u} &= (\partial_n \eta) \gamma_0^{\rm ext}u + \eta g_N && \; \text{on}\; \Gamma, \\
\gamma_0^{\rm int} \widetilde{u} &= 0 && \; \text{on}\; \partial B_{R_{\Omega}}(0).
\end{align*}
With the shift theorem from Assumption~\ref{ass:shift2} we get with the trace inequality
$\norm{(\partial_n \eta) \gamma_0^{\rm int}u}_{H^{1/2}(\Gamma)} 
\lesssim \norm{u}_{H^1(B'\backslash\Gamma)}$ that
\begin{eqnarray*}
\norm{u}_{H^{3/2+\epsilon}(B\backslash\Gamma)} \leq  
\norm{\widetilde{u}}_{H^{3/2+\epsilon}(B_{R_{\Omega}}(0)\backslash\Gamma)}
&\lesssim&  \norm{\nabla\eta\cdot\nabla u + (\Delta \eta) u}_{L^2(B_{R_{\Omega}}(0)\backslash\Gamma)}  
\\ & &+
\norm{(\partial_n\eta) \gamma_0^{\rm int} u}_{H^{\epsilon}(\Gamma)} + 
\norm{(\partial_n\eta) \gamma_0^{\rm ext} u}_{H^{\epsilon}(\Gamma)} +
\norm{\eta g_N}_{H^{\epsilon}(\Gamma)} 
 \\
&\lesssim& \norm{u}_{H^1(B'\backslash\Gamma)} + \norm{\eta g_N}_{H^{\epsilon}(\Gamma)},
\end{eqnarray*}
which proves the second statement.
\end{proof}

The following lemma collects mapping properties of the double-layer operator $K$ 
and the hyper-singular operator $W$
that exploit the present setting of piecewise smooth geometries: 
\begin{lemma}\label{lem:potentialregK}
Define the \emph{double layer potential} $\widetilde{K}$ by 
\begin{equation}
\label{eq:def-Ktilde}
\widetilde{K}{\varphi}(x) := \int_{\Gamma}\partial_{n_y}G(x-y){\varphi}(y) ds_y, 
\qquad x \in \mathbb{R}^d \setminus\Gamma. 
\end{equation}
\begin{enumerate}[(i)]
\item 
\label{item:lem:potentialregK-i}
The double layer potential $\widetilde{K}$ is a bounded linear operator 
from $H^{1+s}(\Gamma)$ to $H^{3/2+s}(B_{R_\Omega}(0)\backslash\Gamma)$ for $-1\leq s \leq \alpha_N$. 
\item 
\label{item:lem:potentialregK-ii}
The double layer operator $K$ is a bounded linear operator from $H^{1/2+s}(\Gamma)$ to $H^{1/2+s}(\Gamma)$
for $-1/2 \leq s \leq 1/2+\alpha_N$. 
\item \label{item:lem:potentialregK-iii}
The hyper singular operator $W$ is a bounded linear operator from $H^{1/2+s}(\Gamma)$ to $H^{-1/2+s}(\Gamma)$
for $-1/2-\alpha_N \leq s \leq 1/2+\alpha_N$. 
\end{enumerate}
\end{lemma}
\begin{proof}
\emph{Proof of (\ref{item:lem:potentialregK-i}):}
With the mapping properties of the single layer potential $\widetilde{V}\in L(H^{s}(\Gamma),H^{3/2+s}(B_{R_\Omega}(0)\backslash\Gamma))$ from 
Lemma~\ref{lem:potentialreg}, the mapping properties of the solution operator of the Dirichlet problem from
Assumption~\ref{ass:shift} 
($T:H^{1+s}(\Gamma)\rightarrow H^{3/2+s}(B_{R_\Omega}(0)\backslash\Gamma)$), and the assumption $\alpha_N \leq \alpha_D$, 
the mapping properties 
of $\widetilde{K}$ follow from Green's formula by expressing $\widetilde{K}$ in terms of $\widetilde{V}$, $T$, and
the Newton potential $\mathcal{N}$. For details, we refer to \cite[Thm.~{3.1.16}]{SauterSchwab}, where the 
case $s \in (-1,0)$ is shown. 

\emph{Proof of (\ref{item:lem:potentialregK-ii}):} The case $-1/2 \leq s \leq 1/2$ is taken from
\cite[Thm.~{3.1.16}]{SauterSchwab}. For $s \in (1/2,1/2+\alpha_N]$ the result follows from 
part (\ref{item:lem:potentialreg-i}), the definition of the norm $\|\cdot\|_{H^{s+1/2}(\Gamma)}$ given in
(\ref{eq:tracenorm}), and $K = \gamma_0^{\rm int}\widetilde{K}+\frac{1}{2}\mathrm{Id}$.

\emph{Proof of (\ref{item:lem:potentialregK-iii}):} The case $-1/2 \leq s \leq 1/2$ is taken from
\cite[Thm.~{3.1.16}]{SauterSchwab}. Since $W = -\partial_n \widetilde{K}$, we get with a facewise trace estimate
as in the proof of Lemma~\ref{lem:shiftapriori}, estimate \eqref{eq:facewisenormalder}, that
\begin{equation*}
\norm{W\varphi}_{H^{-1/2+s}(\Gamma)}=\norm{\partial_n \widetilde{K}\varphi}_{H^{-1/2+s}(\Gamma)} \lesssim 
\norm{\widetilde{K}\varphi}_{H^{1+s}(\Omega)} \lesssim \norm{\varphi}_{H^{1/2+s}(\Gamma)},
\end{equation*}
which finishes the proof for the case $s \in (1/2,1/2+\alpha_N]$.
With the symmetry of $W$, the case $s \in [-1/2-\alpha_N,-1/2)$ follows.
\end{proof}

For a smooth function $\eta$, we define the commutators 
\begin{eqnarray}\label{eq:commutatorHS}
\mathcal{C}_{\eta} \widehat{\varphi} &:=& W(\eta \widehat{\varphi}) - \eta W \widehat{\varphi}, \\
\mathcal{C}_{\eta}^{\eta}(\widehat{\varphi}) &:=& \mathcal{C}_{\eta}(\eta \widehat{\varphi}) - \eta \mathcal{C}_{\eta}(\widehat{\varphi}) = 
W(\eta^2\widehat{\varphi})-2\eta W(\eta\widehat{\varphi})+\eta^2 W (\widehat{\varphi}).  
\end{eqnarray}
By the mapping properties of $W$, both operators map $H^{1/2}(\Gamma) \rightarrow H^{-1/2}(\Gamma)$. However, 
$\mathcal{C}_{\eta}$ is in fact an operator of order $0$ and $\mathcal{C}_{\eta}^{\eta}$ is an operator of positive order:
 
\begin{lemma}\label{lem:commutatorHypSing}
Fix $\eta \in C_0^\infty(\mathbb{R}^d)$. 
\begin{enumerate}[(i)]
 \item\label{item:lem:commutatorK-i}
 Let $s \in (-1/2,1/2)$. Then, the commutator $\mathcal{C}_\eta$ 
can be extended in a unique way to a bounded linear operator $H^s(\Gamma) \rightarrow H^s(\Gamma)$ that 
satisfies the bound 
\begin{equation}
\label{eq:lem:commutatorHypSing-10}
\|{\mathcal C}_\eta \varphi\|_{H^{s}(\Gamma)} \leq C \|\varphi\|_{H^s(\Gamma)} 
\qquad \forall \varphi \in H^s(\Gamma). 
\end{equation}
The constant $C$ depends only on $\Omega$ and the choice of $s$.  
Furthermore, the operator 
is skew-symmetric (with respect to the extended $L^2$-inner product).

 \item\label{item:lem:commutatorK-ii}
The commutator $\mathcal{C}_{\eta}^{\eta}$ is a symmetric and continuous mapping 
$\mathcal{C}_{\eta}^{\eta}:H^{-\alpha_N}(\Gamma) \rightarrow H^{\alpha_N}(\Gamma)$.
The continuity constant depends only on $\|\eta\|_{W^{1,\infty}(\mathbb{R}^d)}$, $\Omega$, and the constants
appearing in Assumption~\ref{ass:shift2}. 

\end{enumerate}
\end{lemma}
\begin{proof}
\emph{Proof of (\ref{item:lem:commutatorK-i}):}
\emph{1.~step:} We show (\ref{eq:lem:commutatorHypSing-10}) for the range $0 < s < 1/2$. 
For $\varphi \in H^{1/2}(\Gamma)$, consider the potential 
$\widetilde{\mathcal{C}}_{\eta}\varphi := \widetilde{K}(\eta \varphi)-\eta \widetilde{K}(\varphi) - 
\widetilde{V}((\partial_n \eta) \varphi)$ with the
single layer potential $\widetilde{V}$ and the double layer potential $\widetilde{K}$ 
from \eqref{eq:def-Ktilde}.

Using the jump conditions $[\gamma_0 \widetilde{V}\phi] = 0$, $[\partial_n \widetilde{V}\phi] = -\phi$ 
for $\widetilde V$ and additionally the jump relations
$[\gamma_0 \widetilde{K}\phi] = \phi$, $[\partial_n \widetilde{K}\phi] = 0$ 
satisfied by $\widetilde K$ from \cite[Thm.~{3.3.1}]{SauterSchwab}, 
we observe that 
the function $u:= \widetilde{\mathcal{C}}_{\eta}\varphi$ solves 
\begin{align*}
-\Delta u &= 2\nabla \eta \cdot \nabla \widetilde{K}\varphi + (\Delta \eta) \widetilde{K}\varphi && \; \text{in} \; \mathbb{R}^d\backslash\Gamma, \\
[\gamma_0 u] &=0, \quad
[\partial_n u] = -\partial_n \eta [\gamma_0 \widetilde{K}\varphi] - [\partial_n \widetilde{V}(\partial_n \eta \varphi)] = 0 &&\; \text{on} \; \Gamma. 
\end{align*}
The decay of $u$ - the dominant part is the single-layer potential - and the Newton potential
$\mathcal{N}(2\nabla \eta \cdot \nabla \widetilde{K}\varphi + (\Delta \eta) \widetilde{K}\varphi)$
for $\abs{x}\rightarrow \infty$ are the same, which allows us to write 
$u=\mathcal{N}(2\nabla \eta \cdot \nabla \widetilde{K}\varphi + (\Delta \eta) \widetilde{K}\varphi)$.
With the mapping properties of the Newton potential and the standard mapping properties of 
$\widetilde{K}$ from \cite[Thm.~{3.1.16}]{SauterSchwab} it follows that
\begin{equation}\label{eq:estCtilde}
\norm{u}_{H^{3/2+s}(B_{R_{\Omega}}(0)\backslash\Gamma)} \lesssim \norm{\nabla \eta \cdot \nabla \widetilde{K}\varphi + 
\Delta \eta \widetilde{K}\varphi}_{H^{-1/2+s}(B_{R_{\Omega}}(0)\backslash\Gamma)} 
\lesssim \norm{\widetilde{K}\varphi}_{H^{1/2+s}(B_{R_{\Omega}}(0)\backslash\Gamma)}\lesssim 
\norm{\varphi}_{H^{s}(\Gamma)}.
\end{equation}
The trace estimate applied facewise as in the proof of Lemma~\ref{lem:shiftapriori}, estimate \eqref{eq:facewisenormalder},
and \eqref{eq:estCtilde} lead to 
\begin{equation}
\label{eq:100}
\norm{\partial_n \widetilde{\mathcal C}_\eta \varphi}_{H^s(\Gamma)} = \norm{\partial_n u}_{H^{s}(\Gamma)} 
\lesssim \norm{u}_{H^{3/2+s}(B_{R_{\Omega}}(0)\backslash\Gamma)}  \lesssim \norm{\varphi}_{H^{s}(\Gamma)}.
\end{equation}
Furthermore, using Lemma~\ref{lem:potentialreg}, (\ref{item:lem:potentialreg-i}) we arrive at 
\begin{equation}
\label{eq:200}
\|\partial_n \widetilde V (\eta \varphi)\|_{H^{s}(\Gamma)} \lesssim 
\|\widetilde V (\eta \varphi)\|_{H^{3/2+s}(B_{R_{\Omega}}(0)\backslash\Gamma)} \lesssim \|\eta \varphi\|_{H^{s}(\Gamma)}. 
\end{equation}
Next, we identify $\partial_n \widetilde {\mathcal C}_\eta$. With 
$W  = -\partial_n \widetilde K$, $K' = \partial_n \widetilde V - \frac{1}{2}\operatorname*{Id}$, 
$K = \frac{1}{2} \operatorname*{Id} + \gamma_0^{\rm int} \widetilde K$, we compute 
$$
\partial_n \widetilde{\mathcal C}_\eta = \eta  W \varphi - W (\eta \varphi) - K'((\partial_n \eta) \varphi) - 
(\partial_n \eta) K \varphi 
$$
Recalling the mapping property $K',K:H^s(\Gamma)\rightarrow H^s(\Gamma)$ and the relation 
$\partial_n \widetilde V = \frac{1}{2} \operatorname*{Id} - K'$ we get with the aid of (\ref{eq:100}), (\ref{eq:200})
\begin{equation}
\label{eq:3000} 
\|W(\eta \varphi) - \eta W \varphi\|_{H^s(\Gamma)} \lesssim \|\varphi\|_{H^s(\Gamma)}. 
\end{equation}
\emph{2.~step:} Since $H^{1/2}(\Gamma)$ is dense in $H^s(\Gamma)$, $s \in (0,1/2)$, the operator ${\mathcal C}_\eta$
can be extended (in a unique way) to a bounded linear operator $H^s(\Gamma) \rightarrow H^s(\Gamma)$. 

\emph{3.~step:} The operator ${\mathcal C}_\eta$ is skew-symmetric: The operator $W$ maps 
$H^{1/2}(\Gamma) \rightarrow H^{-1/2}(\Gamma)$ and is symmetric. The skew-symmetry of ${\mathcal C}_\eta$ then 
follows from a direct calculation. 

\emph{4.~step:} The skew-symmetry of ${\mathcal C}_\eta$ allows us to extend (in a unique way) the operator 
as an operator $H^{-s}(\Gamma) \rightarrow H^{-s}(\Gamma)$ for $0<s<1/2$ by the following argument: 
For $\varphi$, $\psi \in H^{1/2}(\Gamma)$ we compute 
\begin{equation}
\label{eq:400}
\langle {\mathcal C}_\eta \varphi, \psi\rangle = 
- \langle \varphi, {\mathcal C}_\eta \psi\rangle.  
\end{equation}
Since ${\mathcal C}_\eta: H^{s}(\Gamma) \rightarrow H^{s}(\Gamma)$ for $0 < s < 1/2$, we see that, 
$\varphi \mapsto \langle \varphi, {\mathcal C}_\eta \psi\rangle$ 
on the right-hand side of (\ref{eq:400}) extends to a bounded linear functional on $H^{-s}(\Gamma)$. 
Hence, ${\mathcal C}_\eta: H^{-s}(\Gamma) \rightarrow H^{-s}(\Gamma)$ for $0 < s < 1/2$. 

\emph{5.~step:} 
We have 
${\mathcal C}_\eta: H^s(\Gamma) \rightarrow H^s(\Gamma)$ for $s \in (-1/2,1/2)\setminus \{0\}$. 
An interpolation argument allows us to extend the boundedness to the remaining case $s = 0$. \\

\emph{Proof of (\ref{item:lem:commutatorK-ii}):}
Since 
$$
\Delta \widetilde{\mathcal{C}}_{\eta}(\eta\widehat{\varphi})-\eta \Delta \widetilde{\mathcal{C}}_{\eta} \widehat{\varphi}  = 
-2\nabla \eta \cdot \nabla \widetilde{\mathcal{C}}_{\eta}\widehat{\varphi} - 
\Delta \eta \widetilde{\mathcal{C}}_{\eta}\widehat{\varphi}-2\abs{\nabla \eta}^2\widetilde{K}\widehat{\varphi}
-\Delta(\eta \widetilde{V}(\partial_n \widehat{\varphi})),
$$
the function $v:=\widetilde{\mathcal{C}}_{\eta}^{\eta}\widehat{\varphi} := \widetilde{\mathcal{C}}_{\eta}(\eta\widehat{\varphi})-\eta \widetilde{\mathcal{C}}_{\eta}\widehat{\varphi}$. 
 solves 
\begin{align*}
-\Delta v &= 4\nabla \eta \cdot \nabla \widetilde{\mathcal{C}}_{\eta}\widehat{\varphi} + 
2\Delta \eta \widetilde{\mathcal{C}}_{\eta}\widehat{\varphi} +
2\abs{\nabla \eta}^2\widetilde{K}\widehat{\varphi} + \Delta(\eta \widetilde{V}(\partial_n \widehat{\varphi}))
&&\; \text{in}\, \mathbb{R}^d\backslash\Gamma,\\
[\gamma_0 v] &=0, \quad
[\partial_n v] = 0 &&\; \text{on}\, \Gamma. 
\end{align*}
Again, the decay of $v$ and the Newton potential applied to the right-hand side of the equation 
are the same, and the mapping properties of the Newton potential provide
\begin{eqnarray}\label{eq:commutatorvolregW}
\norm{v}_{H^{3/2+\alpha_N}(B_{R_{\Omega}}(0)\backslash\Gamma)} &\lesssim& 
\norm{4\nabla \eta \cdot \nabla \widetilde{\mathcal{C}}_{\eta}\widehat{\varphi} + 
2\Delta \eta \widetilde{\mathcal{C}}_{\eta}\widehat{\varphi}+
2\abs{\nabla \eta}^2\widetilde{K}\widehat{\varphi}+\Delta(\eta \widetilde{V}(\partial_n \widehat{\varphi}))}_{H^{-1/2+\alpha_N}(B_{R_{\Omega}}(0)\backslash\Gamma)} \nonumber
\\
\nonumber 
&\lesssim& \norm{\widetilde{\mathcal{C}}_{\eta}\widehat{\varphi}}_{H^{1/2+\alpha_N}(B_{R_{\Omega}}(0)\backslash\Gamma)}+
\norm{\widetilde{K}\widehat{\varphi}}_{H^{-1/2+\alpha_N}(B_{R_{\Omega}}(0)\backslash\Gamma)} \\ \nonumber& &+
\norm{\widetilde{V}\widehat{\varphi}}_{H^{1/2+\alpha_N}(B_{R_{\Omega}}(0)\backslash\Gamma)}\\
&\stackrel{\mathrm{Lemma~\ref{lem:potentialreg}}}{\lesssim} & \norm{\widetilde{K}\widehat{\varphi}}_{H^{-1/2+\alpha_N}(B_{R_{\Omega}}(0)\backslash\Gamma)} + 
\norm{\widehat{\varphi}}_{H^{-1+\alpha_N}(\Gamma)} \nonumber \\
&\stackrel{\alpha_N < 1/2}{\lesssim}& 
\norm{\widetilde{K}\widehat{\varphi}}_{H^{1/2-\alpha_N}(B_{R_{\Omega}}(0)\backslash\Gamma)}+
\norm{\widehat{\varphi}}_{H^{-1+\alpha_N}(\Gamma)}. 
\end{eqnarray}
We apply Lemma~\ref{lem:shiftaprioriNeumann} to $\widetilde{K}\widehat{\varphi}-\overline{\widetilde{K}\widehat{\varphi}}$. 
Since $\operatorname*{dist}(\Gamma,\partial B_{R_\Omega}(0)) > 0$ we have that $\widetilde{K}\widehat{\varphi}$ 
is smooth on $\partial B_{R_{\Omega}}(0)$, and we can estimate this term
by an arbitrary negative norm of $\widehat{\varphi}$ on $\Gamma$ to obtain
\begin{eqnarray*}
\norm{\widetilde{K}\widehat{\varphi}-\overline{\widetilde{K}\widehat{\varphi}}}_{H^{1/2-\alpha_N}(B_{R_{\Omega}}(0)\backslash\Gamma)} 
&\stackrel{(\ref{eq:lem:shiftaprioriNeumann-10})}{\lesssim}& 
 \norm{W\widehat{\varphi}}_{H^{-1-\alpha_N}(\Gamma)} + 
  \norm{\widetilde{K}\widehat{\varphi}}_{H^{-\alpha_N}(\partial B_{R_{\Omega}}(0))} \\
&\lesssim& 
  \norm{W\widehat{\varphi}}_{H^{-1-\alpha_N}(\Gamma)}+  \norm{\widehat{\varphi}}_{H^{-\alpha_N}(\Gamma)}.
\end{eqnarray*}
The mean value can be estimated with $r^2 = \abs{x}^2$, the observation $\Delta r^2 = 4$, and integration by parts by 
\begin{eqnarray*}
 \abs{\overline{\widetilde{K}\widehat{\varphi}}} &\lesssim& 
 \abs{\skp{\widetilde{K}\widehat{\varphi},\Delta r^2}} \lesssim 
  \abs{\skp{\gamma_0^{\rm int}\widetilde{K}\widehat{\varphi},\partial_n r^2}}+  
  \abs{\skp{\gamma_1^{\rm int}\widetilde{K}\widehat{\varphi}, r^2}}\\
   &\lesssim& \abs{\skp{\widehat{\varphi},(K'-1/2)\partial_n r^2}}+  
 \abs{\skp{W\widehat{\varphi}, r^2}} \\
  &\lesssim& \norm{\widehat{\varphi}}_{H^{-\alpha_N}(\Gamma)}+  
\norm{W\widehat{\varphi}}_{H^{-1-\alpha_N}(\Gamma)},
  \end{eqnarray*}
  where the last step follows since $K'$ is a bounded operator mapping 
  $H^{\alpha_N}(\Gamma)\rightarrow H^{\alpha_N}(\Gamma)$ from Lemma~\ref{lem:potentialreg}. 
The additional mapping properties of $W$ of Lemma~\ref{lem:potentialregK}, (\ref{item:lem:potentialregK-iii}) and 
inserting this in \eqref{eq:commutatorvolregW} leads together with a facewise trace estimate to
\begin{equation*}
\norm{\partial_n v}_{H^{\alpha_N}(\Gamma)}\lesssim\norm{v}_{H^{3/2+\alpha_N}(B_{R_{\Omega}}(0)\backslash\Gamma)} 
\lesssim \norm{\widehat{\varphi}}_{H^{-\alpha_N}(\Gamma)}.
\end{equation*}
Now, the computation 
$$
\mathcal{C}^{\eta}_{\eta}\widehat{\varphi} = \partial_n \widetilde{\mathcal{C}}^{\eta}_{\eta}\widehat{\varphi}
+K'((\partial_n \eta) \eta \widehat{\varphi})-\eta K'((\partial_n\eta)\widehat{\varphi})+
2(\partial_n \eta) \gamma_0 \widetilde{\mathcal{C}}_{\eta}(\widehat{\varphi})-
(\partial_n \eta)V((\partial_n \eta)\widehat{\varphi}),
$$
the mapping properties of $V$ and the commutator of $K'$ (as normal trace of the commutator 
$\widetilde{\mathcal{C}}_{\eta}$ from Lemma~\ref{lem:commutator}, c.f. \eqref{eq:commutvolest})
prove the lemma.
\end{proof}

\subsection{Symm's integral equation (proof of Theorem~\ref{th:localSLP})}
The main tools in our proofs are the Galerkin orthogonality 
\begin{equation}\label{eq:Galerkinortho}
\skp{V(\phi-\phi_h),\psi_h} = 0 \quad \forall \psi_h \in S^{0,0}(\mathcal{T}_h),
\end{equation}
and a Caccioppoli-type estimate for discrete harmonic functions that satisfy the orthogonality 
\begin{equation}\label{eq:discreteharmonic}
\skp{\gamma_0 v,\psi_h} =0 \quad \forall \psi_h \in S^{0,0}(\mathcal{T}_h), \operatorname*{supp}\psi_h \subset D\cap\Gamma.
\end{equation}
More precisely, the space of \emph{discrete harmonic functions} on an open set $D\subset \mathbb{R}^d$ is defined as
\begin{align}
{\mathcal H}_{h}(D)&:=\{v\in H^1(D\backslash\Gamma) \colon \text{$v$ is harmonic on}\; D\backslash\Gamma,\nonumber\\
&\quad\quad\exists \widetilde{v} \in S^{0,0}({\mathcal T}_h) \;\mbox{s.t.} \; 
[\partial_{n}v]|_{D \cap \Gamma} = \widetilde{v}|_{D \cap \Gamma}, \; \text{$v$ satisfies}\; 
\eqref{eq:discreteharmonic}\}.
\end{align}

\begin{proposition}{\cite[Lemma 3.9]{FMPBEM}}\label{prop:Cacc}
For discrete harmonic functions $u \in {\mathcal H}_{h}(B')$, the interior regularity estimate 
\begin{equation}\label{eq:Caccioppoli}
\norm{\nabla u}_{L^2(B)} \lesssim \frac{h}{\widehat{d}}\norm{\nabla u}_{L^2(B')} + 
\frac{1}{\widehat{d}}\norm{u}_{L^2(B')}
\end{equation}
holds, where $B$ and $B'$ are nested boxes and $\widehat{d}:=\operatorname*{dist}(B,\partial B')>0$ satisfies 
$8h\leq \widehat{d}$. 
The hidden constant depends only on $\Omega, d$, and the $\gamma$-shape regularity of $\mathcal{T}_h$.
\end{proposition}

As a consequence of this interior regularity estimate and Lemma~\ref{lem:shiftapriori}, we 
get an estimate for the jump of the normal derivative of a discrete harmonic potential.

\begin{lemma}\label{lem:normaltraceest}
Let Assumption~\ref{ass:shift} hold and $B \subset B' \subset B_{R_{\Omega}}(0)$ be nested boxes 
with $\widehat{d}:=\operatorname*{dist}(B,\partial B')>0$ and $h$ be sufficiently small
so that the assumption of Proposition~\ref{prop:Cacc} holds. Let
$u:=\widetilde{V}\zeta_h$ with 
$\zeta_h \in S^{0,0}(\mathcal{T}_h)$ and assume 
$u \in {\mathcal H}_h(B')$. 
Let $\widehat{\Gamma}\subset B\cap\Gamma$ and $\eta \in C^{\infty}_0(\mathbb{R}^d)$ be an arbitrary 
cut-off function satisfying $\eta \equiv 1$ on $B'$. Then,
\begin{eqnarray}
\norm{[\partial_n u]}_{L^2(\widehat{\Gamma})} \leq C \left(
h^{\alpha_D/(1+2\alpha_D)}\norm{\eta\zeta_h}_{L^2(\Gamma)} + h^{-1}\norm{\eta V\zeta_h}_{H^{-\alpha_D}(\Gamma)}+
\norm{\zeta_h}_{H^{-1/2}(\Gamma)}\right).
\end{eqnarray}
The constant $C>0$ depends only on $\Omega,d,\widehat{d}$, the $\gamma$-shape regularity of $\mathcal{T}_h$, 
$\|\eta\|_{W^{1,\infty}(\mathbb{R}^d)}$, and the constants appearing in Assumption~\ref{ass:shift}.
\end{lemma}
\begin{proof} 
We split the function $u = u_{\text{far}} + u_{\text{near}} $, where 
the near field $u_{\text{near}}$ and the far field $u_{\text{far}}$ solve the Dirichlet problems
\begin{alignat*}{2}
 -\Delta u_{\text{near}} &= 0 \qquad \text{in} \; B_{R_{\Omega}}(0)\backslash\Gamma, &\qquad   
\gamma_0 u_{\text{near}} &= \eta V\zeta_h \qquad \quad \;\; \text{on}\, \Gamma\cup\partial B_{R_{\Omega}}(0), \\
-\Delta u_{\text{far}} &= 0 \qquad \text{in} \; B_{R_{\Omega}}(0)\backslash\Gamma, & \qquad 
\gamma_0u_{\text{far}} &= (1-\eta) V\zeta_h \quad \text{on}\, \Gamma \cup \partial B_{R_{\Omega}}(0).
\end{alignat*}
We first consider $\gamma_1^{\rm int} u_{\rm near}$ - the case $\gamma_1^{\rm ext} u_{\rm near}$ is treated analogously.

Let $\widehat{\eta}$ be another cut-off function satisfying $\widehat{\eta} \equiv 1$ on $\widehat{\Gamma}$ and
$\operatorname*{supp} \widehat{\eta} \subset B$.
The multiplicative trace inequality, see, e.g., \cite[Thm. A.2]{Melenk}, implies for any $\epsilon \leq 1/2$ that
\begin{eqnarray}
\nonumber 
 \norm{\gamma_1^{\rm int} u_{\text{near}}}_{L^2(\widehat{\Gamma})} &\lesssim&  
\norm{\gamma_1^{\rm int}(\widehat{\eta} u_{\text{near}})}_{L^2(B\cap\Gamma)}
\lesssim \norm{\nabla(\widehat{\eta} u_{\text{near}})}_{L^2(\Omega)}^{2\epsilon/(1+2\epsilon)}
\norm{\nabla(\widehat{\eta} u_{\text{near}})}_{H^{1/2+\epsilon}(\Omega)}^{1/(1+2\epsilon)}\\
\label{eq:lem:normaltraceest-10}
&\lesssim& \norm{\nabla(\widehat\eta u_{\text{near}})}_{L^2(B)}^{2\epsilon/(1+2\epsilon)}
\norm{\widehat{\eta}u_{\text{near}}}_{H^{3/2+\epsilon}(B)}^{1/(1+2\epsilon)}.
\end{eqnarray}
Since $u_{\text{near}} \in {\mathcal H}_h(B')$, we use the interior regularity estimate \eqref{eq:discreteharmonic}
for the first term on the right-hand side of 
(\ref{eq:lem:normaltraceest-10}), 
and 
the second term of (\ref{eq:lem:normaltraceest-10}) can be estimated using (\ref{eq:lem:shiftapriori-20}) 
of Lemma~\ref{lem:shiftapriori}. In total,  we get for $\epsilon \leq \alpha_D<1/2$ that
\begin{align}
\nonumber 
 &\norm{\nabla(\widehat{\eta}u_{\text{near}})}_{L^2(B)}^{2\epsilon/(1+2\epsilon)}
\norm{\widehat{\eta}u_{\text{near}}}_{H^{3/2+\epsilon}(B)}^{1/(1+2\epsilon)} \\ 
\nonumber 
& \qquad \lesssim 
\left(h\norm{\nabla u_{\text{near}}}_{L^2(B')}+ \norm{u_{\text{near}}}_{L^2(B')}\right)^{2\epsilon/(1+2\epsilon)}\nonumber 
\cdot\left(\norm{u_{\text{near}}}_{H^1(B')}+\norm{\eta V\zeta_h}_{H^{1+\epsilon}(\Gamma)} \right)^{1/(1+2\epsilon)} \nonumber\\
&\qquad\lesssim h^{2\epsilon/(1+2\epsilon)}\norm{u_{\text{near}}}_{H^1(B')} + 
\norm{u_{\text{near}}}_{L^2(B')}^{2\epsilon/(1+2\epsilon)}\norm{u_{\text{near}}}_{H^1(B')}^{1/(1+2\epsilon)}
 \nonumber \\
\nonumber 
&\qquad\quad  + \norm{u_{\text{near}}}_{L^2(B')}^{2\epsilon/(1+2\epsilon)}
\norm{\eta V\zeta_h}_{H^{1+\epsilon}(\Gamma)}^{1/(1+2\epsilon)} + h^{2\epsilon/(1+2\epsilon)}\norm{\nabla u_{\text{near}}}_{L^2(B')}^{2\epsilon/(1+2\epsilon)}
\norm{\eta V\zeta_h}_{H^{1+\epsilon}(\Gamma)}^{1/(1+2\epsilon)} \\
\label{eq:nearfieldtemp}
& \qquad =: T_1 + T_2 + T_3 + T_4. 
\end{align}
Let $\mathcal{I}_h$ be the nodal interpolation operator.
The mapping properties of $V$ from Lemma~\ref{lem:potentialreg},
(\ref{item:lem:potentialreg-ii}), the commutator $C_{\eta}$ from \eqref{eq:commutator} 
as well as an inverse inequality, see, e.g., \cite[Thm. 3.2]{GHS}, lead to
\begin{eqnarray}
\label{eq:lem:normaltraceest-100}
\norm{\eta V\zeta_h}_{H^{1+\epsilon}(\Gamma)} &\lesssim& 
\norm{V(\eta \zeta_h)}_{H^{1+\epsilon}(\Gamma)} + \norm{C_{\eta} \zeta_h}_{H^{1+\epsilon}(\Gamma)} 
\lesssim \norm{\eta \zeta_h}_{H^{\epsilon}(\Gamma)} + \norm{\zeta_h}_{H^{-1+\epsilon}(\Gamma)} \nonumber \\
&\lesssim& \norm{\mathcal{I}_h(\eta) \zeta_h}_{H^{\epsilon}(\Gamma)} + 
\norm{(\eta - \mathcal{I}_h\eta)\zeta_h}_{H^{\epsilon}(\Gamma)} +\norm{\zeta_h}_{H^{-1+\epsilon}(\Gamma)} \nonumber \\
&\lesssim& h^{-\epsilon}\norm{\mathcal{I}_h(\eta) \zeta_h}_{L^{2}(\Gamma)} + 
h\norm{\zeta_h}_{H^{\epsilon}(\Gamma)} + h^{1-\epsilon}\norm{\zeta_h}_{L^2(\Gamma)} +\norm{\zeta_h}_{H^{-1+\epsilon}(\Gamma)} \nonumber \\
&\lesssim& h^{-\epsilon}\left(\norm{(\eta - \mathcal{I}_h \eta) \zeta_h}_{L^{2}(\Gamma)} + \norm{\eta\zeta_h}_{L^{2}(\Gamma)}\right)
+\norm{\zeta_h}_{H^{-1+\epsilon}(\Gamma)}\nonumber \\
&\lesssim& h^{-\epsilon}\left( \norm{\eta\zeta_h}_{L^2(\Gamma)} + \norm{\zeta_h}_{H^{-1}(\Gamma)} \right).
\end{eqnarray}
With the classical {\sl a priori} estimate for the inhomogeneous Dirichlet problem in the $H^1$-norm, the commutator 
$C_{\eta}$, and Lemma~\ref{lem:commutator}, we estimate 
\begin{align}
\label{eq:aprioriH1}
T_1  &= h^{2\varepsilon/(1+\varepsilon)} \norm{u_{\text{near}}}_{H^1(B')} \lesssim 
h^{2\varepsilon/(1+\varepsilon)} \norm{\eta V\zeta_h}_{H^{1/2}(\Gamma)} \nonumber \\ &\lesssim 
h^{2\varepsilon/(1+\varepsilon)} \left(\norm{V(\eta\zeta_h)}_{H^{1/2}(\Gamma)}+\norm{C_{\eta}\zeta_h}_{H^{1/2}(\Gamma)} \right)
\nonumber \\ &\lesssim
h^{2\varepsilon/(1+\varepsilon)}\left( 
\norm{\eta\zeta_h}_{L^{2}(\Gamma)} + \norm{\zeta_h}_{H^{-1-\alpha_D}(\Gamma)}\right), \\
\nonumber 
T_4 &= h^{2\epsilon/(1+2\epsilon)}\norm{\nabla u_{\text{near}}}_{L^2(B')}^{2\epsilon/(1+2\epsilon)}
\norm{\eta V\zeta_h}_{H^{1+\epsilon}(\Gamma)}^{1/(1+2\epsilon)}\nonumber \\
&\stackrel{(\ref{eq:lem:normaltraceest-100})}{\lesssim} h^{2\epsilon/(1+2\epsilon)}
\norm{u_{\text{near}}}_{H^1(B')}^{2\epsilon/(1+2\epsilon)} 
                                   \left(h^{-\epsilon}\norm{\eta\zeta_h}_{L^2(\Gamma)} + 
                                   h^{-\epsilon}\norm{\zeta_h}_{H^{-1}(\Gamma)}\right)^{1/(1+2\epsilon)}\nonumber \\
\label{eq:lem:normaltraceest-200}
&\stackrel{(\ref{eq:aprioriH1})}{\lesssim} h^{\epsilon/(1+2\epsilon)}\left(\norm{\eta\zeta_h}_{L^2(\Gamma)} 
+\norm{\zeta_h}_{H^{-1}(\Gamma)}\right).
\end{align}
We apply (\ref{eq:lem:shiftapriori-10}), 
(since $\eta \equiv 0$ on $\partial B_{R_{\Omega}}(0)$ only the boundary terms for $\Gamma$ appear) 
together with Young's inequality $ab \leq a^p/p + b^q/q$ applied with $p=(1+2\epsilon)/2\epsilon$, 
$q=1+2\epsilon$ to obtain
\begin{align*}
T_3 & = 
\norm{u_{\text{near}}}_{L^2(B')}^{2\epsilon/(1+2\epsilon)}
\norm{\eta V\zeta_h}_{H^{1+\epsilon}(\Gamma)}^{1/(1+2\epsilon)} 
\stackrel{(\ref{eq:lem:shiftapriori-10}), (\ref{eq:lem:normaltraceest-100})}{\lesssim} 
\norm{\eta V\zeta_h}_{H^{-\alpha_D}(\Gamma)}^{2\epsilon/(1+2\epsilon)}
\left(h^{-\epsilon}\norm{\eta\zeta_h}_{L^{2}(\Gamma)}+h^{-\epsilon}\norm{\zeta_h}_{H^{-1}(\Gamma)}\right)^{1/(1+2\epsilon)} \\
&\lesssim h^{-1}\norm{\eta V\zeta_h}_{H^{-\alpha_D}(\Gamma)} + h^{\epsilon}\norm{\eta\zeta_h}_{L^{2}(\Gamma)} 
+h^{\epsilon}\norm{\zeta_h}_{H^{-1}(\Gamma)}.
\end{align*}
Similarly, we get for the second term  in \eqref{eq:nearfieldtemp}
\begin{align*}
T_2 &= 
\norm{u_{\text{near}}}_{L^2(B')}^{2\epsilon/(1+2\epsilon)}
\norm{u_{\text{near}}}_{H^1(B')}^{1/(1+2\epsilon)} \stackrel{(\ref{eq:lem:shiftapriori-10})}{\lesssim} h^{-2\epsilon/(1+2\epsilon)}
\norm{\eta V\zeta_h}_{H^{-\alpha_D}(\Gamma)}^{2\epsilon/(1+2\epsilon)}
h^{2\epsilon/(1+2\epsilon)}\norm{u_{\text{near}}}_{H^1(B')}^{1/(1+2\epsilon)} \\
&\stackrel{(\ref{eq:aprioriH1})}{\lesssim} h^{-1}\norm{\eta V\zeta_h}_{H^{-\alpha_D}(\Gamma)} + 
h^{2\epsilon}\norm{\eta\zeta_h}_{L^{2}(\Gamma)} + h^{2\epsilon}\norm{\zeta_h}_{H^{-1-\alpha_D}(\Gamma)}.
\end{align*}
Inserting everything in \eqref{eq:nearfieldtemp} and using $h\lesssim 1$
gives 
\begin{eqnarray*}
\norm{\partial_n u_{\text{near}}}_{L^2(\widehat{\Gamma})} &\lesssim& 
h^{\epsilon/(1+2\epsilon)}\left(\norm{\eta \zeta_h}_{L^2(\Gamma)}+ \norm{\zeta_h}_{H^{-1}(\Gamma)}\right)+
h^{-1}\norm{\eta V\zeta_h}_{H^{-\alpha_D}(\Gamma)}. 
\end{eqnarray*}
Applying the same argument for the exterior Dirichlet boundary value problem leads to an estimate for the jump
of the normal derivative
\begin{eqnarray*}
\norm{[\partial_n u_{\text{near}}]}_{L^2(\widehat{\Gamma})}
\lesssim h^{\epsilon/(1+2\epsilon)}\left(\norm{\eta \zeta_h}_{L^2(\Gamma)}+ 
\norm{\zeta_h}_{H^{-1}(\Gamma)}\right)+
h^{-1}\norm{\eta V\zeta_h}_{H^{-\alpha_D}(\Gamma)}. 
\end{eqnarray*}
It remains to estimate the far field $u_{\text{far}}$, which can be treated similarly to the near field using 
a trace estimate and Lemma~\ref{lem:shiftapriori}. Applying Lemma~\ref{lem:shiftapriori} with a cut-off function 
$\widetilde{\eta}$ satisfying $\widetilde{\eta} \equiv 1$ on $B$ and $\operatorname*{supp} \widetilde{\eta} \subset B'$
the boundary 
term in (\ref{eq:lem:shiftapriori-20}) disappears since $\widetilde{\eta}(1-\eta)\equiv 0$ , which simplifies the arguments: 
\begin{eqnarray*}
 \norm{[\partial_n u_{\text{far}}]}_{L^2(\widehat{\Gamma})} &\leq&  \norm{[\partial_n(\widehat{\eta}u_{\text{far}})]}_{L^2(\widehat{\Gamma})}
 \lesssim \norm{u_{\text{far}}}_{H^{3/2+\epsilon}(B)}\\
&\stackrel{(\ref{eq:lem:shiftapriori-20})}{\lesssim}& \norm{u_{\text{far}}}_{H^{1}(B')}+ 
\norm{\widetilde{\eta}(1-\eta)V\zeta_h}_{H^{1+\epsilon}(\Gamma)}
=\norm{u_{\text{far}}}_{H^{1}(B')} \\ &\lesssim& \norm{(1-\eta)V\zeta_h}_{H^{1/2}(\Gamma\cup\partial B_{R_{\Omega}}(0))}
\lesssim\norm{\zeta_h}_{H^{-1/2}(\Gamma)},
\end{eqnarray*}
which proves the lemma.
\end{proof}

We use the Galerkin projection $\Pi : H^{-1/2}(\Gamma) \rightarrow S^{0,0}(\mathcal{T}_h)$,
which is, for any $\widehat{\phi} \in H^{-1/2}(\Gamma)$, defined by
\begin{equation}\label{eq:Galerkinprojection}
\skp{V(\widehat{\phi} - \Pi\widehat{\phi}),\psi_h} = 0 \quad \forall \psi_h \in S^{0,0}(\mathcal{T}_h).
\end{equation}
We denote by $I_h$ the $L^2(\Gamma)$-orthogonal projection given by 
\begin{equation*}
 \skp{I_h u ,v_h}_{L^2(\Gamma)} = \skp{u,v_h}_{L^2(\Gamma)} \quad \forall v_h \in S^{0,0}(\mathcal{T}_h).
\end{equation*}
This operator has the following super-approximation property, \cite{nitsche-schatz74}:  
For any discrete function $\psi_h \in S^{0,0}(\mathcal{T}_h)$ and a cut-off function
$\eta$ we have (with implied constants depending on $\|\eta\|_{W^{1,\infty}}$)
\begin{equation}\label{eq:superapprox}
\norm{\eta\psi_h - I_h(\eta \psi_h)}_{L^2(\Gamma)}^2 \lesssim h^2 
\sum_{T\in\mathcal{T}_h}\norm{\nabla(\eta \psi_h)}_{L^2(T)}^2 \lesssim h^2 \norm{\psi_h}_{L^2(\operatorname*{supp} \eta)}^2.
\end{equation}
 
The following lemma provides an estimate for the local Galerkin error and includes the 
key steps to the proof of Theorem~\ref{th:localSLP}.

\begin{lemma}\label{th:localSLP2}
Let the assumptions of Theorem~\ref{th:localSLP} hold. 
Let $\widehat{\Gamma_0}$ be an open subset of $\Gamma$ 
with $\Gamma_0\subset \widehat{\Gamma_0} \subsetneq \Gamma$
and $R:=\operatorname*{dist}(\Gamma_0,\partial\widehat{\Gamma_0}) > 0$. 
Let $h$ be sufficiently small such that at least
$\frac{h}{R}\leq \frac{1}{12}$. 
Assume that $\phi\in L^{2}(\widehat{\Gamma_0})$. Then,
we have
\begin{eqnarray*}
\norm{\phi-\phi_h}_{L^{2}(\Gamma_0)} &\leq& C \Big(  \inf_{\chi_h\in S^{0,0}(\mathcal{T}_h)}\norm{\phi - \chi_h}_{L^{2}(\widehat{\Gamma_0})}  
+\norm{\phi-\phi_h}_{H^{-1/2}(\widehat{\Gamma_0})} \\
& & \quad + h^{\alpha_D/(1+2\alpha_D)}\norm{\phi-\phi_h}_{L^{2}(\widehat{\Gamma_0})}+ \norm{\phi-\phi_h}_{H^{-1-\alpha_D}(\Gamma)}\Big).
\end{eqnarray*} 
The constant $C>0$ depends only on $\Gamma,\Gamma_0,d,R,$ and the $\gamma$-shape regularity of $\mathcal{T}_h$.
\end{lemma}

\begin{proof}
We define $e:= \phi-\phi_h$, open subsets 
$\Gamma_0\subset\Gamma_1\subset \Gamma_2 \dots \subset \Gamma_5 \subset \widehat{\Gamma_0}$,
and volume boxes $B_0 \subset B_1 \subset B_2 \dots \subset B_5 \subset \mathbb{R}^d$, where $B_i\cap\Gamma = \Gamma_i$. 
Throughout the proof, we use multiple cut-off functions $\eta_i \in C_0^{\infty}(\mathbb{R}^d)$, $i=1,\dots,5$. 
These smooth functions $\eta_i$ should satisfy
$\eta_i \equiv 1$ on $\Gamma_{i-1}$, $\operatorname*{supp}\eta \subset B_i$ and 
$\norm{\nabla \eta_i}_{L^{\infty}(B_i)}\lesssim \frac{1}{R}$.
We write
\begin{equation}\label{eq:tmpstart}
\norm{e}_{L^2(\Gamma_0)}^2\leq \norm{\eta_1 e}_{L^2(\Gamma)}^2 = \skp{\eta_1 e, \eta_1 e} =
\skp{e,\eta_1^2 e}.  
\end{equation}
With the Galerkin projection $\Pi$ from \eqref{eq:Galerkinprojection}, we obtain
\begin{equation}\label{eq:tmpstart2}
\norm{\eta_1 e}^2_{L^2(\Gamma)}= 
\skp{e,\eta_1^2 e}  = \skp{\eta_5e,\eta_1^2 e}   = \skp{\Pi(\eta_5 e),\eta_1^2 e}  
+ \skp{\eta_5 e-\Pi(\eta_5 e),\eta_1^2 e}.
\end{equation}
With an inverse inequality and the $L^2$-orthogonal projection $I_h$,
which satisfies the super-approximation property \eqref{eq:superapprox} for $\eta_5 \phi_h$, we get
\begin{align}
\nonumber 
& \norm{\eta_5 \phi_h - \Pi(\eta_5\phi_h)}_{L^2(\Gamma)}  
\lesssim \norm{\eta_5 \phi_h- I_h(\eta_5\phi_h)}_{L^2(\Gamma)} + \norm{I_h(\eta_5 \phi_h)- \Pi(\eta_5\phi_h)}_{L^2(\Gamma)}\\
\nonumber 
&\qquad\lesssim h\norm{\phi_h}_{L^2(\widehat{\Gamma_0})} + h^{-1/2}\norm{I_h(\eta_5 \phi_h)- \Pi(\eta_5\phi_h)}_{H^{-1/2}(\Gamma)}\\
\nonumber 
&\qquad\lesssim h\norm{\phi_h}_{L^2(\widehat{\Gamma_0})} + h^{-1/2}\norm{I_h(\eta_5 \phi_h)- \eta_5\phi_h}_{H^{-1/2}(\Gamma)}
  + 
h^{-1/2}\norm{\eta_5 \phi_h- \Pi(\eta_5\phi_h)}_{H^{-1/2}(\Gamma)} \\ 
\label{eq:th:localSLP-10}
&\qquad\lesssim h\norm{\phi_h}_{L^2(\widehat{\Gamma_0})},
\end{align}
where the last estimate follows from C\'ea's lemma and super-approximation. 
The same argument leads to
\begin{eqnarray}
\nonumber 
\norm{\eta_5 \phi- \Pi(\eta_5\phi)}_{L^2(\Gamma)}
&\lesssim& \norm{\eta_5 \phi- I_h(\eta_5\phi)}_{L^2(\Gamma)} + \norm{I_h(\eta_5 \phi)- \Pi(\eta_5\phi)}_{L^2(\Gamma)}\\
\nonumber 
&\lesssim& \norm{\eta_5 \phi}_{L^2(\Gamma)} + h^{-1/2}\norm{I_h(\eta_5 \phi)- \Pi(\eta_5\phi)}_{H^{-1/2}(\Gamma)}\\
&\lesssim& \norm{\eta_5 \phi}_{L^2(\Gamma)}.
\label{eq:th:localSLP-20}
\end{eqnarray}
In fact, this argument shows $L^2$-stability of $\Pi$: 
\begin{equation}
\label{eq:th:localSLP-30}
\|\Pi \psi\|_{L^2(\Gamma)} \lesssim \|\psi\|_{L^2(\Gamma)} \qquad \forall \psi \in L^2(\Gamma). 
\end{equation}
The bounds (\ref{eq:th:localSLP-10}), (\ref{eq:th:localSLP-20}) together imply 
\begin{eqnarray}\label{eq:estGalerkinproj}
\abs{\skp{\eta_5 e-\Pi(\eta_5 e),\eta_1^2 e}} &\leq& 
\norm{\eta_1^2 e}_{L^2(\Gamma)}\left( \norm{\eta_5 \phi-\Pi(\eta_5 \phi)}_{L^2(\Gamma)} + \norm{\eta_5 \phi_h-\Pi(\eta_5 \phi_h)}_{L^2(\Gamma)}  \right)
\nonumber\\
&\lesssim& \norm{\eta_1 e}_{L^2(\Gamma)}\left( \norm{\eta_5 \phi}_{L^2(\Gamma)} + h\norm{\phi_h}_{L^2(\widehat{\Gamma_0})}\right)
\nonumber\\
&\lesssim& \norm{\eta_1 e}_{L^2(\Gamma)}\left( (1+h)\norm{\phi}_{L^2(\widehat{\Gamma_0})} + h\norm{e}_{L^2(\widehat{\Gamma_0})}\right).  
\end{eqnarray}
For the first term on the right-hand side of
\eqref{eq:tmpstart2}, we want to use Lemma~\ref{lem:normaltraceest}.
Since $[\partial_n\widetilde{V}\zeta_h] = -\zeta_h \in S^{0,0}(\mathcal{T}_h)$ 
for any discrete function $\zeta_h \in S^{0,0}(\mathcal{T}_h)$, 
we need to construct a discrete function satisfying the orthogonality condition \eqref{eq:discreteharmonic}.
Using the Galerkin orthogonality with test functions $\psi_h$ with support $\operatorname*{supp}\psi_h\subset\Gamma_4$ and noting that
$\eta_5 \equiv 1$ on $\operatorname*{supp} \psi_h$, we obtain with the commutator $C_{\eta_5}$ defined in \eqref{eq:commutator}
\begin{eqnarray}
0&=&\skp{Ve,\eta_5\psi_h} = 
\skp{\eta_5Ve,\psi_h} = 
\skp{V(\eta_5e)-C_{\eta_5}e,\psi_h} \nonumber\\ &=& 
\skp{V(\eta_5e)-\eta_5C_{\eta_5}e,\psi_h} =
\skp{V(\eta_5e-V^{-1}(\eta_5C_{\eta_5}e)),\psi_h} \nonumber  \\ &=&
\skp{V(\Pi(\eta_5e)-\Pi(V^{-1}(\eta_5C_{\eta_5}e))),\psi_h}.
\end{eqnarray}
Thus, defining 
\begin{equation}
\label{eq:th:localSLP-100}
\zeta_h := \Pi(\eta_5e) - \xi_h \quad \text{with} \quad\xi_h:=\Pi(V^{-1}(\eta_5C_{\eta_5}e)),
\end{equation}
we get 
on the volume box $B_4\subset \mathbb{R}^d$
a discrete harmonic function
$$u:=\widetilde{V}\zeta_h \in {\mathcal H}_h(B_4).$$ 
The correction $\xi_h$ can be estimated using the $L^2$-stability (\ref{eq:th:localSLP-30})
of the Galerkin projection, the mapping properties of $V^{-1}$, 
$C_{\eta_5}$, and the commutator $C_{\eta_5}^{\eta_5}$ from Lemma~\ref{lem:commutator} by 
\begin{eqnarray}
\norm{\xi_h}_{L^2(\Gamma)}&=&
\norm{\Pi (V^{-1}(\eta_5C_{\eta_5}e))}_{L^2(\Gamma)} \lesssim \norm{V^{-1}(\eta_5C_{\eta_5}e)}_{L^2(\Gamma)}
\lesssim \norm{\eta_5C_{\eta_5}e}_{H^{1}(\Gamma)}\nonumber\\
\label{eq:estimatecorrection}
&\lesssim& \norm{C_{\eta_5}(\eta_5e)}_{H^{1}(\Gamma)}+\norm{C_{\eta_5}^{\eta_5}e}_{H^{1}(\Gamma)}
\lesssim\norm{\eta_5e}_{H^{-1}(\Gamma)}+\norm{e}_{H^{-1-\alpha_D}(\Gamma)}.
\end{eqnarray}
We write
\begin{equation}\label{eq:tmpstart3}
\skp{\Pi(\eta_5 e),\eta_1^2  e} = \skp{\Pi(\eta_5 e) -\xi_h,\eta_1^2  e} + \skp{\xi_h,\eta_1^2  e}
 = \skp{\zeta_h,\eta_1^2 e} + \skp{\xi_h,\eta_1^2e}. 
\end{equation}
For the second term in \eqref{eq:tmpstart3} we use
\begin{eqnarray}\label{eq:tempest1}
\abs{\skp{\xi_h,\eta_1^2  e}} \leq 
\norm{\xi_h}_{L^2(\Gamma)}\norm{\eta_1^2  e}_{L^2(\Gamma)}
\stackrel{(\ref{eq:estimatecorrection})}{ \lesssim } 
\left(\norm{\eta_5e}_{H^{-1}(\Gamma)}+\norm{e}_{H^{-1-\alpha_D}(\Gamma)}\right)\norm{\eta_1 e}_{L^2(\Gamma)}.
\end{eqnarray}
We treat the first term in \eqref{eq:tmpstart3} as follows:
We apply Lemma~\ref{lem:normaltraceest} with the boxes $B_2$ and $B_3$ - since we assumed $h\leq 12R$, the condition
$8h\leq \operatorname*{dist}(B_2,\partial B_3)$ can be fulfilled - to the discrete harmonic function
$\widetilde{V}\zeta_h \in {\mathcal H}_h(B_4)$ and the cut-off function $\eta_4$. The jump condition 
$[\partial_n u] = -\zeta_h$
leads to 
\begin{eqnarray}\label{eq:zetahestimate}
\norm{\zeta_h}_{L^2(\operatorname*{supp} \eta_1)} &\leq& \norm{[\partial_n u]}_{L^2(\Gamma_1)} \nonumber \\ &\lesssim& 
h^{\alpha_D/(1+2\alpha_D)}\norm{\eta_4\zeta_h}_{L^2(\Gamma)} + h^{-1}\norm{\eta_4 V\zeta_h}_{H^{-\alpha_D}(\Gamma)}+
\norm{\zeta_h}_{H^{-1/2}(\Gamma)}.
\end{eqnarray}
The definition of $\zeta_h$, the bound \eqref{eq:estimatecorrection}, and the $H^{-1/2}$-stability of the Galerkin projection lead to
\begin{eqnarray}\label{eq:estzetah}
\norm{\zeta_h}_{H^{-1/2}(\Gamma)} &\lesssim& 
\norm{\eta_5e}_{H^{-1/2}(\Gamma)} + \norm{\xi_h}_{H^{-1/2}(\Gamma)}\nonumber \\
&\lesssim& \norm{\eta_5e}_{H^{-1/2}(\Gamma)}+
\norm{e}_{H^{-1-\alpha_D}(\Gamma)}.
\end{eqnarray}
With the $L^2$-stability (\ref{eq:th:localSLP-30}) of the Galerkin projection and (\ref{eq:estimatecorrection})
we get
\begin{eqnarray}\label{eq:estzetah2}
\norm{\zeta_h}_{L^{2}(\Gamma)} &\lesssim& 
\norm{\eta_5e}_{L^{2}(\Gamma)} + \norm{\xi_h}_{L^{2}(\Gamma)}\nonumber \\
&\lesssim&
\norm{\eta_5e}_{L^{2}(\Gamma)} + \norm{e}_{H^{-1-\alpha_D}(\Gamma)}.
\end{eqnarray}
We use the orthogonality of $\zeta_h$ on $\Gamma_4$ expressed in (\ref{eq:discreteharmonic})
and the $L^2$-orthogonal projection $I_h$ to estimate
\begin{align}\label{eq:estVzetah}
&\norm{\eta_4 V\zeta_h}_{H^{-\alpha_D}(\Gamma)} =
\sup_{w\in H^{\alpha_D}(\Gamma)}\frac{\skp{V\zeta_h,\eta_4 w}}{\norm{w}_{H^{\alpha_D}(\Gamma)}} 
= \sup_{w\in H^{\alpha_D}(\Gamma)}\frac{\skp{V\zeta_h,\eta_4w-I_h(\eta_4w)}}{\norm{w}_{H^{\alpha_D}(\Gamma)}}\nonumber\\ 
&\qquad\lesssim
\sup_{w\in H^{\alpha_D}(\Gamma)}\frac{\norm{\eta_5V\zeta_h}_{H^{1}(\Gamma)}
\norm{\eta_4w-I_h(\eta_4w)}_{H^{-1}(\Gamma)}}{\norm{w}_{H^{\alpha_D}(\Gamma)}} \nonumber 
\lesssim 
h^{1+\alpha_D}\left(\norm{\eta_5\zeta_h}_{L^{2}(\Gamma)}+\norm{C_{\eta_5}\zeta_h}_{H^1(\Gamma)}\right) \\ 
&\qquad\lesssim
h^{1+\alpha_D}\left(\norm{\eta_5\zeta_h}_{L^{2}(\Gamma)}+\norm{\zeta_h}_{H^{-1}(\Gamma)}\right).
\end{align}
Inserting \eqref{eq:estzetah}--\eqref{eq:estVzetah} in \eqref{eq:zetahestimate} and using $h\lesssim 1$,
we arrive at
\begin{eqnarray}\label{eq:estzetahlocal}
\norm{\zeta_h}_{L^2(\operatorname*{supp} \eta_1)} &\lesssim& \left(h^{\alpha_D/(1+2\alpha_D)}+h^{\alpha_D}\right)
\norm{\eta_5\zeta_h}_{L^2(\Gamma)}
+ \norm{\zeta_h}_{H^{-1/2}(\Gamma)}\nonumber \\
&\lesssim& h^{\alpha_D/(1+2\alpha_D)}\norm{\eta_5 e}_{L^2(\Gamma)}+\norm{\eta_5 e}_{H^{-1/2}(\Gamma)} + \norm{e}_{H^{-1-\alpha_D}(\Gamma)}.
\end{eqnarray}
Combining \eqref{eq:tmpstart2}, \eqref{eq:tmpstart3} with \eqref{eq:estGalerkinproj}, \eqref{eq:tempest1},
\eqref{eq:estzetahlocal}, and finally \eqref{eq:tmpstart}, we get
\begin{equation*}
\norm{\eta_1 e}_{L^2(\Gamma)}^2 \lesssim \left(\norm{\phi}_{L^2(\widehat{\Gamma_0})}+ \norm{e}_{H^{-1/2}(\widehat{\Gamma_0})}+
h^{\alpha_D/(1+2\alpha_D)}\norm{e}_{L^{2}(\widehat{\Gamma_0})}
+\norm{e}_{H^{-1-\alpha_D}(\Gamma)}\right)\norm{\eta_1 e}_{L^2(\Gamma)}.
\end{equation*}
Since we only used the Galerkin orthogonality as a property of the error $\phi-\phi_h$, we may write
$\phi-\phi_h = (\phi-\chi_h)+(\chi_h - \phi_h)$ for arbitrary $\chi_h \in S^{0,0}(\mathcal{T}_h)$ 
and we have proven the inequality claimed in Lemma~\ref{th:localSLP2}.
\end{proof}

In order to prove Theorem~\ref{th:localSLP}, we need a lemma:
\begin{lemma}
\label{lemma:approximation-in-H-1/2}
For every $\delta > 0$ there is a bounded linear operator 
$J_\delta:H^{-1}(\Gamma) \rightarrow L^{2}(\Gamma)$ with the following
properties:
\begin{enumerate}[(i)]
\item (stability): For every $-1 \leq s \leq t \leq 0$ there is $C_{s,t} > 0$ (depending only
on $s$, $t$, $\Omega$) such that 
$\|J_\delta u\|_{H^{t}(\Gamma)} \leq \delta^{s-t} C_s\|J_\delta u\|_{H^s(\Gamma)} $ 
for all $u \in H^s(\Gamma)$.
\item (locality): for $\omega \subset \Gamma$ the restriction $(J_\delta u)|_\omega$ depends only on 
$u|_{\omega_{\rho}}$ with
$\omega_\delta:= \cup_{x \in \omega} B_\delta(x) \cap \Gamma$. 
\item (approximation): 
For every $-1 \leq t \leq s \leq 1$ there  is $C_{s,t} > 0$ (depending only on $s$, $t$, $\Omega$) 
such that $\|u - J_\delta u\|_{H^{t}(\Gamma)} \leq C_{s,t} \delta^{s-t} \|u\|_{H^s(\Gamma)}$ 
for all $u \in H^s(\Gamma)$.  
\end{enumerate}
\end{lemma}
\begin{proof}
Operators with such properties are obtained by the usual mollification procedure (on a length scale $O(\delta)$ 
for domains in $\mathbb{R}^d$). This technique can be generalized to the present setting of surfaces with the aid 
of localization and charts. 
\end{proof}

Now, we can prove our main result, 
a local estimate for the Galerkin-boundary element error for Symm's integral equation in the $L^2$-norm.

\begin{proof}[of Theorem~\ref{th:localSLP}]
Starting with Lemma~\ref{th:localSLP2}, it remains to estimate the two terms
$h^{\alpha_D/(1+2\alpha_D)}\norm{e}_{L^{2}(\widehat{\Gamma_0})}$ and $\norm{e}_{H^{-1/2}(\widehat{\Gamma_0})}$, 
where $e:= \phi-\phi_h$.

We start with the latter.
Let 
$\widehat{\eta} \in C^{\infty}(\mathbb{R}^d)$ be a cut-off-function with $\widehat{\eta} \equiv 1$ on $\widehat{\Gamma_0}$, 
$\operatorname*{supp} \widehat{\eta} \subset B^{\widehat{\Gamma_0}}_{R/2}=\{x\in\mathbb{R}^d:\operatorname*{dist}(\{x\},\widehat{\Gamma_0})< \frac{R}{2}\}$ 
and $\norm{\nabla \widehat{\eta}}_{L^{\infty}}\lesssim \frac{1}{R}$. 
Let $\widetilde{\eta}$ be another cut-off function with $\widetilde{\eta}=1$ on 
$B^{\widehat{\Gamma_0}}_{R/2+h}$ and $\operatorname*{supp}{\widetilde{\eta}}\cap\Gamma\subset \widehat{\Gamma_1}$, where 
 $\operatorname*{dist}(\widehat{\Gamma_0},\partial{\widehat{\Gamma_{1}}})\geq R$.
Select $\delta = ch $ with a constant $c = O(1)$ such that the operator $J_{ch}$ of 
Lemma~\ref{lemma:approximation-in-H-1/2} has the support property 
$\operatorname*{supp} J_{ch} (\widehat \eta) \subset B^{\widehat{\Gamma_0}}_{R/2+h}$.
We will employ the operator $I_h \circ J_{ch}:H^{-1}(\Gamma) \rightarrow S^{0,0}(\Gamma)$
with the $L^2$-orthogonal projection $I_h$. It is easy
to see that we may assume that 
\begin{equation}
\label{eq:th:localSLP-200}
\operatorname*{supp} (I_h \circ J_{ch} (\widehat \eta)) \subset 
B^{\widehat{\Gamma_0}}_{R/2+h}. 
\end{equation}
Concerning the approximation properties, we have 
\begin{align}
\nonumber 
\|u - I_h \circ J_{ch} u\|_{H^{-1}(\Gamma)} & \leq 
\|u - J_{ch} u\|_{H^{-1}(\Gamma)} + 
\|J_{ch} u - I_h \circ J_{ch} u\|_{H^{-1}(\Gamma)} 
\\
\label{eq:th:localSLP-300} 
& \lesssim (ch)^{1/2} \|u\|_{H^{-1/2}(\Gamma)} + 
h \|J_{ch} u\|_{L^2(\Gamma)} 
 \lesssim h^{1/2} \|u\|_{H^{-1/2}(\Gamma)}. 
\end{align}
With the definition of the commutators $C_{\widehat{\eta}}$, $C_{\widehat{\eta}}^{\widehat{\eta}}$, the Galerkin orthogonality satisfied by $e$,
and the fact that $V:H^{-1/2}(\Gamma) \rightarrow H^{1/2}(\Gamma)$ is an isomorphism,
we get 
\begin{align}
\label{eq:dualitylocalH12}
&\norm{\widehat{\eta}e}_{H^{-1/2}(\Gamma)} = \sup_{w\in H^{1/2}(\Gamma)}
\frac{\skp{\widehat{\eta}e,w}}{\norm{w}_{H^{1/2}(\Gamma)}}
\lesssim
\sup_{\psi\in H^{-1/2}(\Gamma)}\frac{\skp{\widehat{\eta}e,V\psi}}{\norm{\psi}_{H^{-1/2}(\Gamma)}} \nonumber\\ &\qquad =  
\sup_{\psi\in H^{-1/2}(\Gamma)}\frac{\skp{Ve,\widehat{\eta}\psi}+\skp{C_{\widehat{\eta}}e,\psi}}{\norm{\psi}_{H^{-1/2}(\Gamma)}} \nonumber
=
\sup_{\psi\in H^{-1/2}(\Gamma)}\frac{\skp{Ve,\widehat{\eta}\psi-I_h\circ J_{ch} (\widehat{\eta}\psi)}+\skp{C_{\widehat{\eta}}e,\psi}}{\norm{\psi}_{H^{-1/2}(\Gamma)}}\nonumber 
\\&\qquad\stackrel{(\ref{eq:th:localSLP-200})}{=}
\sup_{\psi\in H^{-1/2}(\Gamma)}\frac{\skp{\widetilde{\eta}Ve,\widehat{\eta}\psi-I_h\circ J_{ch}(\widehat{\eta}\psi)}+
\skp{C_{\widehat{\eta}}e,\psi}}{\norm{\psi}_{H^{-1/2}(\Gamma)}} \nonumber
\\ \nonumber&\qquad=
\sup_{\psi\in H^{-1/2}(\Gamma)}\frac{\skp{V(\widetilde{\eta}e),\widehat{\eta}\psi-I_h\circ J_{ch}(\widehat{\eta}\psi)}-
\skp{C_{\widetilde{\eta}}e,\widehat{\eta}\psi-I_h\circ J_{ch}(\widehat{\eta}\psi)}
-\skp{e,C_{\widehat{\eta}}\psi}}{\norm{\psi}_{H^{-1/2}(\Gamma)}}\nonumber \\
&\qquad=
\sup_{\psi\in H^{-1/2}(\Gamma)}\frac{\skp{V(\widetilde{\eta}e),\widehat{\eta}\psi-I_h\circ J_{ch}(\widehat{\eta}\psi)}-
\skp{C_{\widetilde{\eta}}(\widetilde{\eta}e),\widehat{\eta}\psi-I_h\circ J_{ch}(\widehat{\eta}\psi)}}{\norm{\psi}_{H^{-1/2}(\Gamma)}}\nonumber \\
&\qquad \qquad\qquad
+\frac{\skp{C_{\widetilde{\eta}}^{\widetilde{\eta}}e,\widehat{\eta}\psi-I_h\circ J_{ch}(\widehat{\eta}\psi)}
-\skp{e,C_{\widehat{\eta}}\psi}}{\norm{\psi}_{H^{-1/2}(\Gamma)}}\nonumber \\
&\qquad\lesssim  \sup_{\psi\in H^{-1/2}(\Gamma)}\frac{\left(\norm{\widetilde{\eta}e}_{L^{2}(\Gamma)}+\norm{e}_{H^{-1-\alpha_D}(\Gamma)} \right)
\norm{\widehat{\eta}\psi-I_h\circ J_{ch}(\widehat{\eta}\psi)}_{H^{-1}(\Gamma)}+
\norm{e}_{H^{-1-\alpha_D}(\Gamma)}\norm{\psi}_{H^{-1+\alpha_D}(\Gamma)}}{\norm{\psi}_{H^{-1/2}(\Gamma)}}  \nonumber
\\ &\qquad\lesssim h^{1/2}\norm{e}_{L^{2}(\widehat{\Gamma_1})}+\norm{e}_{H^{-1-\alpha_D}(\Gamma)}.
\end{align}
The first term on the right-hand side of \eqref{eq:dualitylocalH12} can be treated in the same way as the term 
$h^{\alpha_D/(1+2\alpha_D)}\norm{e}_{L^{2}(\widehat{\Gamma_1})}$ from the right-hand side of Lemma~\ref{th:localSLP2}.

We set $m:=\lceil\frac{(1+\alpha_D)(1+2\alpha_D)}{\alpha_D} \rceil$.
The assumption $C_{\alpha_D}\frac{h}{R}\leq \frac{1}{12}$ allows us to define $m$ nested domains $\widehat{\Gamma_i}$, 
$i=0,\dots,m-1$
such that $\operatorname*{dist}(\widehat{\Gamma_i},\partial{\widehat{\Gamma_{i+1}}})\geq R$, 
$\widehat{\Gamma_{m}}\subset \widehat{\Gamma}$.
Since the term $h^{\alpha_D/(1+2\alpha_D)}\norm{e}_{L^{2}(\widehat{\Gamma_1})}$ again contains 
a local $L^2$-norm, we may use Lemma~\ref{th:localSLP2} and \eqref{eq:dualitylocalH12} again on the larger set 
$\widehat{\Gamma_2}\subsetneq \Gamma$ to estimate
\begin{eqnarray*}
 h^{\alpha_D/(1+2\alpha_D)}\norm{e}_{L^{2}(\widehat{\Gamma_1})} &\lesssim& 
 h^{\alpha_D/(1+2\alpha_D)}\Big(  \inf_{\chi_h\in S^{0,0}(\mathcal{T}_h)}\norm{\phi - \chi_h}_{L^{2}(\widehat{\Gamma_2})}  
\\
& & \qquad\qquad \qquad + h^{\alpha_D/(1+2\alpha_D)}\norm{e}_{L^{2}(\widehat{\Gamma_2})}+ \norm{e}_{H^{-1-\alpha_D}(\Gamma)}\Big).
\end{eqnarray*}
Inserting this into the initial estimate of Lemma~\ref{th:localSLP2} (using $h\leq 1$) leads to
\begin{eqnarray*}
\norm{e}_{L^{2}(\Gamma_0)} &\leq& C \Big(  \inf_{\chi_h\in S^{0,0}(\mathcal{T}_h)}
\norm{\phi - \chi_h}_{L^{2}(\widehat{\Gamma_2})}  \\
& & \qquad  + h^{2\alpha_D/(1+2\alpha_D)}\norm{\phi-\phi_h}_{L^{2}(\widehat{\Gamma_2})}+ \norm{\phi-\phi_h}_{H^{-1-\alpha_D}(\Gamma)}\Big).
\end{eqnarray*}
Now, the $L^2$-term on the right-hand side is multiplied by $h^{2\alpha_D/(1+2\alpha_D)}$, i.e.,
the square of the initial factor. Iterating this argument $m-2$-times, provides the factor
$h^{m\alpha_D/(1+2\alpha_D)}$, and by choice of $m$, we have $h^{1+\alpha_D}\leq h^{m\alpha_D/(1+2\alpha_D)}$.
Together with an inverse estimate we 
obtain
\begin{eqnarray*}
h^{1+\alpha_D}\norm{e}_{L^2(\widehat{\Gamma})} &\leq& h^{1+\alpha_D}\norm{\phi-\chi_h}_{L^2(\widehat{\Gamma})} + 
h^{1+\alpha_D}\norm{\phi_h-\chi_h}_{L^2(\widehat{\Gamma})}\\
&\lesssim& h^{1+\alpha_D}\norm{\phi-\chi_h}_{L^2(\widehat{\Gamma})} + 
\norm{\phi_h-\chi_h}_{H^{-1-\alpha_D}(\widehat{\Gamma})}\\
&\lesssim& h^{1+\alpha_D}\norm{\phi-\chi_h}_{L^2(\widehat{\Gamma})} + 
\norm{e}_{H^{-1-\alpha_D}(\widehat{\Gamma})}+\norm{\phi-\chi_h}_{H^{-1-\alpha_D}(\widehat{\Gamma})} \\
&\lesssim& \norm{\phi-\chi_h}_{L^{2}(\widehat{\Gamma})} + \norm{e}_{H^{-1-\alpha_D}(\Gamma)},
\end{eqnarray*}
which proves the theorem.
\end{proof}

\begin{proof}[of Corollary~\ref{cor:localSLP}]
The assumption $\phi \in H^{-1/2+\alpha}(\Gamma) \cap H^{\beta}(\widetilde{\Gamma})$ leads to
\begin{eqnarray*}
\inf_{\chi_h \in S^{0,0}(\mathcal{T}_h)}\norm{\phi-\chi_h}_{L^{2}(\widehat{\Gamma})}&\lesssim& h^{\beta}\norm{\phi}_{H^{\beta}(\widetilde{\Gamma})} \\
\norm{e}_{H^{-1/2}(\Gamma)}&\lesssim& h^{\alpha}\norm{\phi}_{H^{-1/2+\alpha}(\Gamma)},
\end{eqnarray*}
where the second estimate is the standard global error estimate for the BEM, see \cite{SauterSchwab}.

It remains to estimate $\norm{e}_{H^{-1-\alpha_D}(\Gamma)}$, which is treated with a duality argument:
We note that Assumption~\ref{ass:shift} and the jump relations imply the following shift theorem for $V$: 
If $w \in H^{1+\alpha_D}(\Gamma)$ and $\psi$ solves 
$V\psi = w \in H^{1+\alpha_D}(\Gamma)$, then $\psi \in H^{\alpha_D}(\Gamma)$ and $\|\psi\|_{H^{\alpha_D}(\Gamma)}
\lesssim \|w\|_{H^{1+\alpha_D}(\Gamma)}$. Hence
\begin{eqnarray*}
\norm{e}_{H^{-1-\alpha_D}(\Gamma)} &=& \sup_{w\in H^{1+\alpha_D}(\Gamma)}\frac{\skp{e,w}}{\norm{w}_{H^{1+\alpha_D}(\Gamma)}}
\lesssim
\sup_{\psi\in H^{\alpha_D}(\Gamma)}\frac{\abs{\skp{e,V\psi}}}{\norm{\psi}_{H^{\alpha_D}(\Gamma)}} =  
\sup_{\psi\in H^{\alpha_D}(\Gamma)}\frac{\abs{\skp{Ve,\psi-\Pi\psi}}}{\norm{\psi}_{H^{\alpha_D}(\Gamma)}}  \\
&\lesssim&
\sup_{\psi\in H^{\alpha_D}(\Gamma)}\frac{\norm{Ve}_{H^{1/2}(\Gamma)}
\norm{\psi-\Pi\psi}_{H^{-1/2}(\Gamma)}}{\norm{\psi}_{H^{\alpha_D}(\Gamma)}}  \lesssim h^{1/2+\alpha_D}
\norm{e}_{H^{-1/2}(\Gamma)}
\\ &\lesssim& h^{1/2+\alpha+\alpha_D} \norm{\phi}_{H^{-1/2+\alpha}(\Gamma)}.
\end{eqnarray*}
Therefore, the term of slowest convergence has an order of $\mathcal{O}(h^{\min\{1/2+\alpha+\alpha_D,\beta\}})$,
which proves the Corollary.
\end{proof}

\begin{remark}
\label{rem:regularity-limiting-convergence}
The term of slowest convergence in the case of high local regularity 
is the global error in the negative $H^{-1-\alpha_D}(\Gamma)$-norm, 
which is treated with a duality argument that uses the maximum amount of additional regularity on the 
polygonal/polyhedral domain. Therefore, further improvements of the 
convergence rate cannot be achieved with our method of proof.
In fact, the numerical examples in the 
next section confirm this observation, i.e., that the best possible convergence is $\mathcal{O}(h^{1/2+\alpha+\alpha_D})$.

The trivial estimate $\norm{\eta e}_{H^{-1/2}(\Gamma)} \lesssim \norm{\eta e}_{L^{2}(\Gamma)} $ immediately implies that
the local convergence in the energy norm is at least of order $\mathcal{O}(h^{1/2+\alpha+\alpha_D})$ as well.
Again, analyzing the proof of Lemma~\ref{th:localSLP2}, we observe that an improvement is impossible,
since the limiting term is once more in the negative $H^{-1-\alpha_D}(\Gamma)$-norm.
\eremk
\end{remark}

\begin{remark} 
Remark~\ref{rem:regularity-limiting-convergence} states that the local rate of convergence is limited by the shift theorem 
of Assumption~\ref{ass:shift}. If the geometry $\Omega$ is smooth, then elliptic shift theorems for the Dirichlet 
problem hold in a wider range, e.g., if $f\in H^{1/2}(\Omega)$, we may get
$u \in H^{5/2}(\Omega)$. It can be checked that in this setting, an estimate of the form 
$$
\|\phi - \phi_h\|_{L^2(\Gamma_0)} \lesssim \inf_{\chi_h \in S^{0,0}({\mathcal T}_h)} \|\phi - \chi\|_{L^2(\widehat \Gamma)} + \|\phi - \phi_h\|_{H^{-2}(\Gamma)} 
$$
is possible since the commutator $C_{\eta_5}^{\eta_5}$ in \eqref{eq:estimatecorrection} maps $H^{-2}(\Gamma)\rightarrow H^1(\Gamma)$
in this case. If an even better shift theorem holds, then the $H^{-2}$-norm can be further weakened by using 
commutators of higher order.
The best possible achievable local rates are then $O(h^\beta)$ in $L^2(\Gamma_0)$ for $\phi \in H^\beta(\widehat\Gamma)$, 
$\beta \in [0,1]$ and $O(h^{1/2+\beta})$ in the $H^{-1/2}(\Gamma_0)$-norm.
\eremk
\end{remark}

\subsection{The hyper-singular integral equation (proof of Theorem~\ref{th:localHypSing})}

We start with the Galerkin orthogonality 
\begin{equation}\label{eq:GalerkinorthoHS}
\skp{W(\varphi-\varphi_h),\psi_h} + \skp{\varphi -\varphi_h,1}\skp{\psi_h,1} = 0 \quad \forall \psi_h \in S^{1,1}(\mathcal{T}_h)
\end{equation}
and a Caccioppoli-type estimate on $D\subset \mathbb{R}^d$ for functions characterized by the orthogonality
\begin{equation}\label{eq:discreteharmonicHS}
\skp{\partial_n u,\psi_h} + \mu \skp{\psi_h,1} = 0 \quad \forall \psi_h \in S^{1,1}(\mathcal{T}_h), 
\operatorname*{supp}\psi_h \subset D\cap\Gamma
\end{equation}
for some $\mu \in \mathbb{R}$.
Here, we define the space of discrete harmonic functions ${\mathcal H}^{\mathcal{N}}_{h}(D,\mu)$ for an 
open set $D\subset\mathbb{R}^d$ and $\mu \in \mathbb{R}$ as
\begin{align}
{\mathcal H}^{\mathcal{N}}_{h}(D,\mu)&:=\{v\in H^1(D\backslash\Gamma) \colon \text{$v$ is harmonic on}\; D\backslash\Gamma, [\partial_n v] = 0,\nonumber\\
&\quad\quad\exists \widetilde{v} \in S^{1,1}({\mathcal T}_h) \;\mbox{s.t.} \; 
[\gamma_0 v]|_{D \cap \Gamma} = \widetilde{v}|_{D \cap \Gamma}, \; \text{$v$ satisfies}\; 
\eqref{eq:discreteharmonicHS}\}.
\end{align}
\begin{proposition}{\cite[Lemma 3.8]{FMPHypSing}}\label{prop:CaccHS}
For discrete harmonic functions $u \in {\mathcal H}^{\mathcal{N}}_{h}(B',\mu)$, we have the interior regularity estimate 
\begin{equation}\label{eq:CaccioppoliHS}
\norm{\nabla u}_{L^2(B\backslash\Gamma)} \lesssim \frac{h}{\widehat{d}}\norm{\nabla u}_{L^2(B'\backslash\Gamma)} + 
\frac{1}{\widehat{d}}\norm{u}_{L^2(B'\backslash\Gamma)} + \abs{\mu},
\end{equation}
where $B$ and $B'$ are nested boxes and
$\widehat{d} := \operatorname*{dist}(B,\partial B')$ satisfies $8h\leq \widehat{d}$. 
The hidden constant depends only on $\Omega, d$, and the $\gamma$-shape regularity of $\mathcal{T}_h$.
\end{proposition}

Again we use the Galerkin projection $\Pi : H^{1/2}(\Gamma) \rightarrow S^{1,1}(\mathcal{T}_h)$ now defined by
\begin{equation}\label{eq:GalerkinprojectionHS}
\skp{W(\varphi - \Pi\varphi),\psi_h}+ \skp{\varphi-\Pi\varphi,1}\skp{\psi_h,1} = 0 
\quad \forall \psi_h \in S^{1,1}(\mathcal{T}_h).
\end{equation}

The following lemma collects approximation properties of the Galerkin projection. These 
properties will be applied in both Lemma~\ref{lem:tracejumpest} and Lemma~\ref{th:localHypSing2} below.

\begin{lemma}\label{lem:Galerkinproj}
 Let $\Pi$ be the Galerkin projection defined in \eqref{eq:GalerkinprojectionHS} and 
 $\eta,\widehat{\eta} \in C_0^{\infty}(\mathbb{R}^d)$ be cut-off functions, where $\widehat{\eta} \equiv 1$ on
 $\operatorname*{supp} \eta$.
 For $\varphi \in H^1(\Gamma)$, we have for $s \in [1/2,1]$ 
 \begin{equation}
  \abs{\eta \varphi - \Pi(\eta \varphi)}_{H^{s}(\Gamma)} \leq C \abs{\widehat{\eta} \varphi}_{H^{s}(\Gamma)}.
 \end{equation}
 For  $\varphi_h \in S^{1,1}(\mathcal{T}_h)$,we have for $s \in [1/2,1]$  
  \begin{equation}
  \abs{\eta \varphi_h - \Pi(\eta \varphi_h)}_{H^s(\Gamma)} \leq C h\abs{\widehat{\eta} \varphi_h}_{H^s(\Gamma)}.
 \end{equation}
The constant $C>0$ depends only on $\Omega$, the $\gamma$-shape regularity of $\mathcal{T}_h$, and
$\|\eta\|_{W^{1,\infty}(\mathbb{R}^d)}$.
\end{lemma}

\begin{proof}
Let $\mathcal{J}_h$ be a quasi-interpolation operator with approximation properties 
in the $H^s$-seminorm, e.g., the Scott-Zhang-projection (\cite{ScottZhang}). Then, super-approximation 
(since $\varphi_h \in S^{1,1}(\mathcal{T}_h)$) and
an inverse inequality, see, e.g., \cite[Thm. 3.2]{GHS}, as well as C\'ea's lemma imply
\begin{eqnarray*}
\abs{\eta \varphi_h- \Pi(\eta\varphi_h)}_{H^{s}(\Gamma)}
&\lesssim&\abs{\eta \varphi_h- \mathcal{J}_h(\eta\varphi_h)}_{H^{s}(\Gamma)} +
\abs{\mathcal{J}_h(\eta \varphi_h)- \Pi(\eta\varphi_h)}_{H^{s}(\Gamma)} \\
&\lesssim& h\abs{\widehat{\eta} \varphi_h}_{H^{s}(\Gamma)}  +
h^{1/2-s}\abs{\mathcal{J}_h(\eta \varphi_h)- \Pi(\eta\varphi_h)}_{H^{1/2}(\Gamma)}\nonumber \\
 &\lesssim& h\abs{\widehat{\eta} \varphi_h}_{H^{s}(\Gamma)} + h^{1/2-s}\abs{\mathcal{J}_h(\eta \varphi_h)- \eta\varphi_h}_{H^{1/2}(\Gamma)}
 \\& &+
 h^{1/2-s}\abs{\eta \varphi_h- \Pi(\eta\varphi_h)}_{H^{1/2}(\Gamma)} \\
  &\lesssim& h\abs{\widehat{\eta} \varphi_h}_{H^{s}(\Gamma)} + h^{1/2-s}\abs{\mathcal{J}_h(\eta \varphi_h)- \eta\varphi_h}_{H^{1/2}(\Gamma)}
  \lesssim h\abs{\widehat{\eta} \varphi_h}_{H^{s}(\Gamma)}.
\end{eqnarray*}
The same argument leads to
\begin{eqnarray*}
\abs{\eta \varphi- \Pi(\eta\varphi)}_{H^{s}(\Gamma)}
&\lesssim& \abs{\widehat{\eta} \varphi}_{H^{s}(\Gamma)},
\end{eqnarray*}
and consequently to the $H^1$-stability of the Galerkin-projection.
\end{proof}

In the following, we need stability and approximation properties of the Scott-Zhang projection $\mathcal{J}_h$
in the space $H^{1+\alpha_N}(\Gamma)$ provided by the following lemma.

\begin{lemma}\label{lem:ScottZhangproj}
 Let $\mathcal{J}_h$ be the Scott-Zhang projection defined in \cite{ScottZhang}.
 Then, for $s \in [0,3/2)$ we have 
 \begin{equation}
\label{eq:lem:ScottZhangproj-10}
  \|{\mathcal{J}}_h u\|_{H^{s}(\Gamma)} \leq C_s \|u\|_{H^{s}(\Gamma)} 
\qquad \forall u \in H^s(\Gamma)
 \end{equation}
and therefore, for every $0 \leq t \leq s < 3/2$ 
  \begin{equation}
\label{eq:lem:ScottZhangproj-20}
  \|u-\mathcal{J}_h u\|_{H^{t}(\Gamma)} \leq C_{s,t} h^{s-t}\|u\|_{H^{s}(\Gamma)}.
  \end{equation}
The constants $C_s$, $C_{s,t}>0$ depend only on $\Omega$, the $\gamma$-shape regularity of $\mathcal{T}_h$,
and $s$, $t$. 
\end{lemma}
\begin{proof}
We start with the proof of (\ref{eq:lem:ScottZhangproj-10}). The stability for the case $s = 1$ is 
given in \cite{ScottZhang} and the stability for the case $s = 0$ (note that $\Gamma$ is a closed surface
without boundary) is discussed in \cite[Lemma~{7}]{hypsing}. By interpolation, (\ref{eq:lem:ScottZhangproj-10})
follows for $0 < s < 1$. The starting point for the proof of (\ref{eq:lem:ScottZhangproj-10}) for $s \in (1,3/2)$ 
is that, by Remark~\ref{rem:alternative-norms}, (\ref{item:rem:alternative-norms-iii}), we may focus on 
a single affine piece $\Gamma_i$ of $\Gamma$ and can exploit that the notion of $H^s(\Gamma_i)$ coincides
with the standard notion on intervals (in 1D) and polygons (in 2D). 
In particular, $H^s(\Gamma_i)$ can be
defined as the interpolation space between $H^1(\Gamma_i)$ and $H^2(\Gamma_i)$. 

Since ${\mathcal{J}}_h u \in C^0(\Gamma)$, 
Remark~\ref{rem:alternative-norms}, (\ref{item:rem:alternative-norms-iii}) implies for $s \in (1,3/2)$ 
$$
\|{\mathcal J}_h u\|_{H^s(\Gamma)} \sim \sum_{i=1}^N \|{\mathcal J}_h u\|_{H^s(\Gamma_i)} 
\qquad \mbox{ and } \qquad 
\|u\|_{H^s(\Gamma)} \sim \sum_{i=1}^N \|u\|_{H^s(\Gamma_i)}. 
$$
It therefore suffices to show $\|{\mathcal I}_h u\|_{H^s(\Gamma_i)} \leq C \|u\|_{H^s(\Gamma_i)}$. 

Since $H^s(\Gamma_i)$ is an interpolation space between $H^1(\Gamma_i)$ and $H^2(\Gamma_i)$, 
we can find (cf. \cite{bramble-scott78}), for every $t > 0$, a function $u_t \in H^2(\Gamma_i)$ with 
\begin{eqnarray}\label{eq:mollifierstab}
  \|u_t\|_{H^2(\Gamma_i)} &\lesssim& t^{s-2} \|u\|_{H^{s}(\Gamma_i)}, \quad    
  \|u_t\|_{H^{s}(\Gamma_i)} \lesssim \|u\|_{H^{s}(\Gamma_i)}\nonumber\\
 \|u-u_t\|_{H^1(\Gamma_i)} &\lesssim& t^{s-1}\|u\|_{H^{s}(\Gamma_i)}. 
 \end{eqnarray}
 Let $I_h$ be an approximation operator with the simultaneous approximation property 
  \begin{eqnarray}\label{eq:simapprox}
 \|u_t-I_h u_t\|_{H^{s}(\Gamma_i)}+  h^{-(s-1)}\|u_t-I_h u_t\|_{H^{1}(\Gamma_i)} &\lesssim& 
 h^{2-s} \|u_t\|_{H^2(\Gamma_i)},
 \end{eqnarray}
 see, e.g., \cite{bramble-scott78}, \cite[Thm. 14.4.2]{BrennerScott}.
 With an inverse inequality, cf. \cite[Appendix]{DFGHS}, the $H^1$-stability of the Scott-Zhang projection,
 and \eqref{eq:mollifierstab}, \eqref{eq:simapprox}, we estimate
 \begin{align*}
 & \|u-\mathcal{J}_h u\|_{H^{s}(\Gamma_i)} \\
& \lesssim \|u-u_t\|_{H^{s}(\Gamma_i)} +  
 \|u_t- I_h u_t\|_{H^{s}(\Gamma_i)} 
+  \|\mathcal{J}_h(I_h(u_t)- u_t)\|_{H^{s}(\Gamma_i)} +  
 \|\mathcal{J}_h(u- u_t)\|_{H^{s}(\Gamma_i)}
\\
 &\lesssim \|u-u_t\|_{H^{s}(\Gamma_i)} +  \|u_t- I_h u_t\|_{H^{s}(\Gamma_i)}  + 
 h^{-(s-1)}\|u_t- I_h u_t\|_{H^{1}(\Gamma_i)} + h^{-(s-1)} \|u- u_t\|_{H^{1}(\Gamma_i)} \\
  &\stackrel{\eqref{eq:simapprox}}{\lesssim} \|u-u_t\|_{H^{s}(\Gamma_i)} 
  +h^{2-s}\|u_t\|_{H^2(\Gamma_i)}
  + h^{-(s-1)} \|u - u_t\|_{H^{1}(\Gamma_i)} \\
    &\stackrel{\eqref{eq:mollifierstab}}{\lesssim} \|u\|_{H^{s}(\Gamma_i)} 
    +h^{2-s}t^{s-2}\|u\|_{H^{s}(\Gamma_i)}
  + h^{-(s-1)}t^{s-1} \|u\|_{H^{s}(\Gamma_i)}.
 \end{align*}
 Choosing $t=\mathcal{O}(h)$, we get the $H^{s}(\Gamma_i)$-stability of $\mathcal{J}_h$ and thus also the
$H^s(\Gamma)$-stability of ${\mathcal{J}}_h$. 
 

We only prove the approximation property (\ref{eq:lem:ScottZhangproj-20}) for $s  \in (1,3/2)$ as the 
case $s \in [0,1]$ is covered by standard properties of the Scott-Zhang operator. 

{\em Case $1 \leq t \leq s <3/2$:} 
we observe with the stability properties of ${\mathcal J}_h$ and the approximation properties of $I_h$
\begin{align}
\label{eq:tge1}
\|u - {\mathcal J}_h u\|_{H^t(\Gamma)} &\sim  \sum_{i=1}^N \|u - {\mathcal J}_h u\|_{H^t(\Gamma_i)} 
\lesssim h^{s-t} \sum_{i=1}^N \|u\|_{H^s(\Gamma_i)} 
\sim h^{s-t} \|u\|_{H^s(\Gamma)}. 
\end{align}
{\em Case $t=0$:} 
we observe with the stability properties of ${\mathcal J}_h$ and the approximation properties of $I_h$
\begin{align}
\label{eq:t=0}
\|u - {\mathcal J}_h u\|_{L^2(\Gamma)} &\sim  \sum_{i=1}^N \|u - {\mathcal J}_h u\|_{L^2(\Gamma_i)} 
\lesssim h^{s} \sum_{i=1}^N \|u\|_{H^s(\Gamma_i)} 
\sim h^{s} \|u\|_{H^s(\Gamma)}.  
\end{align}
{\em Case $0 < t < 1$:} The remaining cases are obtained with the aid of an interpolation 
inequality: 
\begin{align*}
\|u - {\mathcal J}_h u\|_{H^t(\Gamma)} & \lesssim 
\|u - {\mathcal J}_h u\|^{1-t}_{L^2(\Gamma)} 
\|u - {\mathcal J}_h u\|^{t}_{H^1(\Gamma)} 
\stackrel{\eqref{eq:t=0}, \eqref{eq:tge1}}{ \lesssim }
h^{s(1-t)} h^{(s-1)t} \|u\|_{H^s(\Gamma)} 
 = h^{s-t} \|u\|_{H^s(\Gamma)}, 
\end{align*}
which concludes the proof. 
  \end{proof}

\begin{lemma}\label{lem:tracejumpest}
Let Assumption~\ref{ass:shift2} hold and $B \subset B' \subset  B''$ be nested boxes 
with $\widehat{d}:=\operatorname*{dist}(B,\partial B')=\operatorname*{dist}(B',\partial B'')>0$ and $h$ be sufficiently small
so that the assumption of Proposition~\ref{prop:CaccHS} holds. Let
$u:=\widetilde{K}\zeta_h$ with 
$\zeta_h \in S^{1,1}(\mathcal{T}_h)$ and assume 
$u \in {\mathcal H}^{\mathcal{N}}_h(B'',\mu)$ for the box $B'' \subset B_{R_{\Omega}}(0)$ and some $\mu \in \mathbb{R}$. 
Let $\widehat{\Gamma}\subset B\cap\Gamma$. Then,
\begin{eqnarray}
\abs{[\gamma_0 u]}_{H^1(\widehat{\Gamma})} \leq C \left(
h^{\alpha_N}\abs{\zeta_h}_{H^1(\Gamma)} +
\norm{\zeta_h}_{H^{1/2}(\Gamma)}+ \abs{\mu}\right).
\end{eqnarray}
The constant $C>0$ depends only on $\Omega,\widehat{d}$, the $\gamma$-shape regularity of $\mathcal{T}_h$, 
$\|\eta\|_{W^{1,\infty}(\mathbb{R}^d)}$, and the constants appearing in Assumption~\ref{ass:shift2}.
\end{lemma}
\begin{proof} 
{\bf Step 1:} Splitting into near and far-field.

Let $\eta \in C^{\infty}_0(\mathbb{R}^d)$ be a 
cut-off function satisfying $\eta \equiv 1$ on $B'$ and $\operatorname*{supp} \eta \subset B''$.
We define the near-field $u_{\text{near}}$ and the far field $u_{\text{far}}$ as potentials
$u_{\text{near}}:= \widetilde{K}v_h - \overline{\widetilde{K}v_h}$, with 
$\overline{\widetilde{K}v_h}:=\frac{1}{\abs{\Omega}}\int_{\Omega}{\widetilde{K}v_h}$, 
$u_{\text{far}}:=\widetilde{K}\nu_h-\overline{\widetilde{K}\nu_h}$, where 
$v_h, \nu_h \in S^{1,1}(\mathcal{T}_h)$
are BEM solutions of 
\begin{eqnarray*}
\skp{Wv_h,\psi_h} &=& \skp{\eta W\zeta_h - \eta z, \psi_h} \quad \forall \psi_h \in S^{1,1}(\mathcal{T}_h) \\
\skp{W\nu_h,\psi_h} &=& \skp{(1-\eta) W\zeta_h + \eta z, \psi_h} \quad \forall \psi_h \in S^{1,1}(\mathcal{T}_h)
\end{eqnarray*}
with $\skp{v_h,1}=0=\skp{\nu_h,1}$.
Here, $z$ is a function with $z\equiv \mu$ on $\Gamma\cap B'$ such that the compatibility condition
$\skp{\eta W\zeta_h - \eta z,1}=\skp{(\eta-1)W\zeta_h - \eta z,1} = 0$ holds. Since 
$\skp{W\zeta_h,1}=0$ such a function exists. 
More precisely, we choose $z \in L^2(\Gamma)$ to be the piecewise constant function
$$z:= \left\{
\begin{array}{l}
 \mu \qquad\hspace{20mm}\textrm{on}\hspace{2mm} \Gamma \cap B' , \\
\frac{\skp{\eta W\zeta_h,1}-\mu\int_{\Gamma \cap B'}\eta}{\int_{(B''\backslash B')\cap\Gamma}\eta} \quad \textrm{otherwise}.
 \end{array}
 \right.$$

The function $v_h+\nu_h$ solves 
\begin{equation*}
\skp{W(v_h+\nu_h),\psi_h} = \skp{W\zeta_h, \psi_h} \quad \forall \psi_h \in S^{1,1}(\mathcal{T}_h),
\end{equation*}
which implies $v_h+\nu_h = \zeta_h + c$ for a constant $c$. Therefore, 
$v:=u_{\text{near}}+u_{\text{far}} = u + \widetilde{K}c - \overline{\widetilde{K}(v_h+\nu_h)}$. Since $[\gamma_0 \widetilde{K}c] = c$ this implies
\begin{eqnarray*}
\abs{[\gamma_0u]}_{H^1(\widehat{\Gamma})} = 
\abs{[\gamma_0v]}_{H^1(\widehat{\Gamma})} \lesssim 
\abs{[\gamma_0 u_{\text{near}}]}_{H^1(\widehat{\Gamma})}+\abs{[\gamma_0u_{\text{far}}]}_{H^1(\widehat{\Gamma})}.
\end{eqnarray*}
The definition of $z$ and $\eta \equiv 1$ on $B'$ lead to
\begin{eqnarray*}
\norm{\eta(z-\mu)}_{L^2(\Gamma)}^2   &=& 
\int_{(B''\backslash B')\cap\Gamma}\eta^2\left(\frac{\skp{\eta W\zeta_h,1}-\mu\int_{\Gamma \cap B'}\eta}{\int_{(B''\backslash B')\cap\Gamma}\eta}-\mu\right)^2
= 
\int_{(B''\backslash B')\cap\Gamma}\eta^2\left(\frac{\skp{\eta(W\zeta_h-\mu),1}}{\int_{(B''\backslash B')\cap\Gamma}\eta}\right)^2
\\ &\lesssim&
\abs{\skp{\eta(W\zeta_h-\mu),1}}^2.
\end{eqnarray*}
Consequently, we obtain
\begin{equation}\label{eq:estz} 
\norm{\eta(z-\mu)}_{L^2(\Gamma)} \lesssim \abs{\skp{\eta(W\zeta_h-\mu),1}} \lesssim 
\norm{\eta(W\zeta_h-\mu)}_{H^{-1-\alpha_N}(\Gamma)} \lesssim h^{1+\alpha_N}
\left(\norm{W\zeta_h}_{L^2(\Gamma)}+\abs{\mu}\right).
\end{equation}
The last inequality follows from the orthogonality of $W\zeta_h$ to discrete functions in $S^{1,1}(\mathcal{T}_h)$ 
on $B''$ and the arguments shown  
in \eqref{eq:estvhtmp} below (specifically: go through the arguments of \eqref{eq:estvhtmp} with $z \equiv \mu$).

{\bf Step 2:} Approximation of the near field.

 Let $\mathcal{J}_h$ denote the Scott-Zhang projection.
The ellipticity of $W$ on $H^{1/2}(\Gamma)/\mathbb{R}$ and the orthogonality \eqref{eq:discreteharmonicHS} of $W\zeta_h$ imply
\begin{eqnarray}\label{eq:estvhtmp}
 \norm{v_h}_{H^{1/2}(\Gamma)} &\lesssim& \norm{\eta W\zeta_h - \eta z}_{H^{-1/2}(\Gamma)} = 
 \sup_{w\in H^{1/2}(\Gamma)} \frac{\skp{\eta W\zeta_h - \eta z,w}}{\norm{w}_{H^{1/2}(\Gamma)}}\nonumber \\
 &=&  \sup_{w\in H^{1/2}(\Gamma)} \frac{\skp{W\zeta_h,\eta w- \mathcal{J}_h(\eta w)}-\skp{\eta z,w}+\mu\skp{\mathcal{J}_h(\eta w),1}}{\norm{w}_{H^{1/2}(\Gamma)}} \nonumber\\
 &=&  \sup_{w\in H^{1/2}(\Gamma)} \frac{\skp{W\zeta_h,\eta w- \mathcal{J}_h(\eta w)}-\skp{\eta(z-\mu),w}-\mu\skp{\eta w- \mathcal{J}_h(\eta w),1}}{\norm{w}_{H^{1/2}(\Gamma)}} \nonumber\\
 &\lesssim&  \sup_{w\in H^{1/2}(\Gamma)} \frac{\left(\norm{W\zeta_h}_{L^2(\Gamma)}+\abs{\mu}\right)\norm{\eta w- \mathcal{J}_h(\eta w)}_{L^2(\Gamma)}+
 \norm{\eta(z-\mu)}_{H^{-1/2}(\Gamma)}\norm{w}_{H^{1/2}(\Gamma)}}{\norm{w}_{H^{1/2}(\Gamma)}} \nonumber\\
&\lesssim&  h^{1/2}\left(\norm{W\zeta_h}_{L^2(\Gamma)}+\abs{\mu}\right)+
 \norm{\eta(z-\mu)}_{H^{-1/2}(\Gamma)} \nonumber\\
 &\stackrel{(\ref{eq:estz})}{\lesssim}& h^{1/2}\left(\norm{W\zeta_h}_{L^2(\Gamma)}+\abs{\mu}\right).
 \end{eqnarray}
 With the same arguments and Lemma~\ref{lem:ScottZhangproj} we may estimate 
\begin{eqnarray}
\label{eq:estWxihztmp}
\norm{\eta W\zeta_h - \eta z}_{H^{-1-\alpha_N}(\Gamma)}  
\lesssim  h^{1+\alpha_N}\left(\norm{W\zeta_h}_{L^2(\Gamma)} + \abs{\mu} \right)+
 \norm{\eta(z-\mu)}_{H^{-1-\alpha_N}(\Gamma)}.
 \end{eqnarray}
 Let $\psi$ solve $W\psi = w-\overline{w}$ for $w \in H^{\alpha_N}(\Gamma)$. Then $\psi \in H^{1+\alpha_N}(\Gamma)$.
 Together with the mapping properties of $W$ from Lemma~\ref{lem:potentialregK},
 $\skp{v_h,1} = 0$, the definition of $v_h$, and the stability and approximation properties of $\mathcal{J}_h$
 from Lemma~\ref{lem:ScottZhangproj}, we obtain
\begin{align}\label{eq:estWvhnegnorm}
&\norm{Wv_h}_{H^{-1-\alpha_N}(\Gamma)}  \lesssim \norm{v_h}_{H^{-\alpha_N}(\Gamma)} =
\sup_{w\in H^{\alpha_N}(\Gamma)}\frac{\skp{v_h,w}}{\norm{w}_{H^{\alpha_N}(\Gamma)}} = 
\sup_{w\in H^{\alpha_N}(\Gamma)}\frac{\skp{v_h,w-\overline{w}}}{\norm{w}_{H^{\alpha_N}(\Gamma)}} \nonumber\\
&\qquad \leq \sup_{\psi\in H^{1+\alpha_N}(\Gamma)}\frac{\abs{\skp{v_h,W\psi}}}{\norm{\psi}_{H^{1+\alpha_N}(\Gamma)}} 
=\sup_{\psi\in H^{1+\alpha_N}(\Gamma)}\frac{\abs{\skp{Wv_h,\psi -\mathcal{J}_h \psi} + \skp{Wv_h,\mathcal{J}_h\psi}}}{\norm{\psi}_{H^{1+\alpha_N}(\Gamma)}} \nonumber \\
&\qquad=\sup_{\psi\in H^{1+\alpha_N}(\Gamma)}\frac{\abs{\skp{Wv_h,\psi -\mathcal{J}_h \psi} + \skp{\eta(W\zeta_h-z),\mathcal{J}_h\psi}}}{\norm{\psi}_{H^{1+\alpha_N}(\Gamma)}} \nonumber \\
&\qquad\lesssim\sup_{\psi\in H^{1+\alpha_N}(\Gamma)}\frac{\norm{Wv_h}_{H^{-1/2}(\Gamma)}\norm{\psi -\mathcal{J}_h \psi}_{H^{1/2}(\Gamma)} + 
\norm{\eta(W\zeta_h-z)}_{H^{-1-\alpha_D}(\Gamma)}\norm{\mathcal{J}_h\psi}_{H^{1+\alpha_D}(\Gamma)}}{\norm{\psi}_{H^{1+\alpha_N}(\Gamma)}} \nonumber \\
&\qquad\stackrel{\eqref{eq:estvhtmp},\eqref{eq:estWxihztmp},\eqref{eq:estz}}{\lesssim} h^{1+\alpha_N}\left( \norm{W\zeta_h}_{L^2(\Gamma)} + \abs{\mu}\right).
\end{align}
With the mapping properties of $W$ from Lemma~\ref{lem:potentialregK}, 
an inverse estimate, and \eqref{eq:estvhtmp} we obtain for $0\leq\epsilon \leq \alpha_N$
\begin{eqnarray}\label{eq:estWvhposnorm}
\norm{Wv_h}_{H^{\epsilon}(\Gamma)}&\lesssim& 
\norm{v_h}_{H^{1+\epsilon}(\Gamma)}\lesssim h^{-\epsilon-1/2}\norm{v_h}_{H^{1/2}(\Gamma)}\nonumber \\
&\stackrel{\eqref{eq:estvhtmp}}{\lesssim}& h^{-\epsilon}\left(\norm{W\zeta_h}_{L^2(\Gamma)}+\abs{\mu}\right)
\nonumber \\ 
&\lesssim& h^{-\epsilon}\left(\norm{\zeta_h}_{H^1(\Gamma)}+\abs{\mu}\right).
\end{eqnarray}
We first consider $\gamma_0^{\rm int} u_{\rm near}$ - the case $\gamma_0^{\rm ext} u_{\rm near}$ is treated analogously.
By construction of $u_{\rm near}$, we have
\begin{eqnarray}\label{eq:orthonostab}
 \skp{\partial_n u_{\rm near},\psi_h}&=&  \skp{-Wv_h,\psi_h} =  \skp{-\eta W\zeta_h + \eta z,\psi_h} \nonumber \\ 
 &=& \skp{\partial_n u,\psi_h} + \mu \skp{\psi_h,1} =   0 \quad
\forall \psi_h \in S^{1,1}(\mathcal{T}_h), \operatorname*{supp}\psi_h \subset B'\cap\Gamma
\end{eqnarray}
since $z\equiv \mu$, $\eta \equiv 1$ on $\operatorname*{supp} \psi_h$. Therefore, 
$u_{\text{near}} \in {\mathcal H}^{\mathcal{N}}_h(B',0)$.

Let $\widehat{\eta}$ be another cut-off function satisfying $\widehat{\eta} \equiv 1$ on $\widehat{\Gamma}$ and
$\operatorname*{supp} \widehat{\eta} \subset B$.
The multiplicative trace inequality, see, e.g., \cite[Thm. A.2]{Melenk}, implies for any $\epsilon \leq 1/2$ that
\begin{eqnarray}
\nonumber 
 \abs{\gamma_0^{\rm int}u_{\text{near}}}_{H^1(\widehat{\Gamma})} &\lesssim&  
\norm{\nabla(\widehat{\eta} u_{\text{near}})}_{L^2(B\cap\Gamma)}
\lesssim \norm{\nabla(\widehat{\eta} u_{\text{near}})}_{L^2(\Omega)}^{2\epsilon/(1+2\epsilon)}
\norm{\nabla(\widehat{\eta} u_{\text{near}})}_{H^{1/2+\epsilon}(\Omega)}^{1/(1+2\epsilon)}\\
\label{eq:lem:normaltraceestHS-10}
&\lesssim& \norm{\nabla(\widehat\eta u_{\text{near}})}_{L^2(B)}^{2\epsilon/(1+2\epsilon)}
\norm{\widehat{\eta}u_{\text{near}}}_{H^{3/2+\epsilon}(B)}^{1/(1+2\epsilon)}.
\end{eqnarray}
Since $u_{\text{near}} \in {\mathcal H}^{\mathcal{N}}_h(B',0)$,
we may use the interior regularity estimate \eqref{eq:discreteharmonicHS} with $\mu = 0$
for the first term on the right-hand side of 
(\ref{eq:lem:normaltraceestHS-10}). 
The second factor of (\ref{eq:lem:normaltraceestHS-10}) can be estimated using (\ref{eq:lem:shiftaprioriNeumann-20}) 
of Lemma~\ref{lem:shiftaprioriNeumann}. In total,  we get for $\epsilon \leq \alpha_N<1/2$ that
\begin{align}
\nonumber 
 &\norm{\nabla(\widehat{\eta}u_{\text{near}})}_{L^2(B)}^{2\epsilon/(1+2\epsilon)}
\norm{\widehat{\eta}u_{\text{near}}}_{H^{3/2+\epsilon}(B)}^{1/(1+2\epsilon)} \\ 
\nonumber 
& \qquad \lesssim 
\left(h\norm{\nabla u_{\text{near}}}_{L^2(B')}+ \norm{u_{\text{near}}}_{L^2(B')}\right)^{2\epsilon/(1+2\epsilon)}\nonumber 
\cdot\left(\norm{u_{\text{near}}}_{H^1(B')}+\norm{ \partial_n u_{\text{near}} }_{H^{\epsilon}(\Gamma)} \right)^{1/(1+2\epsilon)} \nonumber\\
&\qquad\lesssim h^{2\epsilon/(1+2\epsilon)}\norm{u_{\text{near}}}_{H^1(B')} + 
\norm{u_{\text{near}}}_{L^2(B')}^{2\epsilon/(1+2\epsilon)}\norm{u_{\text{near}}}_{H^1(B')}^{1/(1+2\epsilon)}
 \nonumber \\
\nonumber 
&\qquad\quad  + \norm{u_{\text{near}}}_{L^2(B')}^{2\epsilon/(1+2\epsilon)}
\norm{Wv_h}_{H^{\epsilon}(\Gamma)}^{1/(1+2\epsilon)} + h^{2\epsilon/(1+2\epsilon)}\norm{\nabla u_{\text{near}}}_{L^2(B')}^{2\epsilon/(1+2\epsilon)}
\norm{W v_h}_{H^{\epsilon}(\Gamma)}^{1/(1+2\epsilon)} \\
\label{eq:nearfieldtempHS}
& \qquad =: T_1 + T_2 + T_3 + T_4. 
\end{align}
The mapping properties of $\widetilde{K}$ imply with \eqref{eq:estvhtmp} and \eqref{eq:estWvhposnorm}
\begin{align}
\label{eq:aprioriH1Neumann}
T_1  &= h^{2\varepsilon/(1+2\varepsilon)} \norm{u_{\text{near}}}_{H^1(B')}\lesssim 
h^{2\varepsilon/(1+2\varepsilon)} 
\norm{v_h}_{H^{1/2}(\Gamma)} \stackrel{\eqref{eq:estvhtmp}}{\lesssim}
h^{2\varepsilon/(1+2\varepsilon)+1/2}\left( 
\norm{\zeta_h}_{H^{1}(\Gamma)} +\abs{\mu} \right), \\
\nonumber 
T_4 &= h^{2\epsilon/(1+2\epsilon)}\norm{\nabla u_{\text{near}}}_{L^2(B')}^{2\epsilon/(1+2\epsilon)}
\norm{Wv_h}_{H^{\epsilon}(\Gamma)}^{1/(1+2\epsilon)} 
\stackrel{\eqref{eq:estWvhposnorm}}{\lesssim} h^{2\epsilon/(1+2\epsilon)}\left( 
\norm{\zeta_h}_{H^{1}(\Gamma)} +\abs{\mu} \right).
\label{eq:lem:normaltraceestHS-200}
\end{align}
We apply (\ref{eq:lem:shiftaprioriNeumann-10}) - $u_{\text{near}}$ has mean zero - and since $\widetilde{K} v_h$ is 
smooth on $\partial B_{R_{\Omega}}(0)$, we can estimate 
$\norm{\widetilde{K}v_h}_{H^{-\alpha_N}(\partial B_{R_{\Omega}}(0))} \lesssim \norm{v_h}_{H^{-\alpha_N}(\Gamma)}$.
Together with \eqref{eq:estWvhposnorm}, \eqref{eq:estWvhnegnorm}, and Young's inequality this leads to
\begin{align*}
T_3 & = 
\norm{u_{\text{near}}}_{L^2(B')}^{2\epsilon/(1+2\epsilon)}
\norm{Wv_h}_{H^{\epsilon}(\Gamma)}^{1/(1+2\epsilon)} \\
&\stackrel{(\ref{eq:lem:shiftaprioriNeumann-10}),(\ref{eq:estWvhposnorm})}{\lesssim} h^{-\epsilon/(1+2\epsilon)}
\left(\norm{Wv_h}_{H^{-1-\alpha_N}(\Gamma)}+\norm{\widetilde{K}v_h}_{H^{-\alpha_N}(\partial B_{R_{\Omega}}(0))}\right)^{2\epsilon/(1+2\epsilon)}
\left( 
\norm{\zeta_h}_{H^{1}(\Gamma)} +\abs{\mu} \right)^{1/(1+2\epsilon)} \\
&\lesssim h^{-1}\left(\norm{Wv_h}_{H^{-1-\alpha_N}(\Gamma)} + \norm{v_h}_{H^{-\alpha_N}(\Gamma)}\right) + h^{\epsilon}\left( 
\norm{\zeta_h}_{H^{1}(\Gamma)} +\abs{\mu} \right)  \stackrel{\eqref{eq:estWvhnegnorm}}{\lesssim} \left(h^{\alpha_N} + h^{\epsilon} \right)
\left( \norm{\zeta_h}_{H^{1}(\Gamma)} +\abs{\mu} \right). 
\end{align*}
Similarly, we get for the second term  in \eqref{eq:nearfieldtempHS}
\begin{align*}
T_2 &= 
\norm{u_{\text{near}}}_{L^2(B')}^{2\epsilon/(1+2\epsilon)}
\norm{u_{\text{near}}}_{H^1(B')}^{1/(1+2\epsilon)} \\&\stackrel{(\ref{eq:lem:shiftaprioriNeumann-10})}{\lesssim} 
h^{-\epsilon/(1+2\epsilon)}
\left(\norm{Wv_h}_{H^{-1-\alpha_N}(\Gamma)}+\norm{\widetilde{K}v_h}_{H^{-\alpha_N}(\partial B_{R_{\Omega}}(0))}\right)^{2\epsilon/(1+2\epsilon)} \cdot \\
&\qquad \;\; h^{(\epsilon+1/2)/(1+2\epsilon)}\left(\norm{\zeta_h}_{H^{1}(\Gamma)}+\abs{\mu}\right)^{1/(1+2\epsilon)} \\
&\lesssim h^{-1/2}\left(\norm{Wv_h}_{H^{-1-\alpha_N}(\Gamma)}+ \norm{v_h}_{H^{-\alpha_N}(\Gamma)} \right)+ 
h^{1/2+\epsilon}\left(\norm{\zeta_h}_{H^{1}(\Gamma)} +\abs{\mu} \right)
\\&\lesssim 
\left(h^{1/2+\alpha_N} + h^{1/2+\epsilon} \right)
\left( \norm{\zeta_h}_{H^{1}(\Gamma)} +\abs{\mu} \right).
\end{align*}
Inserting everything in \eqref{eq:nearfieldtempHS} and choosing $\epsilon = \alpha_N$ gives 
\begin{eqnarray*}
\abs{\gamma_0^{\rm int} u_{\text{near}}}_{H^1(\widehat{\Gamma})} &\lesssim& 
(h^{2\alpha_N/(1+2\alpha_N)+1/2} + h^{1/2+\alpha_N} +h^{\alpha_N}+h^{2\alpha_N/(1+2\alpha_N)})
\left(\norm{\zeta_h}_{H^1(\Gamma)}+\abs{\mu}\right) \\
&\lesssim& h^{\alpha_N}\left(\norm{\zeta_h}_{H^1(\Gamma)}+\abs{\mu} \right). 
\end{eqnarray*}
Applying the same argument for the exterior trace leads to an estimate for the jump
of the trace
\begin{eqnarray*}
\abs{[\gamma_0 u_{\text{near}}]}_{H^1(\widehat{\Gamma})}
\lesssim h^{\alpha_N}\left(\norm{\zeta_h}_{H^1(\Gamma)}+\abs{\mu} \right). 
\end{eqnarray*}
{\bf Step 3:} Approximation of the far field.

We define the function $\nu  \in H^{1/2}(\Gamma)$
as the solution of
\begin{equation*}
W\nu  = (1-\eta) W\zeta_h + \eta z, \qquad \skp{\nu ,1}=0.
\end{equation*}
Then, we have 
\begin{equation*}
\skp{W(\nu -\nu_h),\psi_h} = 0 \qquad \forall \psi_h \in S^{1,1}(\mathcal{T}_h).
\end{equation*}
Let $\widehat{u}_{\rm far} := \widetilde{K}\nu-\overline{\widetilde{K}\nu}$ where 
$\overline{\widetilde{K}\nu}:=\frac{1}{\abs{\Omega}}\skp{\widetilde{K}\nu,1}_{L^2(\Omega)}$
and $\widehat{\eta}$ be another cut-off function with 
$\widehat{\eta} \equiv 1$ on $\widehat{\Gamma}$ and $\operatorname*{supp} \widehat\eta \subset B$. Then, with the 
Galerkin projection $\Pi$, the triangle inequality and the jump conditions of $\widetilde{K}$ imply
\begin{eqnarray}\label{eq:farfieldcont}
 \abs{[\gamma_0 u_{\text{far}}]}_{H^1(\widehat{\Gamma})} =  
 \abs{\widehat{\eta}\nu_h}_{H^1(\widehat{\Gamma})} \leq  \abs{\widehat{\eta}\nu_h-\Pi(\widehat\eta \nu)}_{H^1(\widehat{\Gamma})}  +  
 \abs{\Pi(\widehat\eta \nu)}_{H^1(\widehat{\Gamma})}. 
 \end{eqnarray}
 The smoothness of $\widetilde{K}\nu$ on $\partial B_{R_{\Omega}}(0)$ 
 and the coercivity of $W$ on $H^{1/2}(\Gamma)/\mathbb{R}$ lead to
 \begin{equation*}
  \norm{\widetilde{K}\nu-\overline{\widetilde{K}\nu}}_{H^{1/2}(\partial B_{R_{\Omega}}(0))} \lesssim \norm{\nu}_{H^{1/2}(\Gamma)}
  \lesssim \norm{W\nu}_{H^{-1/2}(\Gamma)}. 
 \end{equation*}
 We apply Lemma~\ref{lem:shiftaprioriNeumann} with a cut-off function $\widetilde{\eta}$ satisfying
 $\widetilde{\eta} \equiv 1$ on $B$ and $\operatorname*{supp} \widetilde\eta \subset B'$. Then 
 $\eta \equiv 1$ and $z \equiv \mu$ on $B'$ imply $\widetilde{\eta}(1-\eta)\equiv 0$ and 
 $\widetilde\eta \eta z = \widetilde{\eta} \mu$.
 The $H^1$-stability of the Galerkin projection from Lemma~\ref{lem:Galerkinproj}, a facewise trace estimate,
 and similar estimates as for the near field imply
 \begin{eqnarray}\label{eq:estfarfield1}
 \abs{\Pi(\widehat\eta \nu )}_{H^1(\widehat{\Gamma})} &\lesssim&  \abs{\widehat\eta\nu }_{H^1(\Gamma)}
\lesssim \norm{\widehat{u}_{\text{far}}}_{H^{3/2+\epsilon}(B\backslash\Gamma)}\nonumber\\
&\stackrel{(\ref{eq:lem:shiftaprioriNeumann-20})}{\lesssim}& \norm{\widehat u_{\text{far}}}_{H^{1}(B'\backslash\Gamma)}+ 
\norm{\widetilde{\eta}((1-\eta)W\zeta_h + \eta z)}_{H^{\epsilon}(\Gamma)}\nonumber \\
&\lesssim&\norm{\widehat u_{\text{far}}}_{H^{1}(B'\backslash\Gamma)} + 
\abs{\mu}\norm{\widetilde{\eta}}_{H^{\epsilon}(\Gamma)}\nonumber \\ &\lesssim&
\norm{(1-\eta)W\zeta_h + \eta z}_{H^{-1/2}(\Gamma)}+ 
\norm{\widetilde{K}\nu-\overline{\widetilde{K}\nu}}_{H^{1/2}(\partial B_{R_{\Omega}}(0))}+\abs{\mu} \nonumber\\ 
&\lesssim&
\norm{(1-\eta)W\zeta_h + \eta z}_{H^{-1/2}(\Gamma)}+\abs{\mu} 
\lesssim \norm{\zeta_h}_{H^{1/2}(\Gamma)}+\norm{\eta(z -\mu)}_{H^{-1/2}(\Gamma)} + \abs{\mu} \nonumber\\ 
&\stackrel{\eqref{eq:estz}}{\lesssim}& \norm{\zeta_h}_{H^{1/2}(\Gamma)}+ \abs{\mu}.
\end{eqnarray}
It remains to estimate the first term on the right-hand side of \eqref{eq:farfieldcont}. With an inverse estimate 
and Lemma~\ref{lem:Galerkinproj} we get
\begin{eqnarray}\label{eq:estPienu}
\abs{\widehat{\eta}\nu_h-\Pi(\widehat\eta \nu )}_{H^1(\widehat{\Gamma})} &\lesssim& 
\abs{\widehat{\eta}\nu_h-\Pi(\widehat\eta \nu_h )}_{H^{1}(\widehat{\Gamma})} + 
h^{-1/2}\abs{\Pi(\widehat{\eta}\nu_h-\widehat\eta \nu )}_{H^{1/2}(\Gamma)} \nonumber\\
&\lesssim& 
h\abs{\nu_h}_{H^{1}(\Gamma)} + 
h^{-1/2}\abs{\Pi(\widehat{\eta}\nu_h-\widehat\eta \nu )}_{H^{1/2}(\Gamma)}\nonumber\\
&\lesssim& 
h^{1/2}\norm{\nu_h}_{H^{1/2}(\Gamma)} + 
h^{-1/2}\abs{\Pi(\widehat{\eta}\nu_h-\widehat\eta \nu )}_{H^{1/2}(\Gamma)}.
\end{eqnarray}
We use the abbreviation $e_{\nu} := \nu - \nu_h$.
The ellipticity of $W$ on $H^{1/2}(\Gamma)/\mathbb{R}$ and the definition of the Galerkin projection $\Pi$ imply
\begin{eqnarray*}
\norm{\Pi(\widehat\eta e_{\nu} )}_{H^{1/2}(\Gamma)}^2 &\lesssim& 
\skp{W(\Pi(\widehat\eta e_{\nu} )),\Pi(\widehat\eta e_{\nu})} + \abs{\skp{\Pi(\widehat\eta e_{\nu} ),1}}^2 \nonumber\\
&=& \skp{W(\Pi(\widehat\eta e_{\nu} )-\widehat\eta e_{\nu}),\Pi(\widehat\eta e_{\nu})} +
\skp{W(\widehat\eta e_{\nu}),\Pi(\widehat\eta e_{\nu})} + \abs{\skp{\Pi(\widehat\eta e_{\nu} ),1}}^2\nonumber \\
&=& \skp{W(\widehat\eta e_{\nu}),\Pi(\widehat\eta e_{\nu})} + 
\skp{\widehat\eta e_{\nu} ,1}\skp{\Pi(\widehat\eta e_{\nu} ),1} \nonumber \\
&\lesssim& \abs{\skp{W(\widehat\eta e_{\nu}),\Pi(\widehat\eta e_{\nu})}} + 
\norm{\widehat\eta e_{\nu}}_{H^{-1/2}(\Gamma)}\norm{\Pi(\widehat\eta e_{\nu} )}_{H^{1/2}(\Gamma)}.
\end{eqnarray*}
With the commutator $\mathcal{C}_{\widehat \eta}$ we get
\begin{eqnarray*}
\skp{W(\widehat \eta e_{\nu}),\Pi(\widehat\eta e_{\nu})} = 
\skp{\widehat\eta W(e_{\nu})+\mathcal{C}_{\widehat{\eta}}e_{\nu},\Pi(\widehat\eta e_{\nu})}. 
\end{eqnarray*}
The definition of the Galerkin projection and the super-approximation properties of the 
Scott-Zhang projection $\mathcal{J}_h$ lead to 
\begin{eqnarray*}
\skp{W(e_{\nu}),\widehat\eta \Pi(\widehat\eta e_{\nu})}
&=&\skp{W(e_{\nu}),\widehat\eta \Pi(\widehat\eta e_{\nu}) -\mathcal{J}_h(\widehat\eta \Pi(\widehat\eta e_{\nu}))}
\\ &\lesssim&
\norm{W(e_{\nu})}_{H^{-1/2}(\Gamma)} 
\norm{\widehat\eta \Pi(\widehat\eta e_{\nu}) -\mathcal{J}_h(\widehat\eta \Pi(\widehat\eta e_{\nu}))}_{H^{1/2}(\Gamma)} \\
&\lesssim& h\norm{\nu -\nu_h}_{H^{1/2}(\Gamma)}\norm{\Pi(\widehat\eta e_{\nu})}_{H^{1/2}(\Gamma)}.
\end{eqnarray*}
For the term involving $\mathcal{C}_{\widehat{\eta}}$, we get 
\begin{eqnarray*}
\abs{\skp{\mathcal{C}_{\widehat{\eta}}(e_{\nu}),\Pi(\widehat\eta e_{\nu} )}}
&\lesssim& \norm{\mathcal{C}_{\widehat{\eta}}(\nu-\nu_h )}_{H^{-\alpha_N}(\Gamma)}
\norm{\Pi(\widehat\eta e_{\nu} )}_{H^{1/2}(\Gamma)} 
\\
&\lesssim& \norm{\nu-\nu_h}_{H^{-\alpha_N}(\Gamma)}
\norm{\Pi(\widehat\eta e_{\nu})}_{H^{1/2}(\Gamma)}.
\end{eqnarray*}
A duality argument implies $\norm{e_{\nu}}_{H^{-\alpha_N}(\Gamma)}\lesssim h^{1/2+\alpha_N}\norm{\nu}_{H^{1/2}(\Gamma)}$,
for details we refer to the proof of Corollary~\ref{cor:localHS}.
Inserting everything in \eqref{eq:estPienu}  leads to 
\begin{eqnarray*}
 \abs{\widehat{\eta}\nu_h-\Pi(\widehat\eta \nu )}_{H^1(\widehat{\Gamma})} &\lesssim& 
h^{1/2}\norm{\nu_h}_{H^{1/2}(\Gamma)} + h^{1/2}\norm{\nu-\nu_h}_{H^{1/2}(\Gamma)} + h^{\alpha_N}\norm{\nu}_{H^{1/2}(\Gamma)} 
\\ &\lesssim& h^{\alpha_N} \norm{(1-\eta)W\zeta_h + \eta z}_{H^{-1/2}(\Gamma)} \lesssim 
h^{\alpha_N}\left(\norm{\zeta_h}_{H^{1/2}(\Gamma)} + \abs{\mu} \right).
\end{eqnarray*}
Finally, this implies with \eqref{eq:farfieldcont} and \eqref{eq:estfarfield1} that
\begin{equation*}
  \abs{[\gamma_0 u_{\text{far}}]}_{H^1(\widehat{\Gamma})} \lesssim 
  (1+h^{\alpha_N})\left(\norm{\zeta_h}_{H^{1/2}(\Gamma)} + \abs{\mu}\right),
\end{equation*}
which proves the lemma.
\end{proof}

\begin{lemma}\label{th:localHypSing2}
Let $\varphi,\varphi_h$ be solutions of \eqref{eq:BIEHS}, \eqref{eq:BIEdiscreteHS} and let
$\Gamma_0, \widehat{\Gamma}$ be subsets of $\Gamma$ with $\Gamma_0\subset \widehat{\Gamma} \subsetneq \Gamma$
and $R:=\operatorname*{dist}(\Gamma_0,\partial\widehat{\Gamma}) > 0$. Let $h$ be sufficiently small such that at least
$\frac{h}{R}\leq \frac{1}{12}$ and $\eta \in C_0^{\infty}(\mathbb{R}^d)$ be an arbitrary cut-off function with 
$\eta \equiv 1$ on $\Gamma_0$, $\operatorname*{supp} \eta \cap \Gamma \subset \widehat{\Gamma}$. 
Then, we have 
\begin{eqnarray*}
\norm{\varphi-\varphi_h}_{H^{1}(\Gamma_0)} &\leq& C \Big(  \inf_{\chi_h\in S^{1,1}(\mathcal{T}_h)}
\norm{\varphi - \chi_h}_{H^{1}(\widehat{\Gamma})}  +
h^{\alpha_N}\abs{\varphi-\varphi_h}_{H^{1}(\widehat{\Gamma})} + \\
& &\quad + \norm{\eta(\varphi-\varphi_h)}_{H^{1/2}(\Gamma)}+ \norm{\varphi-\varphi_h}_{H^{-\alpha_N}(\Gamma)}\Big) 
\end{eqnarray*} 
with a constant $C>0$ depending only on $\Gamma,\Gamma_0,\widehat{\Gamma},d,R$, and the $\gamma$-shape regularity of $\mathcal{T}_h$.
\end{lemma}

\begin{proof}
We define $e:= \varphi-\varphi_h$, subsets 
$\Gamma_0\subset\Gamma_1\subset \Gamma_2 \subset \Gamma_3\subset \Gamma_4 \subset \widehat{\Gamma}$,
and volume boxes $B_0 \subset B_1 \subset B_2 \subset B_3 \subset B_4\subset \mathbb{R}^d$, where $B_i\cap\Gamma = \Gamma_i$. 
Throughout the proof, we use multiple cut-off functions $\eta_i \in C_0^{\infty}(\mathbb{R}^d)$, $i=1,\dots,4$. 
These smooth functions $\eta_i$ should satisfy
$\eta_i \equiv 1$ on $\Gamma_{i-1}$, $\operatorname*{supp}\eta \subset B_i$ and 
$\norm{\nabla \eta_i}_{L^{\infty}(B_i)}\lesssim \frac{1}{R}$.

We want to use Lemma~\ref{lem:tracejumpest}. 
Since $[\gamma_0\widetilde{K}\zeta_h] = \zeta_h \in S^{1,1}(\mathcal{T}_h)$ for any discrete function $\zeta_h \in S^{1,1}(\mathcal{T}_h)$, 
we need to construct a discrete function satisfying the orthogonality \eqref{eq:discreteharmonicHS}.
Using the Galerkin orthogonality with test functions with support $\operatorname*{supp}\psi_h\subset\Gamma_2$ and noting that
$\eta_3 \equiv 1$ on $\operatorname*{supp} \psi_h$, we obtain with the commutator $\mathcal{C}_{\eta_3}$ defined in \eqref{eq:commutatorHS},
the abbreviation $\overline{\eta_3\mathcal{C}_{\eta_3}e}=\frac{1}{\abs{\Gamma}}\skp{\eta_3 \mathcal{C}_{\eta_3}e,1}$,
and the Galerkin projection $\Pi$ from \eqref{eq:GalerkinprojectionHS}
\begin{eqnarray}\label{eq:orthohypsingloc}
0&=&\skp{We,\eta_3\psi_h} + \skp{e,1}\skp{\psi_h,1} = 
\skp{\eta_3We,\psi_h} + \skp{e,1}\skp{\psi_h,1} \nonumber \\ &=& 
\skp{W(\eta_3e)-\mathcal{C}_{\eta_3}e,\psi_h} + \skp{e,1}\skp{\psi_h,1} \nonumber \\ &=&
\skp{W(\eta_3e)-(\eta_3\mathcal{C}_{\eta_3}e-\overline{\eta_3\mathcal{C}_{\eta_3}e}),\psi_h}-
\skp{\overline{\eta_3\mathcal{C}_{\eta_3}e},\psi_h}+\skp{e,1}\skp{\psi_h,1}  \nonumber  \\ &=&
\skp{W(\eta_3e-W^{-1}(\eta_3\mathcal{C}_{\eta_3}e-\overline{\eta_3\mathcal{C}_{\eta_3}e})),\psi_h}- 
\frac{1}{\abs{\Gamma}}\skp{\eta_3 \mathcal{C}_{\eta_3}e,1}\skp{\psi_h,1} +\skp{e,1}\skp{\psi_h,1}  \nonumber  \\ &=&
\skp{W(\Pi(\eta_3e)-\Pi(W^{-1}(\eta_3\mathcal{C}_{\eta_3}e-\overline{\eta_3 \mathcal{C}_{\eta_3}e}))),\psi_h} - 
\frac{1}{\abs{\Gamma}}\skp{\eta_3 \mathcal{C}_{\eta_3}e,1}\skp{\psi_h,1} + \skp{e,1}\skp{\psi_h,1} \nonumber \\ & &
-\skp{\eta_3e-\Pi(\eta_3e),1}\skp{\psi_h,1}. 
\end{eqnarray}
Here and below, we understand the inverse $W^{-1}$ as the inverse of the bijective operator 
$W:H^{1/2}_{*}(\Gamma):=\{v\in H^{1/2}(\Gamma):\skp{v,1} = 0\}\rightarrow H_{*}^{-1/2}(\Gamma):=\{v\in H^{-1/2}(\Gamma):\skp{v,1} = 0\}$.
Since $W^{-1}$ mapps into $H_{*}^{1/2}(\Gamma)$ no additional terms in the orthogonality 
\eqref{eq:orthohypsingloc} appear.
Thus, defining 
\begin{equation*}
\zeta_h := \Pi(\eta_3e) - \xi_h \quad \text{with} \; \xi_h:=\Pi(W^{-1}(\eta_3\mathcal{C}_{\eta_3}e-\overline{\eta_3\mathcal{C}_{\eta_3}e})),
\end{equation*}
we get on a volume box $B_2\subset \mathbb{R}^d$ a discrete harmonic function 
$$u:=\widetilde{K}\zeta_h \in {\mathcal H}^{\mathcal{N}}_{h}(B_2,\mu),$$ 
where 
$\mu = \skp{e,1} -\frac{1}{\abs{\Gamma}}\skp{\eta_3 \mathcal{C}_{\eta_3}e,1}-\skp{\eta_3e -\Pi(\eta_3e),1}
$.

With the Galerkin projection $\Pi$ from \eqref{eq:GalerkinprojectionHS} 
and $\eta_3 \equiv 1$ on $\operatorname*{supp} \eta_1$, 
we write
\begin{eqnarray}\label{eq:tmpstartHS}
\norm{e}_{H^{1}(\Gamma_0)}\lesssim \norm{\eta_1 e}_{H^{1}(\Gamma)} \lesssim  
\norm{\eta_1(\eta_3 e - \Pi(\eta_3 e))}_{H^{1}(\Gamma)} +\norm{\eta_1\zeta_h}_{H^{1}(\Gamma)} + 
\norm{\eta_1 \xi_h}_{H^{1}(\Gamma)}.  
\end{eqnarray}
Lemma~\ref{lem:Galerkinproj} leads to
\begin{eqnarray}\label{eq:stabGalHS}
 \norm{\eta_3 e - \Pi(\eta_3 e)}_{H^{1}(\Gamma)} \lesssim h \norm{\eta_4\varphi_h}_{H^{1}(\Gamma)} + 
 \norm{\eta_4\varphi}_{H^{1}(\Gamma)} \lesssim 
 h \norm{\eta_4 e}_{H^{1}(\Gamma)} + 
 (h+1)\norm{\eta_4\varphi}_{H^{1}(\Gamma)}.
\end{eqnarray}
Using the $H^1$-stability of the Galerkin projection $\Pi$, the mapping properties of $W^{-1}$ and 
$\mathcal{C}_{\eta_3}$ as well as Lemma~\ref{lem:commutatorHypSing}, the correction $\xi_h$ can be estimated by 
\begin{eqnarray}\label{eq:estimatecorrectionHS}
\norm{\Pi(W^{-1}(\eta_3\mathcal{C}_{\eta_3}e-\overline{\eta_3\mathcal{C}_{\eta_3}e}))}_{H^{1}(\Gamma)} &\lesssim& 
\norm{W^{-1}(\eta_3\mathcal{C}_{\eta_3}e-\overline{\eta_3\mathcal{C}_{\eta_3}e})}_{H^{1}(\Gamma)} \nonumber\\
&\lesssim& \norm{\eta_3 \mathcal{C}_{\eta_3}e-\overline{\eta_3\mathcal{C}_{\eta_3}e}}_{L^{2}(\Gamma)} \lesssim
\norm{\eta_3 \mathcal{C}_{\eta_3}e}_{L^{2}(\Gamma)} \nonumber \\
&\lesssim&\norm{\mathcal{C}_{\eta_3}(\eta_3e)}_{L^2(\Gamma)} + \norm{\mathcal{C}_{\eta_3}^{\eta_3}e}_{L^2(\Gamma)} \nonumber\\
&\lesssim& \norm{\eta_3e}_{L^2(\Gamma)} +  \norm{e}_{H^{-\alpha_N}(\Gamma)}.
\end{eqnarray}
For the second term on the right-hand side of \eqref{eq:tmpstartHS} we have 
$\norm{\eta_1\zeta_h}_{H^{1}(\Gamma)}\lesssim\norm{\eta_1\nabla\zeta_h}_{L^{2}(\Gamma)} + \norm{\zeta_h}_{L^{2}(\Gamma)}$.
We apply 
Lemma~\ref{lem:tracejumpest} to $u = \widetilde{K}\zeta_h \in \mathcal{H}_h^{\mathcal{N}}(B_2,\mu)$ and obtain
\begin{eqnarray}\label{eq:tmpzetahHS1}
\norm{\eta_1\nabla\zeta_h}_{L^{2}(\Gamma)} &\lesssim& \abs{\zeta_h}_{H^{1}(\Gamma_1)} = 
\abs{[\gamma_0 u]}_{H^{1}(\Gamma_1)} \nonumber\\ 
&\lesssim& h^{\alpha_N}\abs{\zeta_h}_{H^1(\Gamma)} + \norm{\zeta_h}_{H^{1/2}(\Gamma)} + \abs{\mu}.
\end{eqnarray}
The $H^1$-stability of the Galerkin-projection from Lemma~\ref{lem:Galerkinproj}
and \eqref{eq:estimatecorrectionHS} lead to
\begin{equation}\label{eq:tmpzetahHS2}
\norm{\zeta_h}_{H^{1}(\Gamma)} \lesssim \norm{\eta_3e}_{H^1(\Gamma)} + 
\norm{e}_{H^{-\alpha_N}(\Gamma)}
\end{equation}
as well as 
\begin{equation}\label{eq:estzetahenergy}
\norm{\zeta_h}_{H^{1/2}(\Gamma)} \lesssim  \norm{\eta_3e}_{H^{1/2}(\Gamma)} + \norm{e}_{H^{-\alpha_N}(\Gamma)}.
\end{equation}
With the estimate $\abs{\skp{e,1}} \lesssim \norm{e}_{H^{-\alpha_N}(\Gamma)}$ and previous arguments
(using \eqref{eq:estimatecorrectionHS}, Lemma~\ref{lem:Galerkinproj}, and Lemma~\ref{lem:commutatorHypSing}), 
we may estimate
\begin{equation}\label{eq:tmpmu}
\abs{\mu} \lesssim \norm{e}_{H^{-\alpha_N}(\Gamma)} +\norm{\eta_3 e}_{H^{1/2}(\Gamma)}+ \norm{\eta_3 e}_{L^{2}(\Gamma)}.
\end{equation}
Inserting \eqref{eq:tmpzetahHS2}--\eqref{eq:tmpmu} in \eqref{eq:tmpzetahHS1},
we arrive at
\begin{eqnarray}\label{eq:tmpzetahHS4}
\norm{\eta_1\zeta_h}_{H^{1}(\Gamma)} \lesssim \norm{\eta_1\nabla\zeta_h}_{L^{2}(\Gamma)} + \norm{\zeta_h}_{L^{2}(\Gamma)} &\lesssim& h^{\alpha_N}
\left(\norm{\eta_3 e}_{H^{1}(\Gamma)}+
\norm{e}_{H^{-\alpha_N}(\Gamma)}\right)\nonumber \\ & &+
\norm{\eta_3 e}_{H^{1/2}(\Gamma)}+ \norm{e}_{H^{-\alpha_N}(\Gamma)}\nonumber \\
&\lesssim& h^{\alpha_N}\abs{e}_{H^{1}(\widehat{\Gamma})}+
\norm{\eta_4 e}_{H^{1/2}(\Gamma)}+ \norm{e}_{H^{-\alpha_N}(\Gamma)}.
\end{eqnarray}
Combining \eqref{eq:stabGalHS}, \eqref{eq:estimatecorrectionHS}, 
and \eqref{eq:tmpzetahHS4} in \eqref{eq:tmpstartHS}, we finally obtain
\begin{eqnarray*}
\norm{e}_{H^{1}(\Gamma_0)} &\lesssim&  h \norm{\eta_4 e}_{H^{1}(\Gamma)} + 
 \norm{\eta_4\varphi}_{H^{1}(\Gamma)}+
h^{\alpha_N}\abs{e}_{H^{1}(\widehat{\Gamma})}+\norm{\eta_4 e}_{H^{1/2}(\Gamma)}+ 
\norm{e}_{H^{-\alpha_N}(\Gamma)} \\
&\lesssim& \norm{\varphi}_{H^{1}(\widehat{\Gamma})}+
h^{\alpha_N}\abs{e}_{H^{1}(\widehat{\Gamma})}+\norm{\eta_4 e}_{H^{1/2}(\Gamma)}+ 
\norm{e}_{H^{-\alpha_N}(\Gamma)}.
\end{eqnarray*}
Since we only used the Galerkin orthogonality as a property of the error $e$, we may write
$\varphi-\varphi_h = (\varphi-\chi_h)+(\chi_h - \varphi_h)$ for arbitrary $\chi_h \in S^{1,1}(\mathcal{T}_h)$ with
$\operatorname*{supp} \chi_h \subset \widehat{\Gamma}$ and we have proven the claimed inequality.
\end{proof}

\begin{proof}[of Theorem~\ref{th:localHypSing}]
Starting from Lemma~\ref{th:localHypSing2}, it remains to estimate the terms
and $h^{\alpha_N}\abs{\varphi-\varphi_h}_{H^{1}(\widehat{\Gamma})}$ and 
$\norm{\eta(\varphi-\varphi_h)}_{H^{1/2}(\widehat{\Gamma})}$.

The terms are treated as in the proof of Theorem~\ref{th:localSLP}. Rather than 
using the operator $I_h\circ J_{ch}$ we may use the Scott-Zhang projection.
\end{proof}

\begin{proof}[of Corollary~\ref{cor:localHS}]
The assumption $\varphi \in H^{1/2+\alpha}(\Gamma) \cap H^{1+\beta}(\widetilde{\Gamma})$ leads to
\begin{eqnarray*}
\inf_{\chi_h \in S^{1,1}(\mathcal{T}_h)}\norm{\varphi-\chi_h}_{H^{1}(\widehat{\Gamma})}&\lesssim& h^{\beta}\norm{\varphi}_{H^{1+\beta}(\widetilde{\Gamma})} \\
\norm{e}_{H^{1/2}(\Gamma)}&\lesssim& h^{\alpha}\norm{\varphi}_{H^{1/2+\alpha}(\Gamma)},
\end{eqnarray*}
where the second estimate is the standard global error estimate for the Galerkin BEM applied to the hyper-singular
integral equation, see \cite{SauterSchwab}.

For the remaining term, we use a duality argument. 
Let $\psi$ solve $W\psi = w- \overline{w} \in H^{\alpha_N}(\Gamma)$, $\skp{\psi,1} = 0$, where $\overline{w} = \frac{1}{\abs{\Gamma}}\skp{w,1}$.
Then $\psi \in H^{1+\alpha_N}(\Gamma)$, and since $\skp{e,1} = 0$, we get with the 
Scott-Zhang projection $\mathcal{J}_h$ and Lemma~\ref{lem:ScottZhangproj}
\begin{eqnarray*}
\norm{e}_{H^{-\alpha_N}(\Gamma)} &=& \sup_{w\in H^{\alpha_N}(\Gamma)}\frac{\skp{e,w}}{\norm{w}_{H^{\alpha_N}(\Gamma)}}
= \sup_{w\in H^{\alpha_N}(\Gamma)}\frac{\skp{e,w-\overline{w}}}{\norm{w}_{H^{\alpha_N}(\Gamma)}}
\lesssim
\sup_{\psi\in H^{1+\alpha_N}(\Gamma)}\frac{\abs{\skp{e,W\psi}}}{\norm{\psi}_{H^{1+\alpha_N}(\Gamma)}} \\ 
&=&  
\sup_{\psi\in H^{1+\alpha_N}(\Gamma)}\frac{\abs{\skp{We,\psi-\mathcal{J}_h\psi}}}{\norm{\psi}_{H^{1+\alpha_N}(\Gamma)}}  \\
&\lesssim&
\sup_{\psi\in H^{1+\alpha_N}(\Gamma)}\frac{\norm{We}_{H^{-1/2}(\Gamma)}
\norm{\psi-\mathcal{J}_h\psi}_{H^{1/2}(\Gamma)}}{\norm{\psi}_{H^{1+\alpha_N}(\Gamma)}}  \lesssim h^{1/2+\alpha_N}
\norm{e}_{H^{1/2}(\Gamma)}
\\ &\lesssim& h^{1/2+\alpha+\alpha_N} \norm{\varphi}_{H^{1/2+\alpha}(\Gamma)}.
\end{eqnarray*}
Therefore, the term of slowest convergence has an order of $\mathcal{O}(h^{\min\{1/2+\alpha+\alpha_N,\beta\}})$,
which proves the Corollary.
\end{proof}

\section{Numerical Examples}\label{sec:numerics}
In this section we provide some numerical examples to underline the theoretical results of 
Section~\ref{sec:main-results}.

We only consider Symm's integral equation on quasi-uniform meshes.
Provided the right-hand side and the geometry are smooth enough, it is well-known, that the lowest order boundary 
element method in two dimensions converges in the energy norm
with the rate $N^{-3/2}$, where $N$ denotes the degrees of freedom.
In our examples we will consider problems, where the rate of convergence with uniform refinement is reduced 
due to singularities. 

In order to compute the error between the exact solution and the Galerkin approximation, we prescribe 
the solution $u(r,\theta) = r^{\alpha}\cos(\alpha\theta)$ of Poisson's equation in polar coordinates.
Then, the normal derivative $\phi = \partial_n u$ of $u$ is the solution of 
$$V\phi = (K+1/2)\gamma_0 u.$$
The regularity of $\phi$ is determined by the choice of $\alpha$. In fact, we have 
 $\phi \in H^{-1/2+\alpha-\epsilon}(\Gamma)$, $\epsilon>0$, and locally
$\phi \in H^{1}(\widetilde{\Gamma})$ for all subsets
$\widetilde{\Gamma} \subset \Gamma$ that are a positive distance away from the singularity at the origin. 

The lowest order Galerkin approximation to $\phi$ is computed using the MATLAB-library HILBERT 
(\cite{HILBERT}), where the errors in the $L^2$-norm 
are computed using two point Gau\ss-quadrature. 
The error in the local $H^{-1/2}$-norm
is computed with $\norm{\chi e}_{H^{-1/2}(\Gamma)}^2\sim \skp{V(\chi e),\chi e}$, where $\chi$
is the characteristic function for a union of elements $\Gamma_0\subset\Gamma$.

\subsection{Example 1: L-shaped domain}
We start with examples in two dimensions on a rotated L-shaped domain visualized in Figure~\ref{fig:Lshape}. 

\begin{figure}[h]
\centering
\includegraphics[width=0.40\textwidth]{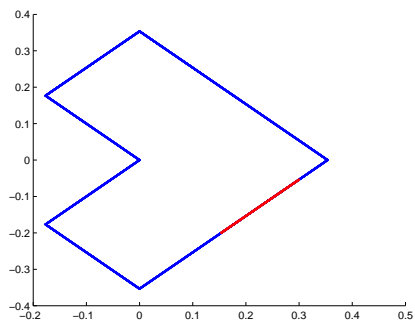}
\caption{L-shaped domain, local error computed on red part.}
\label{fig:Lshape}
\end{figure}

On the L-shaped domain, the dual problem permits solutions of regularity $H^{1/6-\epsilon}(\Gamma)$ 
for arbitrary $\epsilon >0$, so we have $\alpha_D = \frac{1}{6}-\epsilon$.

\begin{figure}[h]
\begin{minipage}{.50\linewidth}
\centering
\includegraphics{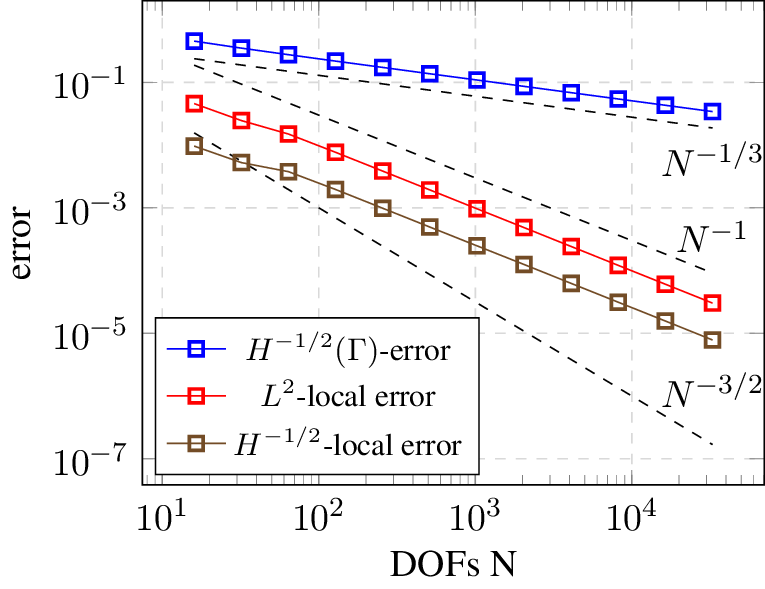}
\end{minipage}
\begin{minipage}{.50\linewidth}
\includegraphics{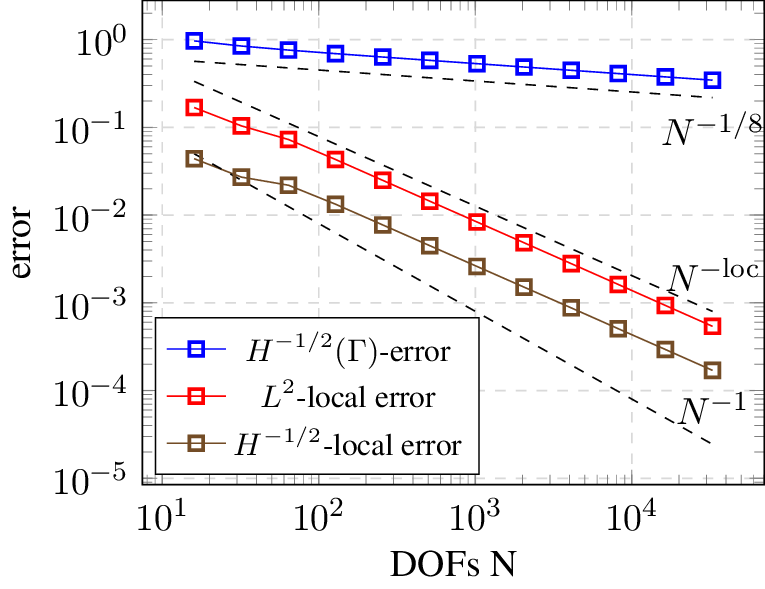}
\end{minipage}
\caption{Local and global convergence of Galerkin-BEM for Symm's equation, L-shaped domain, $\alpha = \frac{1}{3}$ (left), 
$\alpha = \frac{1}{8}$, ${\rm loc} = \frac{19}{24}$ (right).}
 \label{fig:errorLshapeSymm}
\end{figure}

Figure~\ref{fig:errorLshapeSymm}
shows the global convergence rate in the energy norm (blue) as well as the local convergence rates 
on the red part of the boundary ($\Gamma_0$, union of elements) 
in the $L^2$-norm (red) as well as the $H^{-1/2}$-norm (brown). 
The black dotted lines mark the reference curves of order $N^{-\beta}$ for various $\beta > 0$. 

In the left plot of Figure~\ref{fig:errorLshapeSymm} we chose $\alpha = \frac{1}{3}$, which leads to
$\alpha + \alpha_D = \frac{1}{2}-\epsilon$ and, indeed, we observe convergence 
in the local $L^2$-norm of almost order 1, which coincides  with the
theoretical rate obtained in Corollary~\ref{cor:localSLP}. The error
in the local $H^{-1/2}$-norm is smaller than the error
in the $L^2$-norm, but does converge with the same rate, i.e., an improvement of Theorem~\ref{th:localSLP} 
in the energy norm is not possible.
The right plot in Figure~\ref{fig:errorLshapeSymm} shows the same quantities for the choice $\alpha=\frac{1}{8}$. 
Obviously, in this case the rates of convergence are lower, and the local $L^2$-error does not converge 
with the best possible rate of one, but rather with the expected rate of $N^{-19/24}=N^{-1/2-\alpha-\alpha_D}$, 
as predicted by Corollary~\ref{cor:localSLP}.

\subsection{Example 2: Z-shaped domain}
For our second example, we change the geometry to a rotated Z-shaped domain visualized in Figure~\ref{fig:Zshape}.
Here, the dual problem permits solutions of regularity $H^{\alpha_D}(\Gamma)$ with 
$\alpha_D = \frac{1}{14}-\epsilon$.

\begin{figure}[h]
\centering
\includegraphics[width=0.40\textwidth]{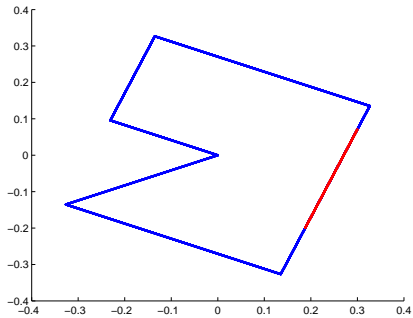}
\vspace{-4mm}
\caption{Z-shaped domain, local error computed on red part.}
\label{fig:Zshape}
\end{figure}

We again observe the expected rate of $N^{-\alpha}$ for the global 
error in the energy norm in Figure~\ref{fig:errorZshape}. However, in contrast to the previous example 
on the L-shaped domain, we do not obtain a rate of one for the local error in the $L^2$-norm 
for the case $\alpha=\frac{1}{3}$, but rather a rate of $N^{-19/21}$, 
since $\frac{1}{2}+\alpha_D + \alpha = \frac{19}{21}-\epsilon$. 
For the choice $\alpha =\frac{1}{8}$, we observe a rate of $\mathcal{O}(N^{-1/2-1/14-1/8}) = \mathcal{O}(N^{-39/56})$,
which once more matches the theoretical rate of $N^{-1/2-\alpha-\alpha_D}$. \newline

\begin{figure}[ht]
\begin{minipage}{.50\linewidth}
\includegraphics{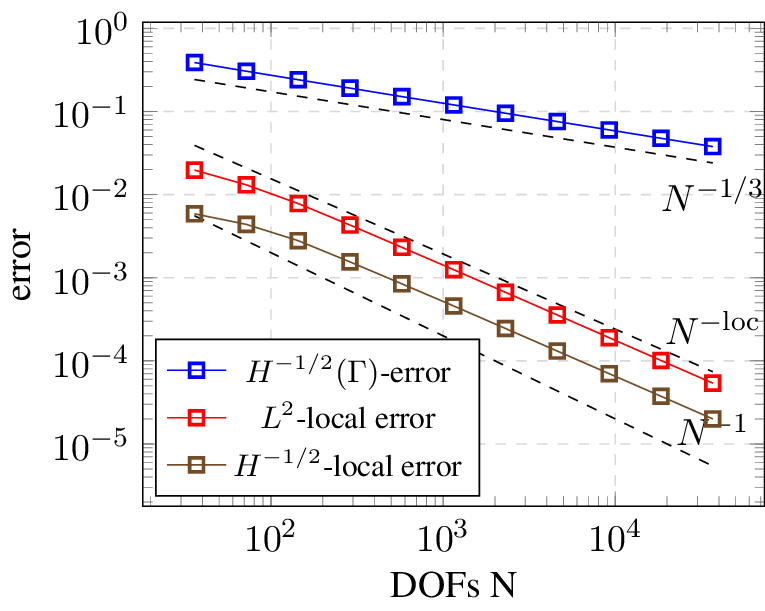}
\end{minipage}
\begin{minipage}{.49\linewidth}
\includegraphics{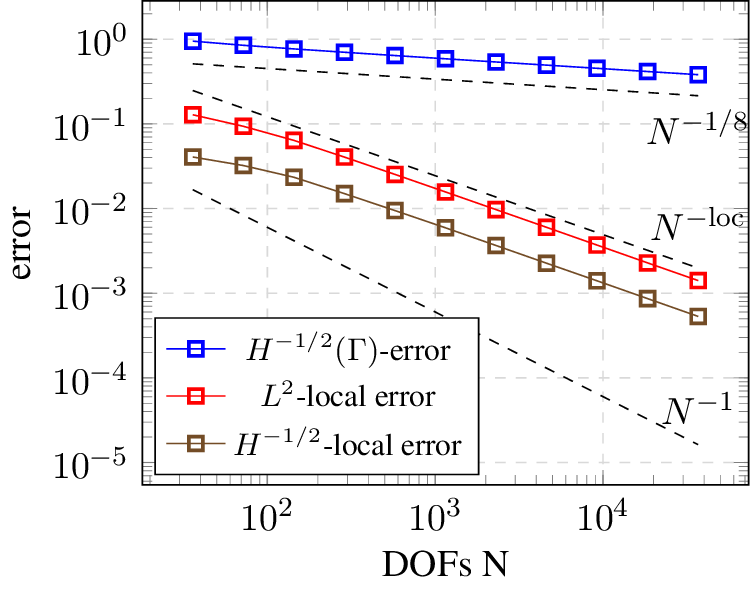}
\end{minipage}
\centering
\caption{Local and global convergence of Galerkin-BEM for Symm's equation, Z-shaped domain, $\alpha = \frac{1}{3}$, ${\rm loc} = \frac{19}{21}$ (left), 
$\alpha = \frac{1}{8}$, ${\rm loc} = \frac{39}{56}$ (right).}
 \label{fig:errorZshape}
\end{figure}

\bibliography{bibliography_1}{}
\bibliographystyle{amsalpha}

\end{document}